\documentclass[a4paper,10pt,reqno]{amsart}
\usepackage[top=2.5cm, bottom=2.5cm, left=2.5cm, right=2.5cm]{geometry}
\usepackage{amssymb,amsmath,amsthm,ascmac,array,bm,multirow}
\usepackage[dvipdfmx]{graphicx}
\usepackage{color}

\allowdisplaybreaks
\numberwithin{equation}{section}

\newtheorem{thm}{Theorem}[section]

\newtheorem{lem}[thm]{Lemma}

\newtheorem{rem}[thm]{Remark}
\newtheorem{ass}[thm]{Assumption}

\newcommand{\mcc}{\mathcal{C}}
\newcommand{\mcf}{\mathcal{F}}

\newcommand{\mcl}{\mathcal{L}}
\newcommand{\mcm}{\mathcal{M}}

\newcommand{\mfp}{\mathfrak{p}}
\newcommand{\mfL}{\mathfrak{L}}
\newcommand{\mbbe}{\mathbb{E}}
\newcommand{\mbbh}{\mathbb{H}}
\newcommand{\mbbn}{\mathbb{N}}
\newcommand{\mbbr}{\mathbb{R}}
\newcommand{\mbbrp}{\mathbb{R}_{+}}
\newcommand{\mbbu}{\mathbb{U}}
\newcommand{\mbby}{\mathbb{Y}}
\newcommand{\mbbz}{\mathbb{Z}}

\newcommand{\al}{\alpha}
\newcommand{\del}{\delta}

\newcommand{\sig}{\sigma}
\newcommand{\ep}{\epsilon}
\newcommand{\D}{\Delta}
\newcommand{\Sig}{\Sigma}
\newcommand{\lam}{\lambda}

\newcommand{\gam}{\gamma}
\newcommand{\Gam}{\Gamma}
\newcommand{\p}{\partial}

\newcommand{\cil}{\xrightarrow{\mcl}} 
\newcommand{\cip}{\xrightarrow{\pr}} 

\newcommand{\argmin}{\mathop{\rm argmin}} 
\newcommand{\argmax}{\mathop{\rm argmax}}
\newcommand{\diag}{\mathop{\rm diag}} 
\newcommand{\tr}{\mathop{\rm tr}} 

\def\ds#1{\displaystyle{#1}} 

\def\nn{\nonumber}

\def\wp{Wiener process}

\def\sumj{\sum_{j=1}^{n}}

\def\pr{\mathbb{P}}
\def\E{\mathbb{E}}
\def\var{\mathrm{var}}

\def\dim{\mathrm{dim}}

\def\tz{\theta_{0}}

\def\tes{\hat{\theta}_{n}}
\def\tet{\tilde{\theta}_{n}}
\def\aet{\tilde{\alpha}_{n}}
\def\bet{\tilde{\beta}_{n}}

\def\aes{\hat{\alpha}_{n}}

\def\bes{\hat{\beta}_{n}}

\def\mqbic{\mathrm{mQBIC}_{n}}
\def\mbic{\mathrm{mBIC}_{n}}

\title[Data driven time scale in Gaussian quasi-likelihood inference]
{Data driven time scale in Gaussian quasi-likelihood inference}

\author[S. Eguchi]{Shoichi Eguchi}
\address{
Center for Mathematical Modeling and Data Science, Osaka University, 
1-3 Machikaneyama-cho, Toyonaka City, Osaka 560-8531, Japan.}
\email{eguchi@sigmath.es.osaka-u.ac.jp}

\author[H. Masuda]{Hiroki Masuda}
\address{Faculty of Mathematics, Kyushu University, 744 Motooka Nishi-ku Fukuoka 819-0395, Japan}
\email{hiroki@math.kyushu-u.ac.jp}

\date{\today}

\begin{document}
\setlength{\baselineskip}{4.5mm}
\setlength{\marginparwidth}{2cm}

\maketitle

\begin{abstract}
We study parametric estimation of ergodic diffusions observed at high frequency.
Different from the previous studies, we suppose that sampling stepsize is unknown, thereby making the conventional Gaussian quasi-likelihood not directly applicable.
In this situation, we construct estimators of both model parameters and sampling stepsize in a fully explicit way, and prove that they are jointly asymptotically normally distributed. High order uniform integrability of the obtained estimator is also derived.
Further, we propose the Schwarz (BIC) type statistics for model selection and show its model-selection consistency.
We conducted some numerical experiments and found that the observed finite-sample performance well supports our theoretical findings.
Also provided is a real data example.
\end{abstract}


\section{Introduction}

Consider the $d$-dimensional parametric ergodic diffusion model given by
\begin{align}
dX_{t}=\sqrt{\tau}a(X_{t},\al)dw_{t}+\tau b(X_{t},\theta)dt,
\nn
\end{align}
where $\theta:=(\al,\beta)$ is the statistical parameter of interest, whose true value is assumed to exist and denoted by $\tz=(\al_{0},\beta_{0})$, $\tau>0$ is a nuisance parameter, and $w$ is an $d$-dimensional standard {\wp}.
Suppose that we observe an equally spaced high-frequency data $(X_{t_{j}})_{j=0}^{n}$ for $t_{j}=t^{n}_{j}=jh_{0,n}$ where $h_{0}=h_{0,n}$ is an unknown sampling stepsize fulfilling that
\begin{align}
\text{$T_{n}:=nh_{0}\to\infty$ \quad and \quad $nh_{0}^{2}\to 0$.}
\label{hm:sampling.design}
\end{align}
We are interested in developing a methodology to estimate $\tz$ from $(X_{t_{j}})_{j=0}^{n}$ \textit{with leaving $h_{0}$ and $\tau$ unspecified}.
The nuisance parameter $\tau$ measures model-time scale: the time-rescaled process $X^{\tau}$ with $X^{\tau}_{t}:=X_{t/\tau}$ satisfies the stochastic differential equation
\begin{equation}
dX^{\tau}_{t}=a(X_{t}^{\tau},\al)dw^{\tau}_{t} + b(X^{\tau}_{t},\theta)dt,
\nonumber
\end{equation}
where $w^{\tau}_{\cdot}:=\sqrt{\tau}w_{\cdot/\tau}$ is a standard {\wp}.
As explained later on (in particular, see Remark \ref{se:rem.taudiff}), in the proposed estimation procedure it is quite natural and even necessary to incorporate the nuisance parameter $\tau$.

There exists a large literature studying parametric estimation of $\tz$ based on $(X_{t_{j}})_{j=0}^{n}$ through the small-time approximation of the conditional mean and variance (location and scale) taken under the law of $X$ associated with $\theta$:
\begin{align}
\E_{\theta}(X_{t_{j}}| X_{t_{j-1}}) &\approx X_{t_{j-1}} + hb(X_{t_{j-1}},\theta), \nn\\
\var_{\theta}(X_{t_{j}}| X_{t_{j-1}}) &\approx ha^{\otimes 2}(X_{t_{j-1}},\al),
\nonumber
\end{align}
where $A^{\otimes}:=AA^{\top}$ for a matrix $A$ with $\top$ denoting the transpose, and where
\begin{equation}
h=h_{n}:=\tau h_{0} \qquad (nh\to\infty, \ nh^{2}\to0).
\nonumber
\end{equation}
Due to the Gaussianity of the driving noise process, this naturally leads to the logarithmic Gaussian quasi-likelihood function (GQLF) based on the small-time Gaussian approximation
\begin{align}
\mcl(X_{t_{j}}|X_{t_{j-1}}=x) \approx N_{d}\left( x+hb(x,\theta),\, ha^{\otimes 2}(x,\al)\right)
\label{hm:gauss.approx}
\end{align}
for the unknown transition probability distribution.
Let us note that the high-frequency setting enables us to develop a unified strategy of parameter estimation for a quite general class of non-linear diffusions.
Under appropriate regularity conditions, this quasi-likelihood is known to be \textit{theoretically} asymptotically efficient.
See \cite{Gob02}, \cite{Kes97}, \cite{UchYos12}, and the references therein.

The existing \textit{theoretical} literature basically supposes that the unknown quantities $h_{0}$ and $\tau$ are given {\it a priori}, that is, the existing theories have been developed under known $h(=\tau h_{0})$.
In practice, the value of $\tau$ is unknown, and there is no absolute correspondence between sampling stepsize and a given time series data associated with a time stamp (the time at which the data is observed, a typical format being \texttt{YYYYMMDD hh:mm:ss}).
One would then get confused with the \textit{practical} problem ``what value is to be assigned to $h_{0}$'',
which in the present case \eqref{hm:sampling.design} should not be too large and too small; for example, which value is to be selected to represent $h_{0}$ to be one minute?
A common \textit{consensus} may be to subjectively assign $h$ with a sufficiently small value $\ll 1$ in an \textit{arbitrary} manner satisfying \eqref{hm:sampling.design}.
Obviously, different values of $h$ lead to different finite-sample performances of estimates;
this is an \textit{arbitrariness problem} which has not received much attention in the literature of high-frequency statistics under $T_{n}\to\infty$, though it might not be of big concern if one has a practical reasoning for assigning a specific value for $h_{0}$
(e.g. when one has several daily-data sets over ten years, then we set $T_{n}=10$ with $h_{0}=1/365$ ($\tau=1$), etc.).
In this respect, $h=\tau h_{0}$ could be regarded as an unknown quantity to be selected in a certain appropriate manner, giving rise to a statistical inference problem of the \textit{GQLF-model time scale} $h$ against \textit{actual-time scale}; of course, cases of random-sampling models and time-changed type processes have the same problem.
Despite its practical importance, theoretical study on unknown $h$ seems to have been lacking and/or ignored in the literature of statistics for high-frequency data.
Single subjective choice of $h$ would be a rather subtle problem as we are considering vanishing $h$ (see Remark \ref{hm:rem_inconsistency}), hence so would be single selection of $h$ as a fine-tuning parameter indexing the statistical model.

\medskip

The objective of this paper is to clarify ``when and how'' we can sidestep the subjective choice of $h$ through the GQLF based on \eqref{hm:gauss.approx}.
In the current case, since the Gaussian quasi-likelihood only looks at the mean and variance structure, we should note that when $h_0$ is unknown the GQLF can enables us only to consistently estimate the product $\tau h_{0}$, making the parameter $\tau$ itself non-identifiable.
Nevertheless, under suitable conditions it is possible to develop an asymptotic distributional theory not only for the parameter $\theta$ of interest but also $h$ as well.
This will be done through the modified logarithmic Gaussian quasi-likelihood function (mGQLF), which is defined through profiling out the variable $h$.
The proposed mGQLF is fully explicit, while producing an estimator having the following good properties under appropriate regularity conditions:
\begin{itemize}
\item The proposed estimator of $\theta=(\al,\beta)$ is asymptotically normally distributed at the same rate as in cases where $\tau$ and $h_{0}$ are known, namely $\sqrt{n}$ for $\al$ and $\sqrt{nh_0}$ for $\beta$, both of which are well-known to be best possible;
\item The value $h$ can be quantitatively estimated without any specific form of $n\mapsto h=h_{n}$.
\end{itemize}
The results are made precise in Theorem \ref{hm:thm2.diffusion.an} in Section \ref{hm:sec_ergo.diff.joint}, 
from which, in particular, it is trivial that we can estimate $\tau$ at rate $\sqrt{n}$ as soon as $h_{0}$ is subjectively given {\it a priori} (see Remark \ref{hm:rem_add1}).
Once we get an estimate of $h$, we can obtain the formal approximate predictive distribution from \eqref{hm:gauss.approx}. 
There we will also provide a two-step estimation procedure which has been well-developed and is nowadays standard in cases where $h_{0}$ and $\tau=1$ are known (see \cite{KamUch15} and \cite{UchYos12}).
Moreover, we also provide handy sufficient conditions for the polynomial type large deviation inequality (PLDI) associated with the mGQLF, which in particular guarantees convergence of moments of the proposed estimator; see \cite[Section 6]{Yos11} for sufficient conditions in the case where $h_{0}$ and $\tau=1$ are known.

Through the proposed estimator, we can give an interpretation of the model-time scale.
Further, the proposed estimation procedure is simple enough, and it should be potentially applicable to models other than the ergodic diffusion, whenever an explicit GQLF is used so that we can remove its dependence on $h$ (see Section \ref{hm:sec_ergo.diff.joint}), possibly including non-ergodic continuous semimartingale models (\cite{GenJac93} and \cite{UchYos13}) and L\'{e}vy driven stochastic differential equation (\cite{Mas13} and \cite{MasUeh17}).
This is also the case for the stepwise estimation procedure considered in Section \ref{hm:sec_ergo.diff.2step}, as long as high-frequency sampling is concerned.

Another objective of this paper is Schwarz type model comparison for the ergodic diffusion models with unknown sampling stepsize.
The classical Bayesian information criterion (BIC), which is derived based on the Bayesian principle for model selection, is used to look for better model description.
We will introduce the BIC type statistics through the stochastic expansion of the proposed mGQLF.
In cases where the candidate models are given by the ergodic diffusion models with known $h$, \cite{EguMas18a} has introduced BIC type statistics, and 
studied their model-selection consistency.
We note that many authors have investigated the information criteria concerning sampled data from stochastic process models; see, for example \cite{SeiKom07}, \cite{Uch10}, \cite{UchYos01}, \cite{UchYos06}, and\cite{UchYos16}.
Still, there has been no previous work concerned with unknown sampling stepsize $h_{0}$.

\medskip


This paper is organized as follows.
In Section \ref{hm:sec_ergodiff}, we describe the basic model setup, propose the modified logarithmic Gaussian quasi-likelihood and parameter estimation method, and then present asymptotic properties of the estimators of the model parameter and $h$.
Furthermore, we give the sufficient conditions of the PLDI under the modified logarithmic Gaussian quasi-likelihood.
In Section \ref{hm:sec_qbic} we derive the BIC type statistics in case where $h$ is unknown and discuss the model selection consistency with respect to the true model.
In Section \ref{se:simulation}, some numerical experiments are carried out to check the numerical performance of our asymptotic results.
All the proofs are given in Section \ref{hm:sec_proofs}.

\medskip

Here are some basic notations used throughout this paper.
Let $\D_{j}Y:=Y_{t_{j}}-Y_{t_{j-1}}$ for a process $Y$, and $f_{j-1}(\theta):=f(X_{t_{j-1}},\theta)$ for any measurable function on $f:\mbbr^{d}\times\Theta$.
We denote by $|A|$ the determinant of a square matrix $A$ and by $\|A\|$ the Frobenius norm of a matrix $A$.
We write $A[B]=\tr(AB^{\top})$ for the matrices $A$ and $B$ of the same sizes.
The symbol $\p_{a}^{k}$ stands for $k$-times partial differentiation with respect to variable $a$.
We denote by $C$ a universal positive constant, which may change at each appearance, and write $A_{n} \lesssim B_{n}$ if $A_{n}\le CB_{n}$ for every $n$ large enough.

\section{Gaussian quasi-likelihood inference with unknown time scale}\label{hm:sec_ergodiff}

\subsection{Setup}

Consider a $d$-dimensional diffusion process given by
\begin{equation}
dX_{t}=\sqrt{\tau}a(X_{t},\alpha)dw_{t}+\tau b(X_{t},\theta)dt, \quad X_{0}=x_{0},
\label{hm:sde1}
\end{equation}
where $\theta:=(\alpha,\beta)\in\Theta_{\alpha}\times\Theta_{\beta}\subset\mbbr^{p_{\alpha}}\times\mbbr^{p_{\beta}}=\mbbr^{p}$ is an unknown parameter,
$a$ is a symmetric $\mbbr^{d}\otimes\mbbr^{d}$-valued function on $\mbbr^{d}\times\Theta_{\alpha}$, $b$ is an $\mbbr^{d}$-valued function on $\mbbr^{d}\times\Theta_{\alpha}\times\Theta_{\beta}$, $\tau>0$ is an unknown constant, $w$ is an $d$-dimensional standard Wiener process, and $x_{0}$ is a random variable independent of $w$.
We assume that \eqref{hm:sampling.design} holds and that there exists a value $\tz=(\alpha_{0},\beta_{0})\in\Theta_{\al}\times\Theta_{\beta}$ which induces the distribution of $X$, which we denote by $\pr$, and also that $\Theta_{\al}$ and $\Theta_{\beta}$ are bounded convex domains.

Let $S(x,\al):=a^{\otimes 2}(x,\al)$ and denote by $\lam_{\min}\{S(x,\al)\}$ the minimum eigenvalue of $S(x,\al)$.

\begin{ass}[Smoothness and non-degeneracy]
The coefficients $a$ and $b$ satisfy that $a,b\in\mcc^{2,3}(\mbbr^{d}\times \Theta)$, and they, together with their partial derivatives, can be continuously extended to the boundary of $\Theta$ as functions of $\theta$. Moreover, the following conditions hold.
\begin{itemize}
\item[(i)] For $x_{1},x_{2}\in\mbbr^{d}$,
\begin{align*}
\sup_{\alpha\in\overline{\Theta}_{\alpha}}\|a(x_{1},\alpha)-a(x_{2},\alpha)\|+\sup_{\theta\in\overline{\Theta}_{\alpha}\times\overline{\Theta}_{\beta}}\|b(x_{1},\theta)-b(x_{2},\theta)\|
\lesssim \|x_{1}-x_{2}\|.
\end{align*}
\item[(ii)] 
There exists a constant $C_{0}\ge 0$ such that for $x\in\mbbr^{d}$ and $i\in \{0,1,2,3\}$ and $j \in \{0,1,2\}$,
\begin{align}
& \sup_{\alpha\in\overline{\Theta}_{\alpha}}\|\p_{x}^{j}\p_{\alpha}^{i}a(x,\alpha)\| + \sup_{\theta\in\overline{\Theta}_{\alpha}\times\overline{\Theta}_{\beta}}\|\p_{x}^{j}\p_{\alpha}^{i}\p_{\beta}^{i}b(x,\theta)\|  
\lesssim (1+\|x\|)^{C_{0}}, \nn\\
& \inf_{\alpha\in\overline{\Theta}_{\alpha}}\lam_{\min}\{S(x,\al)\} \gtrsim (1+\|x\|)^{-C_{0}}.
\nonumber
\end{align}
\end{itemize}
\label{Ass1}
\end{ass}

The Gershgorin circle theorem says that
\begin{equation}
\inf_{\al}\lam_{\min}\{S(x,\al)\} \ge \inf_{\al} \min_{1\le i\le d}\bigg( S_{ii}(x,\al) - \sum_{j\ne i}|S_{ij}(x,\al)| \bigg),
\nonumber
\end{equation}
so that an easy sufficient condition for the last inequality in (ii) is that this lower bound is bounded below by $C(1+|x|)^{-C_{0}}$.

\begin{ass}[Stability]
There exists a probability measure $\pi=\pi_{\tz}$ such that
\begin{align*}
\frac{1}{T}\int_{0}^{T}g(X_{t})dt\cip\int_{\mbbr^{d}}g(x)\pi(dx),\qquad T\to\infty,
\end{align*}
for any measurable function $g \in L^{1}(\pi)$.
In addition, $\sup_{t\in\mbbrp}\E(\|X_{t}\|^{q}) <\infty$ for all $q>0$ in case where the constant $C_{0}$ in Assumption \ref{Ass1}(ii) is positive.
\label{Ass2p}
\end{ass}

It follows from Assumptions \ref{Ass1} and \ref{Ass2p} that
\begin{equation}
\frac{1}{n}\sumj g(X_{t_{j-1}}) \cip \int g(x)\pi(dx),\qquad n\to\infty,
\label{hm:disc.LLN}
\end{equation}
for any measurable function $g$ of at most polynomial growth (see \cite[p.1598]{Mas13}).
There are several polynomially ergodic diffusions with bounded smooth coefficients and uniformly elliptic diffusion coefficient, in which case we may set $C_{0}=0$ in Assumption \ref{Ass1} while the boundedness of moments in Assumption \ref{Ass2p} may fail to hold (see \cite{Ver97}).
In general, one can consult \cite{Gob02} and \cite{Ver87} for easy conditions for the boundedness of (more strongly, exponential) moments;
see also Lemma \ref{hm:lem_pldi-g.exp.ergo}.

\subsection{Joint estimation}\label{hm:sec_ergo.diff.joint}

The logarithmic GQLF (\cite{Kes97}, \cite{UchYos12}) of the true model \eqref{hm:sde1} based on the approximation \eqref{hm:gauss.approx} is given by
\begin{align}
\mbbh_{n}(\theta;h)
&=-\frac{1}{2}\sumj\left\{\log\big|2\pi hS_{j-1}(\alpha)\big|+\frac{1}{h}S_{j-1}^{-1}(\alpha)\left[\big(\D_{j}X-hb_{j-1}(\theta)\big)^{\otimes2}\right]\right\}.
\label{ErLf}
\end{align}
Our objective is to estimate $\theta$ and $h$ simultaneously under \eqref{hm:sampling.design}.

The function $h\mapsto \mbbh_{n}(\theta;h)$ is a.s. smooth in $h>0$.
In order to profile out $h$ from $\mbbh_{n}(\theta;h)$, we consider optimizing $h\mapsto\mbbh_{n}(\theta;h)$ with $\theta$ fixed:
the equation $\p_{h}\mbbh_{n}(\theta;h)=0$ with respect to $h>0$ is equivalent to a certain quadratic equation which admits the a.s. positive explicit solution
\begin{align}
h^{\star}(\theta)
&:= \frac{1}{2}\bigg(\frac{1}{n}\sumj S_{j-1}^{-1}(\alpha)\big[b_{j-1}(\theta)^{\otimes2}\big]\bigg)^{-1}
\nn\\
&{}\qquad 
\cdot\bigg[ -d+\bigg\{
d^{2}+4\bigg(\frac{1}{n}\sumj S_{j-1}^{-1}(\alpha)\big[(\D_{j}X)^{\otimes2}\big]\bigg)\bigg(\frac{1}{n}\sumj S_{j-1}^{-1}(\alpha)\big[b_{j-1}(\theta)^{\otimes2}\big]\bigg)
\bigg\}^{1/2} \bigg].
\nn
\end{align}
This is somewhat complicated, hence under the high-frequency setting we suggest approximating $h^{\star}(\theta)$ by the leading term
\begin{equation}
h(\al):=\frac{1}{nd}\sumj S_{j-1}^{-1}(\alpha)\big[(\D_{j}X)^{\otimes2}\big].
\nn
\end{equation}
Indeed, in Section \ref{hm:sec_proofs} we will observe that $h^{\star}(\theta) = h(\al) + O_{p}(h_{0}^{2})$ with $h(\al)=O_{p}(h_{0})$ uniformly in $\al$.

Then we define the fully explicit modified GQLF (mGQLF) $\tilde{\mbbh}_{n}(\theta)$ by replacing $h$ in $\mbbh_{n}(\theta;h)$ with $h(\alpha)$:
\begin{align}
\tilde{\mbbh}_{n}(\theta)&:=\mbbh_{n}\big(\theta;h(\alpha)\big) \nn\\
&= -\frac{nd}{2}\{1+\log(2\pi) \} - \frac{1}{2}\bigg( \sumj\log\big|S_{j-1}(\alpha)\big|+nd\log h(\al) \bigg) \nn\\
&\quad+\bigg\{\sumj S_{j-1}^{-1}(\alpha)\big[\D_{j}X,b_{j-1}(\theta)\big]
-\frac{h(\al)}{2}\bigg(\sumj S_{j-1}^{-1}(\alpha)\big[b_{j-1}(\theta)^{\otimes2}\big]\bigg)\bigg\}
\nn\\
&= -\frac{nd}{2}\{1+\log(2\pi) \} 
-\frac{1}{2}\bigg\{\sumj\log\big|S_{j-1}(\alpha)\big|+nd\log\bigg(\frac{1}{nd}\sumj S_{j-1}^{-1}(\alpha)\big[(\D_{j}X)^{\otimes2}\big]\bigg)\bigg\} \nn\\
&\quad+\bigg\{\sumj S_{j-1}^{-1}(\alpha)\big[\D_{j}X,b_{j-1}(\theta)\big]
\nn\\
&{}\qquad
-\frac{1}{2}\bigg(\frac{1}{nd}\sumj S_{j-1}^{-1}(\alpha)\big[(\D_{j}X)^{\otimes2}\big]\bigg)\bigg(\sumj S_{j-1}^{-1}(\alpha)\big[b_{j-1}(\theta)^{\otimes2}\big]\bigg)\bigg\}.
\label{hm:mGQMLE.def}
\end{align}
Correspondingly, we define the modified Gaussian quasi-maximum likelihood estimator $\tet$ (mGQMLE) by any maximizer of $\tilde{\mbbh}_{n}$:
\begin{align*}
\tet=(\tilde{\alpha}_{n},\tilde{\beta}_{n})\in\argmax_{\theta\in\overline{\Theta}}\tilde{\mbbh}_{n}(\theta),
\end{align*}
computations of which does not require the value $h$.
The definition of $\tet$ approximately corresponds to a solution to the system of estimating equations
$\left( \p_{h}\mbbh_{n}(\theta;h),\, \p_{\theta}\mbbh_{n}(\theta;h) \right) = (0,0)$ with respect to $\theta$.
We emphasize that the approximation of $h^{\star}(\theta)$ by $h(\al)$ provides us with a very simple form, reducing computational cost in optimization.

\medskip

We need the following identifiability condition with additional non-degeneracy.

\begin{ass}[Identifiability and non-degeneracy]
The following conditions hold for the invariant distribution $\pi(dx)$ in Assumption \ref{Ass2p}.
\begin{itemize}
\item[(i)] The functions $x\mapsto \tr\{S^{-1}(x,\al)S(x,\al_{0})\}$ for $\al\ne\al_{0}$
and $x\mapsto \tr\{S^{-1}(x,\al_{0})\p_{\al}S(x,\al_{0})\}$ are not constant over the support of $\pi$.
\item[(ii)] If $b(\cdot,\alpha_{0},\beta) = b(\cdot,\alpha_{0},\beta_{0})$ $\pi$-a.e., then $\beta=\beta_{0}$.
\end{itemize}
\label{Ass3}
\end{ass}

We implicitly assume that the support of $\pi$ is known a priori; in many cases it equals $\mbbr^{d}$, or $\mbbr_{+}$ when $d=1$.
Moreover, the identifiability condition of $\beta$ is standard.
By contrast, as for $\al$ it is insufficient to only suppose as usual that ``$a(\cdot,\al)=a(\cdot,\al_{0})$ $\pi$-a.e. implies $\al=\al_{0}$''.
Concerning Assumption \ref{Ass3}(i), the former non-constancy ensures the unique maxima of the quasi-relative entropy $\tilde{\mbby}_{0}^{1}(\al)$,
and the latter one does the positive definiteness of the quasi-Fisher information matrix $\tilde{\Gam}_{1,0}$ of $\al$; see Section \ref{proof.AN} for details.
Note that Assumption \ref{Ass3}(i) appropriately excludes the presence of a multiplicative parameter in diffusion coefficient as well as constant diffusion coefficient, both of which are, when they are scalar, to be absorbed into the nuisance parameter $\tau$; for example, $\al$ is non-identifiable in the case $S(x,\al)=\sum_{j=1}^{p_{\alpha}}\al_{j} S_{j}(x)$.

\begin{rem}{\rm
Let us consider the sufficient condition of the former non-constancy of Assumption \ref{Ass3}(i) in the case where the function $S(x,\alpha)$ is given by
\begin{align*}
S(x,\alpha)=\exp\left( \sum_{j=1}^{p_{\alpha}}\alpha_{j}S_{j}(x) \right),
\end{align*}
where $\exp(A):=\sum_{k=0}^{\infty}(k!)^{-1}A^{k}$ for a square matrix $A$, and $S_{1}(x),\ldots,S_{p_{\alpha}}(x)$ are $d\times d$ non-zero matrices.
If $S_{i}(x)S_{j}(x)=S_{j}(x)S_{i}(x)$ for any $x$ and $i,j\in\{1,\ldots,p_{\alpha}\}$,
we have
\begin{align}
& S^{-1}(x,\alpha)S(x,\alpha_{0}) = \exp\bigg( \sum_{j=1}^{p_{\alpha}}(\alpha_{j,0}-\alpha_{j})S_{j}(x) \bigg), \nn\\
& \tr\{S^{-1}(x,\alpha)S(x,\alpha_{0})\} = \sum_{k=1}^{d}\exp(\lambda_{k}^{\prime}(x,\alpha)),
\nonumber
\end{align}
where $\lambda_{1}^{\prime}(x,\alpha),\ldots,\lambda_{d}^{\prime}(x,\alpha)$ denote the eigenvalues of $\sum_{j=1}^{p_{\alpha}}(\alpha_{j,0}-\alpha_{j})S_{j}(x)$.
Hence, the former non-constancy of Assumption \ref{Ass3}(i) holds if the following conditions are satisfied:
\begin{itemize}
\item[(i)] For any $x$ and $i,j\in\{1,\ldots,p_{\alpha}\}$, $S_{i}(x)S_{j}(x)=S_{j}(x)S_{i}(x)$;
\item[(ii)] For any $\alpha\ne\alpha_{0}$, there exists a $k\leq d$ such that the function $x\mapsto\lambda_{k}^{\prime}(x,\alpha)$ is not constant.
\end{itemize}
}\qed
\end{rem}

\begin{rem}{\rm
Let us consider the sufficient condition of the former non-constancy of Assumption \ref{Ass3}(i) in the case where the function $S(x,\alpha)$ is given by
\begin{align*}
S(x,\alpha)=\exp\left( \sum_{j=1}^{p_{\alpha}}\alpha_{j}S_{j}(x) \right),
\end{align*}
where $\exp(A):=\sum_{k=0}^{\infty}(k!)^{-1}A^{k}$ for a square matrix $A$, and $S_{1}(x),\ldots,S_{p_{\alpha}}(x)$ are $d\times d$ non-zero matrices.
If $S_{i}(x)S_{j}(x)=S_{j}(x)S_{i}(x)$ for any $x$ and $i,j\in\{1,\ldots,p_{\alpha}\}$,
we have
\begin{align}
& S^{-1}(x,\alpha)S(x,\alpha_{0}) = \exp\bigg( \sum_{j=1}^{p_{\alpha}}(\alpha_{j,0}-\alpha_{j})S_{j}(x) \bigg), \nn\\
& \tr\{S^{-1}(x,\alpha)S(x,\alpha_{0})\} = \sum_{k=1}^{d}\exp(\lambda_{k}^{\prime}(x,\alpha)),
\nonumber
\end{align}
where $\lambda_{1}^{\prime}(x,\alpha),\ldots,\lambda_{d}^{\prime}(x,\alpha)$ denote the eigenvalues of $\sum_{j=1}^{p_{\alpha}}(\alpha_{j,0}-\alpha_{j})S_{j}(x)$.
Hence, the former non-constancy of Assumption \ref{Ass3}(i) holds if the following conditions are satisfied:
\begin{itemize}
\item[(i)] For any $x$ and $i,j\in\{1,\ldots,p_{\alpha}\}$, $S_{i}(x)S_{j}(x)=S_{j}(x)S_{i}(x)$;
\item[(ii)] For any $\alpha\ne\alpha_{0}$, there exists a $k\leq d$ such that the function $x\mapsto\lambda_{k}^{\prime}(x,\alpha)$ is not constant.
\end{itemize}
}\qed
\end{rem}

\begin{rem}{\rm 
For now, we note that subjective choice of $h$ is sensitive to identify $\al_{0}$.
Suppose that a positive sequence $h'=h'_{n}$ also satisfies \eqref{hm:sampling.design}, while $h^{\prime}/h_{0}\to c$ for some constant $c\neq1$.
Then, we can show that $n^{-1}\{\mbbh_{n}(\theta;h^{\prime})-\mbbh_{n}(\alpha_{0},\beta;h^{\prime})\}\cip\bar{\mbbh}_{0}^{1}(\alpha;c)$ for some $\bar{\mbbh}_{0}^{1}(\alpha;c)$ uniformly in $\al$. It can be seen that $\{\alpha_{0}\} \subsetneqq \argmax_{\alpha}\bar{\mbbh}_{0}^{1}(\alpha;c)$,
hence the inconsistency of any element in $\argmax_{\theta}\mathbb{H}_{n}(\theta;h')$.
Needless to say, the situation is even worse for the cases $h^{\prime}/h_{0}\to\infty$ and $h^{\prime}/h_{0}\to0$, where $n^{-1}\{\mbbh_{n}(\theta;h^{\prime})-\mbbh_{n}(\alpha_{0},\beta;h^{\prime})\}$ no longer has a proper limit.
\label{hm:rem_inconsistency}
}\qed\end{rem}

\medskip

If $h=\tau h_{0}$ is known, the estimator of $\beta$ has the convergence rate $\sqrt{nh}$. In the current setting where $h$ is unknown, we propose to estimate $h$ by
\begin{equation}
\tilde{h} := h(\aet) = \frac{1}{nd}\sumj S_{j-1}^{-1}(\aet)\big[(\D_{j}X)^{\otimes2}\big].
\nonumber
\end{equation}
In practice where $n$ is large enough, we may check whether or not the sampling condition \eqref{hm:sampling.design} by looking at the values $n\tilde{h}$ and $n\tilde{h}^{2}$ which are to be large and small enough, respectively.
If not, in order to make \eqref{hm:sampling.design} more likely we may formally ``shrink'' or ``spread'' $h$ through multiplying $X_{t_{j}}$ by some constant $c>0$:
with replacing $X_{t_{j}}$ by $cX_{t_{j}}$, we have
\begin{equation}
\tilde{h} = \frac{c^{2}}{nd}\sumj S(cX_{t_{j-1}},\aet)^{-1}[(\D_{j}X)^{\otimes 2}],
\label{hm:h.tilde+1}
\end{equation}
hence, under the uniform non-degeneracy of $S$, choosing $c>1$ (resp. $c< 1$) will shrink (resp. spread) value of $\tilde{h}$.

\medskip

Let
\begin{align}
K[u_{1}]&=\frac{1}{d}\int_{\mbbr^{d}}\tr\Big(S^{-1}(x,\alpha_{0})\big(\p_{\alpha}S(x,\alpha_{0})\big)\Big)[u_{1}]\pi(dx), \nn\\
\tilde{\Gam}_{1,0}[u_{1}^{\otimes2}]
&=\frac{1}{2}\int_{\mbbr^{d}}\tr\Big(S^{-1}(x,\alpha_{0})\big(\p_{\alpha}S(x,\alpha_{0})\big)S^{-1}(x,\alpha_{0})\big(\p_{\alpha}S(x,\alpha_{0})\big)\Big)[u_{1}^{\otimes2}]\pi(dx) \nn\\
&\qquad -\frac{1}{2d}\bigg\{\int_{\mbbr^{d}}\tr\Big(S^{-1}(x,\alpha_{0})\big(\p_{\alpha}S(x,\alpha_{0})\big)\Big)[u_{1}]\pi(dx)\bigg\}^{2},
\nn
\\
\tilde{\Gam}_{2,0}[u_{2}^{\otimes2}]
&=\int_{\mbbr^{d}}S^{-1}(x,\alpha_{0})\big[\p_{\beta}b(x,\tz)[u_{2}],\p_{\beta}b(x,\tz)[u_{2}]\big]\pi(dx)
\nn
\end{align}
for $u_{1}\in\mbbr^{p_{\alpha}}$ and $u_{2}\in\mbbr^{p_{\beta}}$.
Now we are in position to state the joint asymptotic normality of the mGQMLE $\tet$ and $\tilde{h}$.

\begin{thm}
Under Assumptions \ref{Ass1}, \ref{Ass2p}, and \ref{Ass3}, the matrices $\tilde{\Gam}_{1,0}$ and $\tilde{\Gam}_{2,0}$ are positive definite and
\begin{equation}
\bigg(\sqrt{n}\bigg(\frac{\tilde{h}}{\tau h_{0}}-1\bigg),\, \sqrt{n}(\tilde{\alpha}_{n}-\alpha_{0}),\, \sqrt{n\tilde{h}}(\tilde{\beta}_{n}-\beta_{0})\bigg)
\cil N_{1+p}\big( 0,\, \Sigma(\tz) \big),
\label{hm:diffusion.an1+}
\end{equation}
where
\begin{align}
\Sigma(\tz) &:= \left(
\begin{array}{ccc}
2/d + K^{\top}\tilde{\Gam}_{1,0}^{-1}K & & \text{{\rm sym.}} \\
-\tilde{\Gam}_{1,0}^{-1}K & \tilde{\Gam}_{1,0}^{-1} & \\
0 & 0 & \tilde{\Gam}_{2,0}^{-1}
\end{array}
\right).
\nonumber
\end{align}
Further we have
\begin{align}
\tilde{K}_{n} &:= \frac{1}{d}\tr\bigg( \frac{1}{n}\sumj S_{j-1}^{-1}(\aet) \left(\p_{\alpha}S_{j-1}(\aet)\right) \bigg) \cip K,
\nn\\
\tilde{\Gam}_{1,n} &:= 
\tr\bigg( \frac{1}{2n} \sumj S_{j-1}^{-1}(\aet) \big(\p_{\alpha}S_{j-1}(\aet) \big)S^{-1}_{j-1}(\aet) \big(\p_{\alpha}S_{j-1}(\aet) \big) \bigg) \nn\\
&\qquad -\frac{1}{2d}\bigg\{ \tr\bigg( \frac{1}{n}\sumj S_{j-1}^{-1}(\aet) \big(\p_{\alpha}S_{j-1}(\aet)\big)\bigg)\bigg\}^{\otimes 2} \cip \tilde{\Gam}_{1,0},
\nn\\
\tilde{\Gam}_{2,n} &:= 
\frac{1}{n}\sumj S^{-1}_{j-1}(\aet)\big[\p_{\beta}b_{j-1}(\aet,\bet),\, \p_{\beta}b_{j-1}(\aet,\bet)\big] \cip \tilde{\Gam}_{2,0},
\nonumber
\end{align}
from which we can obtain a consistent estimator of $\Sig(\tz)$.
\label{hm:thm2.diffusion.an}
\end{thm}

\medskip

In particular, Theorem \ref{hm:thm2.diffusion.an} implies that $\tilde{h}/(\tau h_{0})=1+O_{p}(n^{-1/2})$, hence
\begin{equation}
\frac{\sqrt{n}}{\tilde{h}}(\tilde{h}-\tau h_{0}) \cil N\left( 0,\, \frac{2}{d} + K^{\top}\tilde{\Gam}_{1,0}^{-1}K\right)
\nonumber
\end{equation}
as well.
Therefore, for any $\gam\in(0,1)$, the $100\times(1-\gam)$-percent confidence interval of $h$ is given by
\begin{align*}
\tilde{h} \pm z_{\gam/2} \frac{\tilde{h}}{\sqrt{n}}\sqrt{\frac{2}{d}+\tilde{K}_{n}^{\top}\tilde{\Gam}^{-1}_{1,n}\tilde{K}_{n}},
\end{align*}
where $z_{\gam/2}$ denotes the upper-$\gam/2$ percentile of $N(0,1)$.

\medskip

Several further remarks on Theorem \ref{hm:thm2.diffusion.an} are in order.

\begin{rem}{\rm 
We here do not assume any specific form on $h_{0}$ as a function of $n$.
A more direct estimation is possible upon assuming that, for example, the true sampling stepsize $h_{0}$ takes the form
\begin{equation}
\tau h_{0}=n^{-\kappa_{0}}
\nonumber
\end{equation}
for some unknown constant $\kappa_{0}\in(1/2,1)$. Let 
\begin{align*}
\tilde{\kappa}:=-\frac{\log\tilde{h}}{\log n}
=-\frac{1}{\log n}\log\left(\frac{1}{nd}\sumj S_{j-1}^{-1}(\tilde{\alpha})\big[(\D_{n}X)^{\otimes2}\big]\right),
\end{align*}
so that the equation $\tilde{h}=n^{-\tilde{\kappa}}$ holds.
Then, under the assumptions of Theorem \ref{hm:thm2.diffusion.an}, the delta method yields that
\begin{align*}
\sqrt{n}(\log n)\big(\tilde{\kappa}-\kappa_{0}\big)=-\sqrt{n}\left(\log\frac{\tilde{h}}{\tau h_{0}}-\log 1\right)
\cil N\bigg(0,\, \frac{2}{d} + K^{\top}\tilde{\Gam}_{1,0}^{-1}K\bigg),
\end{align*}
based on which it is straightforward to construct an approximate confidence interval of the index $\kappa_{0}$.

When $h_{0}$ satisfying \eqref{hm:sampling.design} is (subjectively) given {\it a priori} and $\tau$ is unknown, then, in order to verify the local asymptotic normality for $(\tau,\theta)$ we have to deal with the case where diffusion and drift coefficients contain a common parameter $\tau$. Unfortunately, the existing literature (see \cite{Gob02}) does not exactly cover this case. 
Nevertheless it is expected that the local asymptotic normality does hold in this case as well since the rate of convergence of α is strictly faster than that of $\beta$
and the orthogonality between $\al$ and $\beta$ (the block diagonality of the asymptotic covariance matrix) as specified in \eqref{hm:diffusion.an1+}, so that effect of unknown $\al$ in estimating $\beta$ contained in the drift would be negligible.
Moreover, in this case, estimation (a data driven choice) of the time-scale parameter $\tau$ is straightforward from Theorem \ref{hm:thm2.diffusion.an}: writing $\tilde{\tau}_{n}=\tilde{h}_{n}/h_{0}$ and denoting by $\tau_{0}$ the true value of $\tau$ for clarity, we can obtain
\begin{equation}
\sqrt{n}\left( \tilde{\tau}_{n} - \tau_{0} \right) \cil N_{1}\left( 0,\, \tau_{0}^{2}\left(2/d + K^{\top}\tilde{\Gam}_{1,0}^{-1}K\right)\right)
\nonumber
\end{equation}
directly from \eqref{hm:diffusion.an1+}; the same can be said for the stepwise estimator of $\tau$, see \eqref{hm:diffusion.an1_2step} below.
\label{hm:rem_add1}
}\qed\end{rem}

\begin{rem}{\rm 
When $h_{0}$ and $\tau$ are known, the GQMLE defined to be a maximizer $\tes^{\star}=(\aes^{\star},\bes^{\star})$ of $\mbbh_{n}(\cdot;h_{0})$ given by \eqref{ErLf} satisfies that
\begin{equation}
\big(\sqrt{n}(\aes^{\star}-\alpha_{0}),\, \sqrt{n\tau h_{0}}(\bes^{\star}-\beta_{0})\big)\cil N_{p}\left(0,\diag\{\Gam_{1,0}^{\star\ -1},\, \tilde{\Gam}^{-1}_{2,0}\big\}\right),
\label{hm:usual.gqmle.an}
\end{equation}
where for $u_{1}\in\mbbr^{p_{\alpha}}$
\begin{align}
\Gam_{1,0}^{\star}[u_{1}^{\otimes2}]
&=\frac{1}{2}\int_{\mbbr^{d}}\tr\Big(S^{-1}(x,\alpha_{0})\big(\p_{\alpha}S(x,\alpha_{0})\big)S^{-1}(x,\alpha_{0})\big(\p_{\alpha}S(x,\alpha_{0})\big)\Big)[u_{1}^{\otimes2}]\pi(dx).
\nn
\end{align}
See \cite{Kes97} for details.
Further, \cite{Gob02} proved that the estimator $\tes^{\star}$ is be asymptotically efficient under suitable regularity conditions.
Here are two remarks on comparing \eqref{hm:diffusion.an1+} and \eqref{hm:usual.gqmle.an}.
\begin{enumerate}
\item The second term of the right-hand side in the relation
\begin{equation}
\tilde{\Gam}_{1,0} = \Gam_{1,0}^{\star} - \frac{1}{2d}\bigg\{\int_{\mbbr^{d}}\tr\Big(S^{-1}(x,\alpha_{0})\big(\p_{\alpha}S(x,\alpha_{0})\big)\Big)\pi(dx)\bigg\}^{\otimes 2}
\nonumber
\end{equation}
quantitatively shows ``price'' for estimating $\al$ without knowing $h$. See also Remark \ref{hm:rem_info}.

\item The rate of convergence and the asymptotic covariance matrix of $\bet$ are the same as the best one in the case where $h$ is known, entailing that $\bet$ is asymptotically efficient. This is natural because of the asymptotic orthogonality between $\aes^{\star}$ and $\bes^{\star}$.
\end{enumerate}
}\qed\end{rem}

\begin{rem}{\rm 
Let $D_{n}=D_{n}(h_{0}):=\diag\big(\sqrt{n}I_{p_{\al}},\, \sqrt{n\tau h_{0}}I_{p_{\beta}}\big)$.
Then the observed information and formation matrices for estimating $(\theta,h)$ are given by
\begin{align*}
\mathcal{I}_{n}&=\left(\begin{array}{cc}
\mathcal{I}_{\theta,\theta} & \mathcal{I}_{\theta,h}  \\
\text{{\rm sym.}} & \mathcal{I}_{h,h} \\
\end{array}\right)
:=\left(\begin{array}{cc}
-D_{n}^{-1}\p_{\theta}^{2}\mbbh_{n}(\tz;\tau h_{0})D_{n}^{-1} & -\frac{\tau h_{0}}{\sqrt{n}}D_{n}^{-1}\p_{\theta}\p_{h}\mbbh_{n}(\tz;\tau h_{0}) \\
\text{{\rm sym.}} & -\frac{(\tau h_{0})^{2}}{n}\p_{h}^{2}\mbbh_{n}(\tz;\tau h_{0}) \\
\end{array}\right), \\
\mathcal{I}_{n}^{-1}&=\left(\begin{array}{cc}
\mathcal{I}^{\theta,\theta} & \mathcal{I}^{\theta,h} \\
\text{{\rm sym.}} & \mathcal{I}^{h,h} \\
\end{array}\right),
\end{align*}
respectively, where $\mathcal{I}^{\theta,\theta}$ is $p\times p$ matrix, $\mathcal{I}^{\theta,h}$ is $p$-dimensional vector, and $\mathcal{I}^{h,h}$ is $\mbbr$-valued.
Matrix manipulations give the expressions for the formation matrix: 
\begin{align}
\mathcal{I}^{\theta,\theta}&=\big(\mathcal{I}_{\theta,\theta}-\mathcal{I}_{\theta,h}\mathcal{I}_{h,h}^{-1}\mathcal{I}_{\theta,h}\big)^{-1}, \nn\\
\mathcal{I}^{\theta,h}&=-\mathcal{I}_{h,h}^{-1}\mathcal{I}_{\theta,h}\mathcal{I}^{\theta,\theta}, \nn\\
\mathcal{I}^{h,h}&=\mathcal{I}_{h,h}^{-1}+\big(\mathcal{I}^{\theta,\theta}\big)^{-1}\big[\big(\mathcal{I}^{\theta,h}\big)^{\otimes2}\big].
\nn
\end{align}
Then, we can show that
\begin{align*}
\mathcal{I}^{\theta,\theta}&\cip\diag(\tilde{\Gam}_{1,0}^{-1},\tilde{\Gam}^{-1}_{2,0}), \\
(\mathcal{I}^{\theta,h})_{j}&\cip\left\{\begin{array}{cc}
(-\tilde{\Gam}_{1,0}^{-1}K)_{j}, & (1\leq j\leq p_{\alpha}) \\
0,  & (p_{\alpha}< j\leq p)
\end{array}\right. \\
\mathcal{I}^{h,h}&\cip\frac{2}{d}+K^{\top}\tilde{\Gam}_{1,0}^{-1}K,
\end{align*}
Hence, $\mathcal{I}_{n}^{-1}$ converges to $\Sigma_{\tau}(\tz)$ in probability, where
\begin{align}
\Sigma_{\tau}(\tz) &:= \left(
\begin{array}{ccc}
2/d + K^{\top}\tilde{\Gam}_{1,0}^{-1}K & & \text{{\rm sym.}} \\
-\tilde{\Gam}_{1,0}^{-1}K & \tilde{\Gam}_{1,0}^{-1} & \\
0 & 0 & \tilde{\Gam}^{-1}_{2,0}
\end{array}
\right).
\nonumber
\end{align}
Building on these observations, it is expected that $\tilde{\theta}_{n}$ would be asymptotically efficient when $h_{0}$ is unknown, although we do not have a conventional Haj\'{e}k-Le Cam lower bound. We refer to \cite[Chapter 4]{PacSal97} and \cite[Section 7.3]{Hey97} for a systematic account for (quasi-)likelihood inference in the presence of nuisance parameters.
\label{hm:rem_info}
}\qed\end{rem}

\begin{rem}{\rm 
Let us explain why we have included the nuisance parameter $\tau$ in \eqref{hm:sde1} from the very beginning.
Indeed, profiling out the sampling stepsize $h_0$ from the Gaussian quasi-likelihood \eqref{ErLf} as before makes the constant $\tau$ in the diffusion coefficient $\tau S(x,\al)$ hidden.
In the proof of the consistency of $\tilde{\al}_n$ (Section \ref{hm:proof.consistency}), it can be seen that the second term in the rightmost side of \eqref{hm:mGQMLE.def} is concerned (in other words, the function $\tilde{\mbbh}_{1,n}(\al)$ defined in \eqref{hm:def_H1n_al}); the other terms are asymptotically negligible uniformly in $\theta$.
Obviously, the second term is invariant under replacing $S(x,\al)$ by $\sqrt{\tau}S(x,\al)$ for \textit{arbitrary} $\tau>0$, so that we can only identify the diffusion coefficient up to a multiplicative constant. Thus, at this stage it is natural to incorporate the nuisance multiplicative parameter $\tau$ in the diffusion coefficient:
\begin{align}
dX_{t}=\sqrt{\tau}a(X_{t},\alpha)dw_{t}+b(X_{t},\theta)dt,
\label{hm:rem.tau-3}
\end{align}
where the parameter $\al$ is identifiable under Assumption \ref{Ass3}.
However, it follows that the nuisance parameter $\tau \ne 1$ can make the drift parameter $\beta$ non-identifiable unless we replace the drift coefficient $b(x,\theta)$ in \eqref{hm:rem.tau-3} by $\tau b(x,\theta)$ with the same $\tau$ as above: for example, consider the model of the form $dX_{t}=\sqrt{\tau}a(X_{t},\alpha)dw_{t}+\tau'b(X_{t},\theta)dt$ for some constant $\tau'>0$. Then, exactly as in \eqref{hm:al-consis+2} in we can deduce that
\begin{align}
\tilde{\mbby}_{n}^{2}(\beta;\aet)
&\cip 
\tau' \int S^{-1}(x,\al_{0})\left[ b(x,\tz), b(x,\al_{0},\beta)-b(x,\tz)\right] \pi(dx) \nn\\
&{}\qquad -\frac{\tau}{2}\int S^{-1}(x,\al_{0})\left[ b(x,\al_{0},\beta)^{\otimes 2}-b(x,\tz)^{\otimes 2} \right] \pi(dx)
\label{hm:rem.tau-2}
\end{align}
uniformly in $\beta$ for the quasi Kullback-Leibler divergence $\tilde{\mbby}_{n}^{2}(\beta;\aet)$; see Section \ref{hm:proof.consistency} for the definition.
It follows that any maxima of the limit \eqref{hm:rem.tau-2} may differ from the true value $\beta_{0}$ unless $\tau^{\prime}=\tau$, thus validating the form \eqref{hm:sde1}.
Analogous remarks hold for the stepwise-estimation version described in Section \ref{hm:sec_ergo.diff.2step}.
}\qed \label{se:rem.taudiff}
\end{rem}

\subsection{Stepwise estimation}\label{hm:sec_ergo.diff.2step}

The modified GQLF is the mixed-rates type, that is, it consists of the sum of two terms converging to non-trivial limits at different rates.
In cases where $h$ is specified beforehand, it is well-known that stepwise estimation is possible;
see \cite{KamUch15} and \cite{UchYos12}, which can handle rather general sampling scheme than \eqref{hm:sampling.design}, as well as the references therein.
We will show that under \eqref{hm:sampling.design} it is still possible to formulate a two-step estimation procedure.

Let
\begin{align}
\tilde{\mbbh}_{1,n}(\al) &:= 
-\frac{1}{2}\bigg\{\sumj\log\big|S_{j-1}(\alpha)\big|+nd\log\bigg(\frac{1}{nd}\sumj S_{j-1}^{-1}(\alpha)\big[(\D_{j}X)^{\otimes2}\big]\bigg)\bigg\},
\label{hm:def_H1n_al} \\
\tilde{\mbbh}_{2,n}(\al,\beta) &:= \sumj S_{j-1}^{-1}(\alpha)\big[\D_{j}X,b_{j-1}(\theta)\big] \nn\\
&\quad-\frac{1}{2}\bigg(\frac{1}{nd}\sumj S_{j-1}^{-1}(\alpha)\big[(\D_{j}X)^{\otimes2}\big]\bigg)\bigg(\sumj S_{j-1}^{-1}(\alpha)\big[b_{j-1}(\theta)^{\otimes2}\big]\bigg),
\label{se:def_H2n_be}
\end{align}
so that $\tilde{\mbbh}_{n}(\theta) = -(nd/2)\{1+\log(2\pi)\} + \tilde{\mbbh}_{1,n}(\al) + \tilde{\mbbh}_{2,n}(\al,\beta)$; recall \eqref{hm:mGQMLE.def}.
Then, we estimate $\al$ and $\beta$ by $\tet'=(\aet',\bet')$ defined through the following step-by-step manner:
\begin{align}
\aet' &\in \argmax_{\alpha}\tilde{\mbbh}_{1,n}(\al), \label{se:two.al} \\
\bet' &\in \argmax_{\beta}\tilde{\mbbh}_{2,n}(\aet',\beta). \label{se:two.be}
\end{align}
We remark that the contrast function $\beta\mapsto \tilde{\mbbh}_{2,n}(\aet',\beta)$ may be regarded as a time-discretized version of the log-likelihood function of $\beta$ based on a continuous-time observation (see \cite{Kut04}), and also that $\bet'$ is explicit if $\beta\mapsto b(x,\aet',\beta)$ is linear.

The following theorem shows that $\tet$ and $\tet'$ have the same asymptotic distribution.

\begin{thm}
Under Assumptions \ref{Ass1}, \ref{Ass2p}, and \ref{Ass3}, we have
\begin{equation}
\bigg(\sqrt{n}\bigg(\frac{\tilde{h}^{\prime}}{\tau h_{0}}-1\bigg),\, \sqrt{n}(\tilde{\alpha}_{n}^{\prime}-\alpha_{0}),\, \sqrt{n\tilde{h}^{\prime}}(\tilde{\beta}_{n}^{\prime}-\beta_{0})\bigg)\cil N_{1+p}\big(0,\Sigma(\tz)\big),
\label{hm:diffusion.an1_2step}
\end{equation}
where $\tilde{h}^{\prime}=h(\tilde{\alpha}_{n}^{\prime})$.
\label{hm:thm_stepwise.ergo.diff}
\end{thm}

\medskip


\begin{rem}{\rm
In cases where the coefficients have a common parameter, we may follow the \textit{three-step} estimation as in \cite{MasUeh17}, by making use of a finite-sample bias correction.
In first step and second steps, we obtain $\tet'=(\aet',\bet')$ defined by \eqref{se:two.al} and \eqref{se:two.be} as before.
Then, using $\bet'$ and $\tilde{\mbbh}_{n}$, we update the estimator $\aet'$ by
\begin{align*}
\tilde{\alpha}_{n}'' \in\argmax_{\alpha}\tilde{\mbbh}_{n}(\alpha,\bet').
\end{align*}
While the third-step estimator $\aet''$ has the same asymptotic properties as $\aet'$, it may provide us with a significant bias reduction in finite samples.
In the unreported numerical experiments, we observed cases where $\aet''$ certainly reduce the bias of $\aet'$.
}\qed
\end{rem}

\subsection{Polynomial type large deviation inequality} \label{se:sec_pldi}

In this section, we will give sufficient conditions for the PLDI \eqref{pldi1} and \eqref{pldi2} below,
which ensure $L^{q}(\pr)$-boundedness of M- and Bayesian (parameter-integral) type estimators \cite{Yos11}, hence in particular convergence of their moments.
In case where $h$ is known and the coefficients do not have a common parameter, sufficient conditions for the PLDI can be found in \cite[Section 6]{Yos11}.

To state the result we introduce stronger regularity conditions, essentially borrowed from \cite{Gob02}.
Recall that $\lam_{\min}\{S(x,\al)\}$ and $\lam_{\max}\{S(x,\al)\}$ denote the minimum and maximum eigenvalues of $S(x,\al)$, respectively.

\begin{ass}
\begin{enumerate}
\item Assumption \ref{Ass1} holds and there exists a constant $C_{1}\ge 1$ for which
\begin{align}
& \left\| \p_{x}b(x,\theta)\right\| + \left\| \p_{x}S(x,\al)\right\| \le C_{1}, \nn\\
& C_{1}^{-1} \le \lam_{\min}\{S(x,\al)\} \le \lam_{\max}\{S(x,\al)\} \le C_{1},
\nonumber
\end{align}
for each $(x,\theta)$.
\item There exist positive constants $K_{1}$, $K_{2}$, and $\ep_{1}$, for which either one of the following holds:
\begin{enumerate}
\item $\E\{\exp(\ep_{1}\|X_{0}\|^{2})\}<\infty$ and
\begin{equation}
x^{\top} b(x,\theta) \le - K_{1} \|x\|^{2} + K_{2}
\nonumber
\end{equation}
for each $(x,\theta)$, or
\item $\E\{\exp(\ep_{1}\|X_{0}\|)\}<\infty$, $b(x,\theta)$ is essentially bounded, and
\begin{equation}
x^{\top} b(x,\theta) \le - K_{1} \|x\| + K_{2}
\nonumber
\end{equation}
for each $(x,\theta)$.
\end{enumerate}
\end{enumerate}
\label{Ass1+}
\end{ass}

We remark that Assumption \ref{Ass1+} implies Assumption \ref{Ass2p}: see Section \ref{se:proof.pldi} for details.
Further, it is the $g_{q}$-exponential ergodicity \eqref{hm:Ass:pldi1-1} that is essential in the proof of Theorem \ref{se:thm.pldi} below.
We could replace the uniform boundedness and ellipticity of $S$ and the drift condition in Assumption \ref{Ass1+} by any other ones which imply the $g_{q}$-exponential ergodicity.

\medskip

For $u_{1}\in\mbbr^{p_{\alpha}}$ and $u_{2}\in\mbbr^{p_{\beta}}$, we introduce the random fields $\tilde{\mbbz}_{n}^{1}$ and $\tilde{\mbbz}_{n}^{2}$ defined by
\begin{align*}
\tilde{\mbbz}_{n}^{1}(u_{1};\alpha_{0},\beta)&=\exp\left\{\tilde{\mbbh}_{n}\left(\alpha_{0}+\frac{u_{1}}{\sqrt{n}},\beta\right)-\tilde{\mbbh}_{n}\left(\alpha_{0},\beta\right)\right\}, \\
\tilde{\mbbz}_{n}^{2}(u_{2};\tz)&=\exp\left\{\tilde{\mbbh}_{n}\left(\alpha_{0},\beta_{0}+\frac{u_{2}}{\sqrt{n\tau h_{0}}}\right)-\tilde{\mbbh}_{n}\left(\alpha_{0},\beta_{0}\right)\right\},
\end{align*}
respectively.
In order to verify the high-order uniform integrability of the scaled estimators, tail behaviors of these random fields are crucial.
Let $\mbbu_{n}^{1}=\{u_{1}\in\mbbr^{p_{\alpha}};\alpha_{0}+u_{1}/\sqrt{n}\in\Theta_{\alpha}\}$ and $\mbbu_{n}^{2}=\{u_{2}\in\mbbr^{p_{\beta}};\beta_{0}+u_{2}/\sqrt{n\tau h_{0}}\in\Theta_{\beta}\}$.
Next theorem gives the PLDI for joint estimation case.

\begin{thm}
Assume that for some positive constant $\epsilon_{0}$, $nh_{0}\geq n^{\epsilon_{0}}$ for every $n$ large enough.
Let Assumptions \ref{Ass3} and \ref{Ass1+} hold. Then, for any positive number $L$ there exists a constant $C_{L}$ such that
\begin{align}
&\pr\left(\sup_{(u_{1},\beta)\in\{u_{1}\in\mbbu_{n}^{1};r\leq\|u_{1}\|\}\times\Theta_{\beta}}\tilde{\mbbz}_{n}^{1}(u_{1};\alpha_{0},\beta)\geq e^{-r}\right)\leq\frac{C_{L}}{r^{L}}, \label{pldi1} \\
&\pr\left(\sup_{u_{2}\in\{u_{2}\in\mbbu_{n}^{2};r\leq\|u_{2}\|\}}\tilde{\mbbz}_{n}^{2}(u_{2};\tz)\geq e^{-r}\right)\leq\frac{C_{L}}{r^{L}} \label{pldi2} 
\end{align}
for all $n>0$ and $r>0$.
In particular, we have
\begin{equation}
\lim_{n\to\infty}
\E\bigg\{f\bigg( \sqrt{n}\bigg(\frac{\tilde{h}}{\tau h_{0}}-1\bigg),\, 
\sqrt{n}(\tilde{\alpha}_{n}-\alpha_{0}),\, \sqrt{n\tilde{h}}(\tilde{\beta}_{n}-\beta_{0}) \bigg)\bigg\} 
=\int f(y)\phi_{p+1}\left( y;\, 0,\Sig(\tz) \right)dy
\label{se:thm.pldi-1}
\end{equation}
for any continuous function $f:\mbbr^{p}\to\mbbr$ at most of polynomial growth, that is, $\limsup_{\|u\|\to\infty}(1+\|u\|^{r})^{-1}|f(u)|<\infty$ for some $r>0$.
Here, $\phi_{p}(\cdot;\, \mu,\Sig)$ denotes the probability density function of $N_{p}(\mu,\Sig)$.
\label{se:thm.pldi}
\end{thm}

The proof of Theorem \ref{se:thm.pldi} is given in Section \ref{se:proof.pldi}.

\begin{rem}{\rm 
Let us mention the PLDI for stepwise estimation case. 
We define the random fields $\tilde{\mbbz}_{n}^{1}$ and $\tilde{\mbbz}_{n}^{2}$ by
\begin{align*}
\tilde{\mbbz}_{n}^{1}(u_{1};\alpha_{0})&=\exp\left\{\tilde{\mbbh}_{1,n}\left(\alpha_{0}+\frac{u_{1}}{\sqrt{n}}\right)-\tilde{\mbbh}_{1,n}\left(\alpha_{0}\right)\right\}, \\
\tilde{\mbbz}_{n}^{2}(u_{2};\tz)&=\exp\left\{\tilde{\mbbh}_{2,n}\left(\alpha_{0},\beta_{0}+\frac{u_{2}}{\sqrt{n\tau h_{0}}}\right)-\tilde{\mbbh}_{2,n}\left(\alpha_{0},\beta_{0}\right)\right\},
\end{align*}
respectively.
Then, we can show similar statements as \eqref{pldi1} and \eqref{pldi2} under the assumptions of Theorem \ref{se:thm.pldi}.
Moreover, \eqref{se:thm.pldi-1} holds with $\tilde{\alpha}_{n}$, $\tilde{\beta}_{n}$, and $\tilde{h}$ replaced by $\tilde{\alpha}_{n}^{\prime}$, $\tilde{\beta}_{n}^{\prime}$, and $\tilde{h}^{\prime}$.
}\qed
\end{rem}

\section{Consistent model selection}
\label{hm:sec_qbic}

We here consider consistent model selection by the (quasi-)Bayesian information criterion ((Q)BIC for short) studied in \cite{EguMas18a};
previously, the correct form of the classical Schwarz's BIC for ergodic diffusion observed at high frequency was given in \cite[Theorems 3.7, 4.5, and 4.6]{EguMas18a} when $h$ is given {\it a priori}.

Let $\Pi(d\theta)$ denote the prior distribution over $\Theta$.

\begin{ass}
The distribution $\Pi$ admits a bounded Lebesgue density $\mfp(\theta)$ which is continuous and positive at $\tz$.
\label{Ass:prior}
\end{ass}

We are regarding $\tilde{\mbbh}_{n}(\theta)$ as our quasi-likelihood, hence it would be natural to define the modified marginal quasi-log likelihood function as
\begin{equation}
\mfL_{n} := \log\left(\int_{\Theta}\exp\{\tilde{\mbbh}_{n}(\theta)\}\mathfrak{p}(\theta)d\theta\right),
\label{hm:mmqlf_def}
\end{equation}
and select a model which maximizes this quantity.
The next theorem shows the precise asymptotic expansion of this quantity up to the order $O_{p}(1)$.

\begin{thm}
Suppose that Assumptions \ref{Ass1}, \ref{Ass2p}, \ref{Ass3} and \ref{Ass:prior} hold.
Then, we have
\begin{align}
\mfL_{n}
&=\tilde{\mbbh}_{n}(\tz)-\frac{1}{2}p_{\alpha}\log n-\frac{1}{2}p_{\beta}\log\left(nh_{0}\right)+\log \mathfrak{p}(\tz)+\frac{p}{2}\log(2\pi) \nn\\
&\quad-\frac{1}{2}\log\big|\tilde{\Gam}_{1,0}\big|-\frac{1}{2}\log\big|\tau\tilde{\Gam}_{2,0}\big|+\frac{1}{2}\diag(\tilde{\Gam}_{1,0},\tau\tilde{\Gam}_{2,0})^{-1}\big[\tilde{\D}_{n}^{\otimes2}\big]+o_{p}(1). \label{se:thm.bic1}
\end{align}
Further, we have
\begin{align}
\mfL_{n}&=\tilde{\mbbh}_{n}(\tilde{\theta}_{n})-\frac{1}{2}p_{\alpha}\log n-\frac{1}{2}p_{\beta}\log\left(nh(\tilde{\alpha}_{n})\right) +\log \mathfrak{p}(\tilde{\theta}_{n})+\frac{p}{2}\log(2\pi) \nn\\
&\quad-\frac{1}{2}\log\left|-\frac{1}{n}\p_{\alpha}^{2}\tilde{\mbbh}_{n}(\tilde{\theta}_{n})\right|-\frac{1}{2}\log\left|-\frac{1}{nh(\tilde{\alpha}_{n})}\p_{\beta}^{2}\tilde{\mbbh}_{n}(\tilde{\theta}_{n})\right|+o_{p}(1). \label{se:thm.bic2}
\end{align}
\label{se:bic}
\end{thm}

The proof given in Section \ref{proof:thm.bic} goes through as in \cite{EguMas18a} under essentially weaker conditions due to Theorem \ref{hm:thm.qbic}.

\medskip

In view of \eqref{se:thm.bic2}, with the conventional multiplication by $-2$ we obtain
\begin{align*}
-2
\mfL_{n}
&=-2\tilde{\mbbh}_{n}(\tilde{\theta}_{n})+p_{\alpha}\log n+p_{\beta}\log\left( nh(\tilde{\alpha}_{n}) \right)+O_{p}(1) \\
&=-2\tilde{\mbbh}_{n}(\tilde{\theta}_{n})+\log\left|-\p_{\alpha}^{2}\tilde{\mbbh}_{n}(\tilde{\theta}_{n})\right|+\log\left|-\p_{\beta}^{2}\tilde{\mbbh}_{n}(\tilde{\theta}_{n})\right|+O_{p}(1).
\end{align*}
Ignoring the $O_{p}(1)$ parts, we define the \textit{modified Bayesian information criterion (mBIC)} and \textit{modified quasi-Bayesian information criterion (mQBIC)} by
\begin{align*}
\mbic=-2\tilde{\mbbh}_{n}(\tilde{\theta}_{n})+p_{\alpha}\log n+p_{\beta}\log\left( nh(\tilde{\alpha}_{n}) \right)
\end{align*}
and
\begin{align*}
\mqbic=-2\tilde{\mbbh}_{n}(\tilde{\theta}_{n})+\log\left|-\p_{\alpha}^{2}\tilde{\mbbh}_{n}(\tilde{\theta}_{n})\right|+\log\left|-\p_{\beta}^{2}\tilde{\mbbh}_{n}(\tilde{\theta}_{n})\right|,
\end{align*}
respectively, both being completely free from $h$.
Since the difference between mBIC and mQBIC is $O_{p}(1)$, we can regard that two criteria are asymptotically equivalent in the sense of BIC type criterion.
As directly seen by the definition, the mBIC has lower computational load than the mQBIC, and the mQBIC enables us to incorporate a model-complexity bias correction taking the observed information into account.

Suppose that candidates for the diffusion and drift coefficients and $\tau$ are given as
\begin{align}
& a_{1}(x,\alpha_{1}),\dots,a_{M_{1}}(x,\alpha_{M_{1}}), \label{hm:ms.a} \\
& b_{1}(x,\alpha_{m_{1}},\beta_{1}),\ldots,b_{M_{2}}(x,\alpha_{m_{1}},\beta_{M_{2}}), \quad m_{1}=1,\dots,M_{1},
\label{hm:ms.b}
\end{align}
where $\theta_{m_{1},m_{2}}=(\alpha_{m_{1}},\beta_{m_{2}})\in\Theta_{\alpha_{m_{1}}}\times\Theta_{\beta_{m_{2}}}\subset\mbbr^{p_{\alpha_{m_{1}}}}\times\mbbr^{p_{\beta_{m_{2}}}}$.
Then, each candidate model $\mcm_{m_{1},m_{2}}$ is given by
\begin{align*}
dX_{t}=\sqrt{\tau_{m_{1},m_{2}}}a_{m_{1}}(X_{t},\alpha_{m_{1}})dw_{t}+\tau_{m_{1},m_{2}}b_{m_{2}}(X_{t},\theta_{m_{1},m_{2}})dt, \quad t\in[0,T_{n}], \quad X_{0}=x_{0}.
\end{align*}
Here $a_{m_{1}}$ is an $\mbbr^{d}\otimes\mbbr^{d}$-valued function defined on $\mbbr^{d}\times\Theta_{\alpha_{m_{1}}}$, $b_{m_{2}}$ is an $\mbbr^{d}$-valued function defined on $\mbbr^{d}\times\Theta_{\alpha_{m_{1}}}\times\Theta_{\beta_{m_{2}}}$, and $\tau_{m_{1},m_{2}}$ is an unknown positive constant.
Write $(\mcm_{m_{1},m_{2}})_{m_{1}\leq M_{1}, m_{2}\leq M_{2}}$ for the set of all candidate models.
For each candidate model $\mcm_{m_{1},m_{2}}$, we assume that 
there exists a value $\theta_{m_{1},m_{2},0}=(\alpha_{m_{1},0},\beta_{m_{2},0}) \in \Theta_{\alpha_{m_{1}}}\times\Theta_{\beta_{m_{2}}}$ for which 
$a_{m_{1}}(\cdot,\alpha_{m_{1},0})$ and $b_{m_{2}}(\cdot,\theta_{m_{1},m_{2},0})$ coincide with the true (data generating) diffusion and drift coefficients, respectively.
We compute mBIC for each candidate model, say $\mbic^{(1,1)},\ldots,\mbic^{(M_{1},M_{2})}$, and then select
the model having the minimum-mBIC value as the best one, say $\mcm_{m_{1,n}^{\ast},m_{2,n}^{\ast}}$:
\begin{align*}
\{(m_{1,n}^{\ast},m_{2,n}^{\ast})\}=\argmin_{(m_{1},m_{2})}\mbic^{(m_{1},m_{2})},
\end{align*}
where
\begin{align*}
\mbic^{(m_{1},m_{2})}&=-2\tilde{\mbbh}_{n}^{(m_{1},m_{2})}(\tilde{\theta}_{m_{1},m_{2},n})+p_{\alpha_{m_{1}}}\log n+p_{\beta_{m_{2}}}\log \left( nh_{m_{1}}(\tilde{\alpha}_{m_{1},n}) \right), 
\end{align*}
with $\tilde{\theta}_{m_{1},m_{2},n}$ denoting the mGQMLE associated with the mGQLF $\tilde{\mbbh}_{n}^{(m_{1},m_{2})}$ of (\ref{hm:mGQMLE.def}) associated with the model $\mcm_{m_{1},m_{2}}$.
The selection rule when using the mQBIC is given in a similar manner.

\medskip

It is worth mentioning that a \textit{two-step model selection} is possible as in \cite[Section5.2]{EguMas18a}.
We proceed as follows.

\begin{itemize}
\item First, we select the best diffusion coefficient $a_{m_{1,n}^{\ast}}$ among \eqref{hm:ms.a}, where $m_{1,n}^{\ast}$ satisfies $\{m_{1,n}^{\ast}\}=\argmin_{m_{1}}\mbic^{(m_{1})}$ with
\begin{align*}
\mbic^{(m_{1})}=-2\tilde{\mbbh}_{1,n}^{(m_{1})}(\tilde{\alpha}_{m_{1},n}^{\prime})+p_{\alpha_{m_{1}}}\log n,
\end{align*}
$\tilde{\alpha}_{m_{1},n}^{\prime}\in\argmax_{\alpha_{m_{1}}}\tilde{\mbbh}_{1,n}^{(m_{1})}(\alpha_{m_{1}})$, and $\tilde{\mbbh}_{1,n}^{(m_{1})}$ corresponds to
\eqref{hm:def_H1n_al} with ignoring the drift. 
\item Next, among \eqref{hm:ms.b} for $m_{1}=m_{1,n}^{\ast}$, we select the best drift coefficient with index $m_{2,n}^{\ast}$ such that $\{m_{2,n}^{\ast}\}=\argmin_{m_{2}}\mbic^{(m_{2}|m_{1,n}^{\ast})}$, where
\begin{align*}
\mbic^{(m_{2}|m_{1,n}^{\ast})}=-2\tilde{\mbbh}_{2,n}^{(m_{2}|m_{1,n}^{\ast})}(\tilde{\alpha}_{m_{1,n}^{\ast},n}^{\prime},\tilde{\beta}_{m_{2},n}^{\prime})
+p_{\beta_{m_{2}}}\log\left( nh_{m_{1,n}^{\ast}}(\tilde{\alpha}_{m_{1,n}^{\ast},n}^{\prime})\right),
\end{align*}
$\tilde{\beta}_{m_{2},n}^{\prime}\in\argmax_{\beta_{m_{2}}}\tilde{\mbbh}_{2,n}^{(m_{2}|m_{1,n}^{\ast})}(\tilde{\alpha}_{m_{1,n}^{\ast},n}^{\prime},\beta_{m_{2}})$, and $\tilde{\mbbh}_{2,n}^{(m_{2}|m_{1})}$ corresponds to \eqref{se:def_H2n_be} with the previously selected diffusion coefficient plugged-in. 
\item Finally, we select the model $\mcm_{m_{1,n}^{\ast},m_{2,n}^{\ast}}$ as the final best model among the candidates described by \eqref{hm:ms.a} and \eqref{hm:ms.b}.
\end{itemize} 
We can apply this procedure to the mQBIC as well. 
The total number of candidate models in the joint and two-step model selections are $M_{1}\times M_{2}$ and $M_{1}+M_{2}$, respectively.
This indicates that difference between computational costs for the joint and two-step selection procedures becomes more significant when $M_{1}(\geq2)$ or $M_{2}(\geq2)$ (or both) is large.

\medskip


We assume that the model indexes $m_{1,0}$ and $m_{2,0}$ are uniquely determined, that is
\begin{align*}
\{m_{1,0}\}&=\argmin_{m_{1}}\dim(\Theta_{m_{1}}), \\
\{m_{2,0}\}&=\argmin_{m_{2}}\dim(\Theta_{m_{2}}),
\end{align*}
respectively. 
Then, we say that $\mcm_{m_{1,0},m_{2,0}}$ is the optimal model. 
The following theorem ensures that the probability that the true model $\mcm_{m_{1,0},m_{2,0}}$ is selected by using m(Q)BIC tends to 1 as $n\to\infty$.

\begin{thm}
Suppose that Assumptions \ref{Ass1}, \ref{Ass2p}, \ref{Ass3} and \ref{Ass:prior} hold for the all candidate models $\mcm_{m_{1},m_{2}}$.
Then, the joint and two-step model-selection consistencies hold in the following senses.
\begin{enumerate}
\item Suppose that at least one of $m_{1}$ and $m_{2}$ differs from $m_{1,0}$ and $m_{2,0}$, respectively.
Then, we have
\begin{align*}
\lim_{n\to\infty}\mathbb{P}\left(\mbic^{(m_{1,0},m_{2,0})}-\mbic^{(m_{1},m_{2})}<0\right)&=1,
\end{align*} 
and the same statement holds with ``$\mathrm{mBIC}$" replaced by ``$\mathrm{mQBIC}$".
\item For each $(m_{1},m_{2})\in(\{1,\ldots,M_{1}\}\backslash\{m_{1,0}\})\times(\{1,\ldots,M_{2}\}\backslash\{m_{2,0}\})$, we have
\begin{align*}
& \lim_{n\to\infty}\mathbb{P}\left(\mbic^{(m_{1,0})}-\mbic^{(m_{1})}<0\right) =1, \\
& \lim_{n\to\infty}\mathbb{P}\left(\mbic^{(m_{2,0}|m_{1,n}^{\ast})}-\mbic^{(m_{2}|m_{1,n}^{\ast})}<0\right)=1,
\end{align*} 
and the same statements hold with ``$\mathrm{mBIC}$" replaced by ``$\mathrm{mQBIC}$".
\end{enumerate}
\label{se:model.consis}
\end{thm}

%


\begin{rem}{\rm 
When the model is misspecified in the sense that parametric specification of the coefficients is wrong, asymptotic property of estimators can essentially differ from the correctly specified case (see \cite{UchYos11}). In order to guarantee use of the BIC type criteria we need a suitable stochastic expansion of the associated marginal quasi log-likelihood function, which has not been explored in the literature as yet; once the stochastic expansion is derived, then it will be possible to deduce the model selection consistency in the same way as in \cite[Theorem 5.1]{EguMas18a}. We would like to leave this important issue as a future work.
}\qed
\end{rem}

\section{Simulation experiments} \label{se:simulation}

In this section, we present simulation results to evaluate finite sample performance of our estimation procedure.
We use the R package YUIMA \cite{YUIMA14} for generating data.
We set $d=1$ in the examples below, and all the Monte Carlo trials are based on 1000 independent sample paths.
Suppose that we have a sample $\mathbf{X}_{n}=(X_{t_{j}})_{j=0}^{n}$ with $t_{j}=jn^{-2/3}$ (hence $T_{n}=n^{1/3}$) from the true model ($\tau=1$):
\begin{align}
dX_{t}&=\exp \biggl\{\frac{1}{2}(2\sin X_{t}-\cos X_{t}\sin X_{t}) \biggr\}dw_{t}-X_{t}dt,\quad t\in[0,T_{n}], \quad X_{0}=1.
\label{se:simu.true}
\end{align}
The simulations are done for $n=1000,3000$, and $5000$.

\subsection{Parameter estimation} \label{Simu1}

We consider the diffusion process as the target of estimation:
\begin{align*}
dX_{t}=\exp\left\{\frac{1}{2}(\alpha_{1}\cos X_{t}+\alpha_{2}\sin X_{t}+\alpha_{3}\cos X_{t}\sin X_{t})\right\}dw_{t}+(\beta_{1}X_{t}+\beta_{2})dt.
\end{align*}
We set the true parameter values as $\theta_{0}=(\alpha_{0},\beta_{0})=(\alpha_{1,0},\alpha_{2,0},\alpha_{3,0},\beta_{1,0},\beta_{2,0})=(0,2,-1,-1,0)$.
It is easy to check that Assumptions \ref{Ass3} and \ref{Ass1+} hold, in particular,
\begin{align}
S^{-1}(x,\alpha)S(x,\alpha_{0})&=\exp\big\{(\alpha_{1,0}-\alpha_{1})\cos x+(\alpha_{2,0}-\alpha_{2})\sin x \nn\\
&{}\qquad +(\alpha_{3,0}-\alpha_{3})\cos x\sin x\big\}, \nn\\
S^{-1}(x,\alpha_{0})\p_{\alpha}S(x,\alpha_{0})&=(\cos x, \sin x)^{\top}.
\label{check.ass}
\end{align}
We computed the estimator $\tet=(\tilde{\alpha}_{n},\tilde{\beta}_{n})$ through both the proposed method (two-step and joint), and also the quasi-maximum likelihood estimators (QMLEs)
$\tes=(\aes,\bes)$ associated with (\ref{ErLf}) with using the true sampling rate $h_{0}=n^{-2/3}$.
For numerical optimization, we set the initial values of $\alpha_{1}$, $\alpha_{2}$, and $\alpha_{3}$ to be random numbers generated from uniform distribution $U(-1,1)$. 
Moreover, the initial values of $\beta_{1}$ and $\beta_{2}$ are generated from uniform distribution $U(-2,0)$.

Table \ref{esti1} summarizes the mean and standard deviation of the estimators. On this example, as is expected from our theoretical results,
we can observe in terms of the standard deviations that performance of $\tet$ is overall inferior to the known-$h$ case and that their performances tend to get close each other as $n$ increases.
Further, it is worth noting that the performances of estimating $h=h_{0}$ are equally good for the joint and two-step cases.

Here we have (recall \eqref{hm:thm2.diffusion.an-1})
\begin{align*}
\bar{u}_{n}^{\al}&=(\bar{u}_{1,n}^{\al},\bar{u}_{2,n}^{\al},\bar{u}_{3,n}^{\al})=\left(-\frac{1}{n}\p_{\alpha}^{2}\tilde{\mbbh}_{n}(\tet)\right)^{\frac{1}{2}} \sqrt{n}(\tilde{\alpha}_{n}-\alpha_{0}), \\
\bar{u}_{n}^{\beta}&=(\bar{u}_{1,n}^{\beta},\bar{u}_{2,n}^{\beta})=\left(-\frac{1}{n\tilde{h}}\p_{\beta}^{2}\tilde{\mbbh}_{n}(\tet)\right)^{\frac{1}{2}} \sqrt{n\tilde{h}}(\tilde{\beta}_{n}-\beta_{0}),
\end{align*}
respectively.
Figures \ref{hist1} and \ref{hist2} show the histograms of $\bar{u}_{n}^{\al}$ and $\bar{u}_{n}^{\beta}$ in the case of $n=5000$,
each corresponding to the results of the two-step and the joint estimations.

The residuals are given by
\begin{align*}
\tilde{\epsilon}_{j}=\frac{\D_{j}X-\tilde{h}(\tilde{\beta}_{1,n}X_{t_{j-1}}+\tilde{\beta}_{2,n})}{\sqrt{\tilde{h}}\exp\left\{\frac{1}{2}(\tilde{\alpha}_{1,n}\cos X_{t_{j-1}}+\tilde{\alpha}_{2,n}\sin X_{t_{j-1}}+\tilde{\alpha}_{3,n}\cos X_{t_{j-1}}\sin X_{t_{j-1}})\right\}}, \quad j=1,2,\ldots,n,
\end{align*}
which are expected to form an i.i.d. standard normal random variables.
Figure \ref{hist3} shows the histogram of $\tilde{\epsilon}=(\tilde{\epsilon}_{1},\ldots,\tilde{\epsilon}_{n})$ for $n=5000$, based on the 1000th sample data and results of estimation,
from which we can observe good performance of the standard-normal approximation.

\begin{table}[t]
\begin{center}
\caption{\footnotesize The mean and the standard deviation (s.d.) of the estimators (the true parameter $\theta_{0}=(0,2,-1,-1,0)$).}
\begin{tabular}{l r r r r r r} \hline
$n=1000$ & \multicolumn{6}{l}{} \\
\multicolumn{7}{c}{} \\[-8pt]
two-step & $\tilde{\alpha}_{1,n}$ & $\tilde{\alpha}_{2,n}$ & $\tilde{\alpha}_{3,n}$ & $\tilde{\beta}_{1,n}$ & $\tilde{\beta}_{2,n}$ & $\tilde{h}/h_{0}$ \\ \hline
mean & -0.0715 & 1.8710 & -0.8460 & -1.4484 & -0.0940 & 1.1068  \\ 
s.d. & 0.3079 & 0.3684 & 0.4766 & 0.8892 & 0.5061 & 0.3266 \\ \hline\hline
\multicolumn{7}{c}{} \\[-8pt]
joint & $\tilde{\alpha}_{1,n}$ & $\tilde{\alpha}_{2,n}$ & $\tilde{\alpha}_{3,n}$ & $\tilde{\beta}_{1,n}$ & $\tilde{\beta}_{2,n}$ & $\tilde{h}/h_{0}$ \\ \hline
mean & 0.0329 & 1.9367 & -0.9229 & -1.6618 & -0.1264 & 1.0102 \\ 
s.d. & 0.3217 & 0.3888 & 0.4973 & 1.2128 & 0.6238 & 0.2822 \\ \hline\hline
\multicolumn{7}{c}{} \\[-8pt]
QMLEs & $\tilde{\alpha}_{1,n}$ & $\tilde{\alpha}_{2,n}$ & $\tilde{\alpha}_{3,n}$ & $\tilde{\beta}_{1,n}$ & $\tilde{\beta}_{2,n}$ & $\tilde{h}/h_{0}$ \\ \hline
mean & 0.0057 & 1.9209 & -0.9380 & -1.4937 & -0.2236 & -- \\ 
s.d. & 0.0620 & 0.2806 & 0.3830 & 0.6064 & 0.3058 & -- \\ \hline\hline
$n=3000$ & \multicolumn{6}{l}{} \\
\multicolumn{7}{c}{} \\[-8pt]
two-step & $\tilde{\alpha}_{1,n}$ & $\tilde{\alpha}_{2,n}$ & $\tilde{\alpha}_{3,n}$ & $\tilde{\beta}_{1,n}$ & $\tilde{\beta}_{2,n}$ & $\tilde{h}/h_{0}$ \\ \hline
mean & -0.0314 & 1.9474 & -0.9356 & -1.2864 & -0.0619 & 1.0374  \\ 
s.d. & 0.1510 & 0.1826 & 0.2447 & 0.5429 & 0.3517 & 0.1469 \\ \hline\hline
\multicolumn{7}{c}{} \\[-8pt]
joint & $\tilde{\alpha}_{1,n}$ & $\tilde{\alpha}_{2,n}$ & $\tilde{\alpha}_{3,n}$ & $\tilde{\beta}_{1,n}$ & $\tilde{\beta}_{2,n}$ & $\tilde{h}/h_{0}$ \\ \hline
mean & 0.0085 & 1.9786 & -0.9724 & -1.3539 & -0.0799 & 1.0025 \\ 
s.d. & 0.1525 & 0.1895 & 0.2529 & 0.6031 & 0.3842 & 0.1431 \\ \hline\hline
\multicolumn{7}{c}{} \\[-8pt]
QMLEs & $\tilde{\alpha}_{1,n}$ & $\tilde{\alpha}_{2,n}$ & $\tilde{\alpha}_{3,n}$ & $\tilde{\beta}_{1,n}$ & $\tilde{\beta}_{2,n}$ & $\tilde{h}/h_{0}$ \\ \hline
mean & 0.0036 & 1.9745 & -0.9832 & -1.3357 & -0.1633 & -- \\ 
s.d. & 0.0354 & 0.1340 & 0.1877 & 0.4892 & 0.2418 & -- \\ \hline\hline
$n=5000$ & \multicolumn{6}{l}{} \\
\multicolumn{7}{c}{} \\[-8pt]
two-step & $\tilde{\alpha}_{1,n}$ & $\tilde{\alpha}_{2,n}$ & $\tilde{\alpha}_{3,n}$ & $\tilde{\beta}_{1,n}$ & $\tilde{\beta}_{2,n}$ & $\tilde{h}/h_{0}$ \\ \hline
mean & -0.0207 & 1.9644 & -0.9543 & -1.2535 & -0.0625 & 1.0229  \\ 
s.d. & 0.1037 & 0.1348 & 0.1815 & 0.4488 & 0.3044 & 0.0954 \\ \hline\hline
\multicolumn{7}{c}{} \\[-8pt]
joint & $\tilde{\alpha}_{1,n}$ & $\tilde{\alpha}_{2,n}$ & $\tilde{\alpha}_{3,n}$ & $\tilde{\beta}_{1,n}$ & $\tilde{\beta}_{2,n}$ & $\tilde{h}/h_{0}$ \\ \hline
mean & 0.0103 & 1.9858 & -0.9813 & -1.3324 & -0.0939 & 0.9968 \\ 
s.d. & 0.1131 & 0.1476 & 0.1970 & 0.5655 & 0.3606 & 0.0998 \\ \hline\hline
\multicolumn{7}{c}{} \\[-8pt]
QMLEs & $\tilde{\alpha}_{1,n}$ & $\tilde{\alpha}_{2,n}$ & $\tilde{\alpha}_{3,n}$ & $\tilde{\beta}_{1,n}$ & $\tilde{\beta}_{2,n}$ & $\tilde{h}/h_{0}$ \\ \hline
mean & 0.0031 & 1.9826 & -0.9866 & -1.2947 & -0.1495 & -- \\ 
s.d. & 0.0259 & 0.1105 & 0.1488 & 0.4268 & 0.2129 & -- \\ \hline
\end{tabular}
\label{esti1}
\end{center}
\end{table}

\begin{figure}[t]
\caption{Results of estimating $\alpha_{1}$(top left), $\alpha_{2}$(top right), $\alpha_{3}$(center left), $\beta_{1}$(center right), and $\beta_{2}$(bottom) ((i): two-step, (ii): joint, (iii) QMLE). The red line in each figure indicates the true value.}
\begin{tabular}{c}
\ \\

\begin{minipage}{0.5 \hsize}
\begin{center}
\includegraphics[scale=0.49]{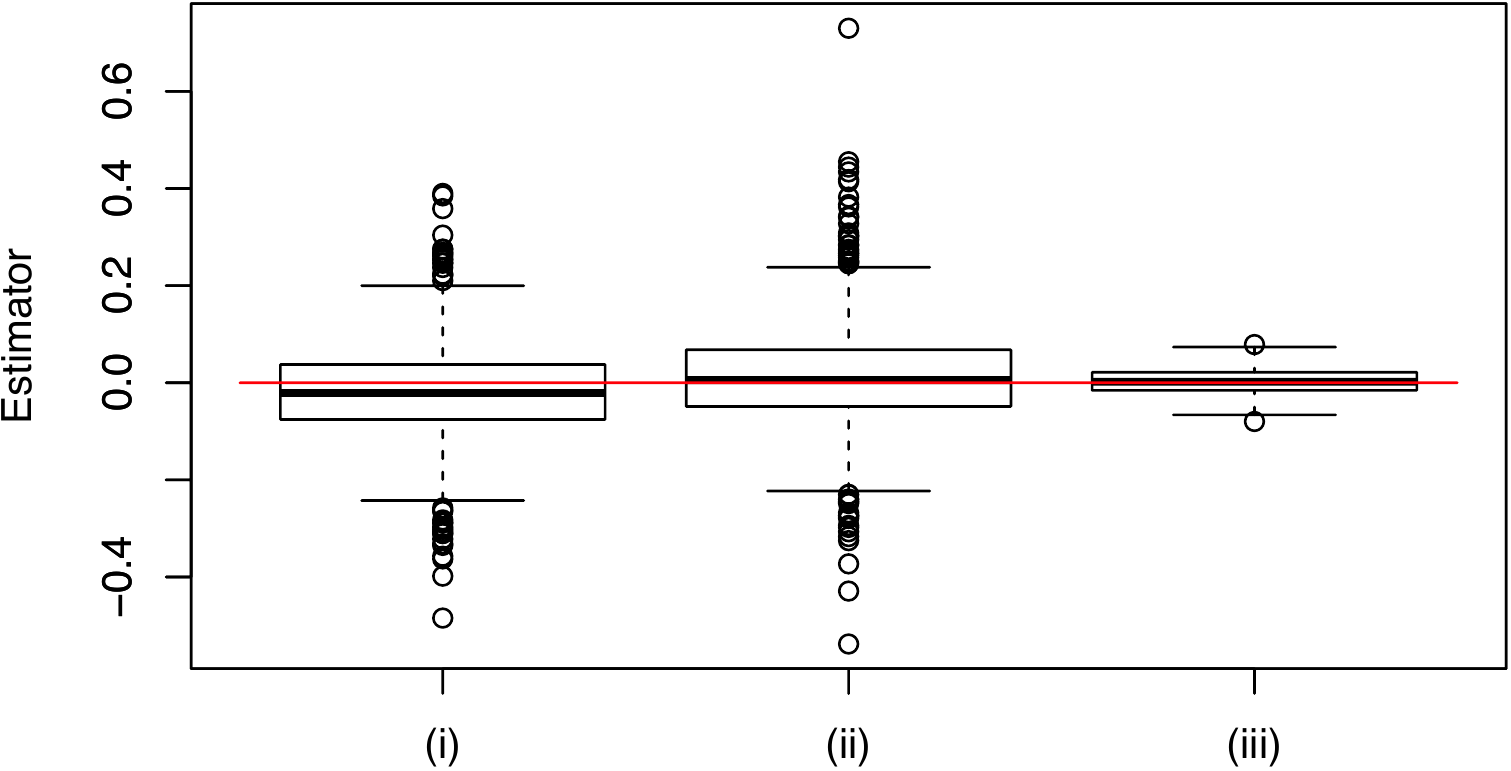}
\end{center}
\end{minipage}

\begin{minipage}{0.5 \hsize}
\begin{center}
\includegraphics[scale=0.49]{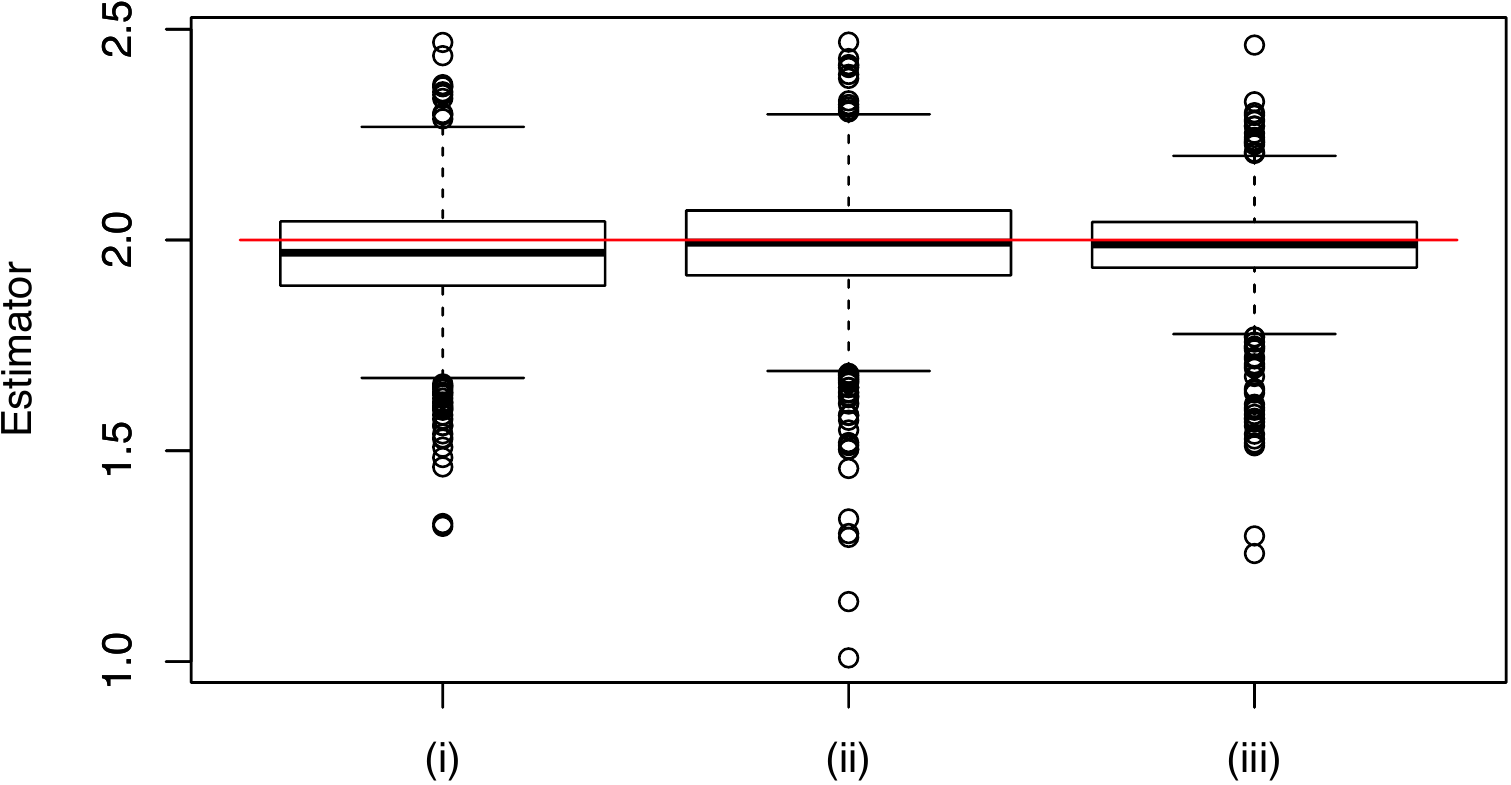}
\end{center}
\end{minipage}

\bigskip \\

\begin{minipage}{0.5 \hsize}
\begin{center}
\includegraphics[scale=0.49]{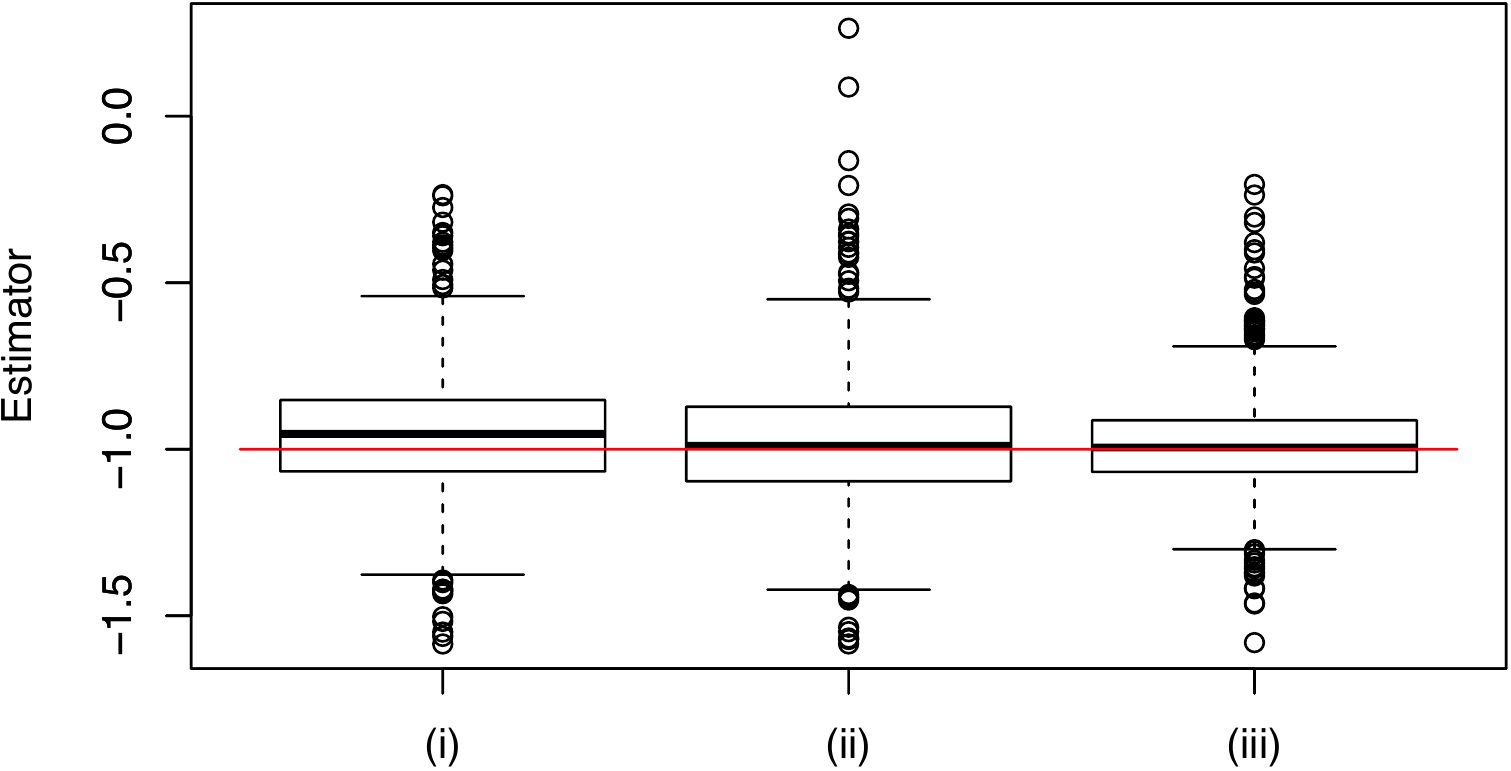}
\end{center}
\end{minipage}

\begin{minipage}{0.5 \hsize}
\begin{center}
\includegraphics[scale=0.49]{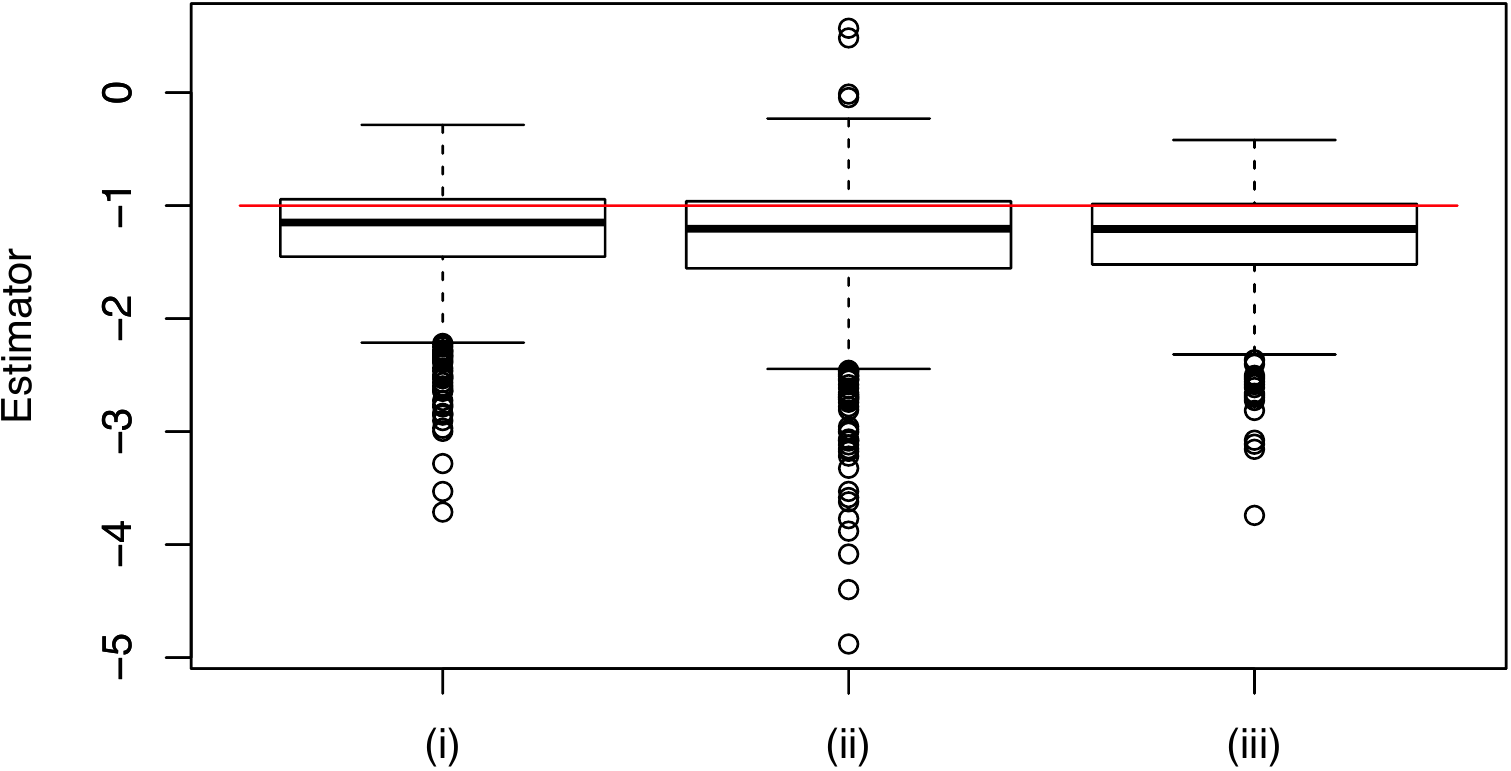}
\end{center}
\end{minipage}

\bigskip \\

\begin{minipage}{0.5 \hsize}
\begin{center}
\includegraphics[scale=0.49]{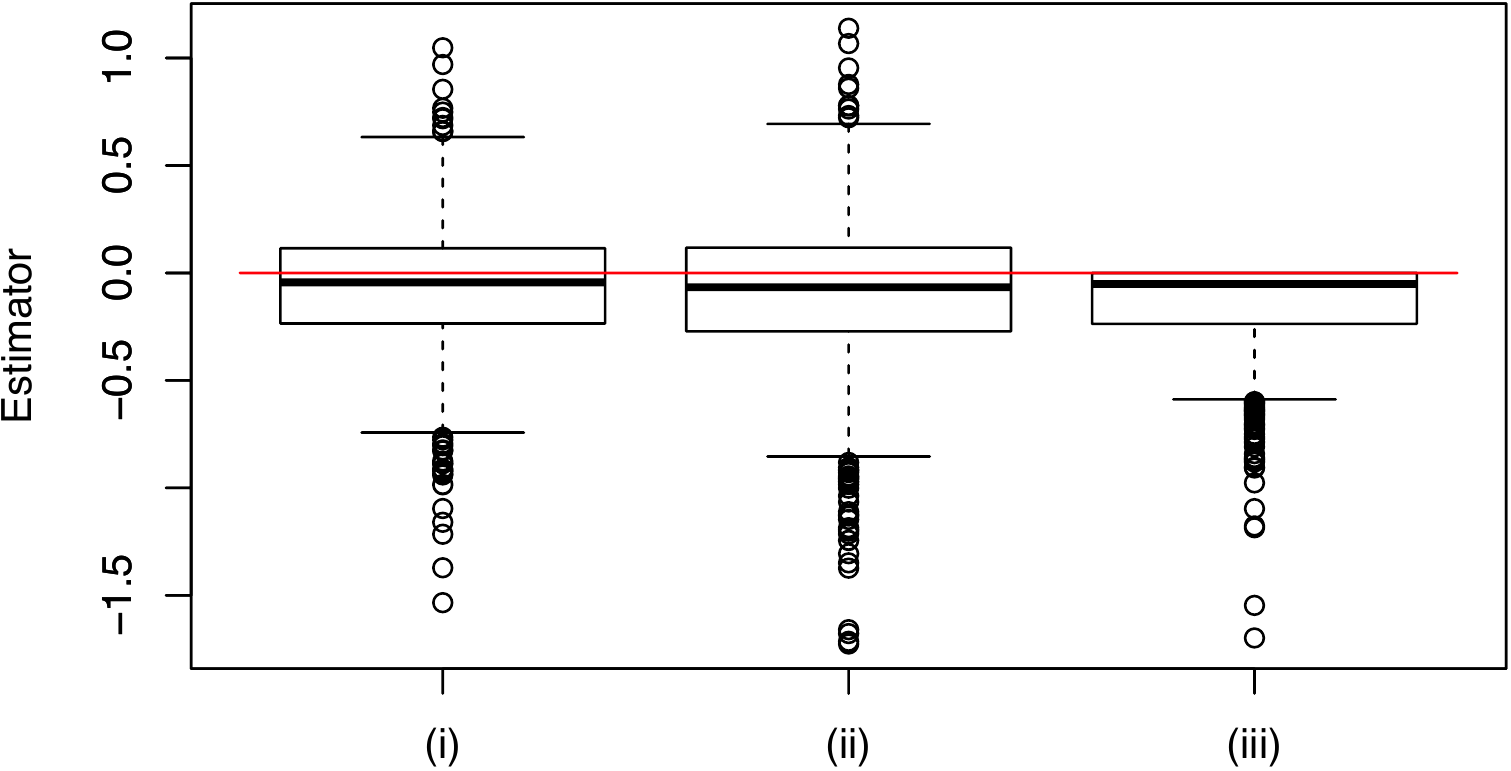}
\end{center}
\end{minipage}

\end{tabular}
\label{boxplot}
\end{figure}


\begin{figure}[h]
\caption{Histograms of $\bar{u}^{\al}_{n}$ and $\bar{u}^{\beta}_{n}$ based on the two-step estimation
(top left: $\bar{u}^{\al}_{1,n}$, top right: $\bar{u}^{\al}_{2,n}$, center left: $\bar{u}^{\al}_{3,n}$, center right: $\bar{u}^{\beta}_{1,n}$, bottom: $\bar{u}^{\beta}_{2,n}$).}
\begin{tabular}{c}
\ \\

\begin{minipage}{0.5 \hsize}
\begin{center}
\includegraphics[scale=0.49]{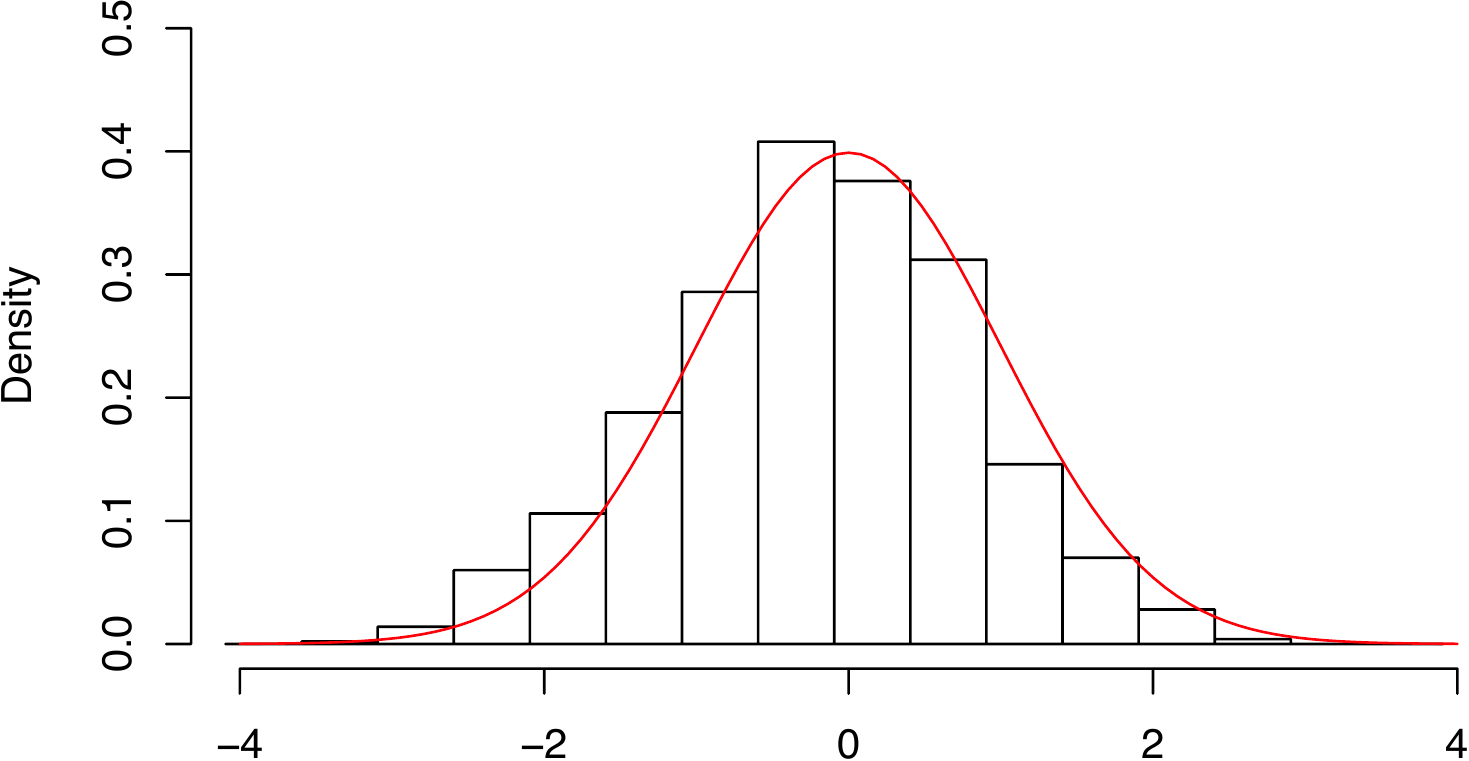}
\end{center}
\end{minipage}

\begin{minipage}{0.5 \hsize}
\begin{center}
\includegraphics[scale=0.49]{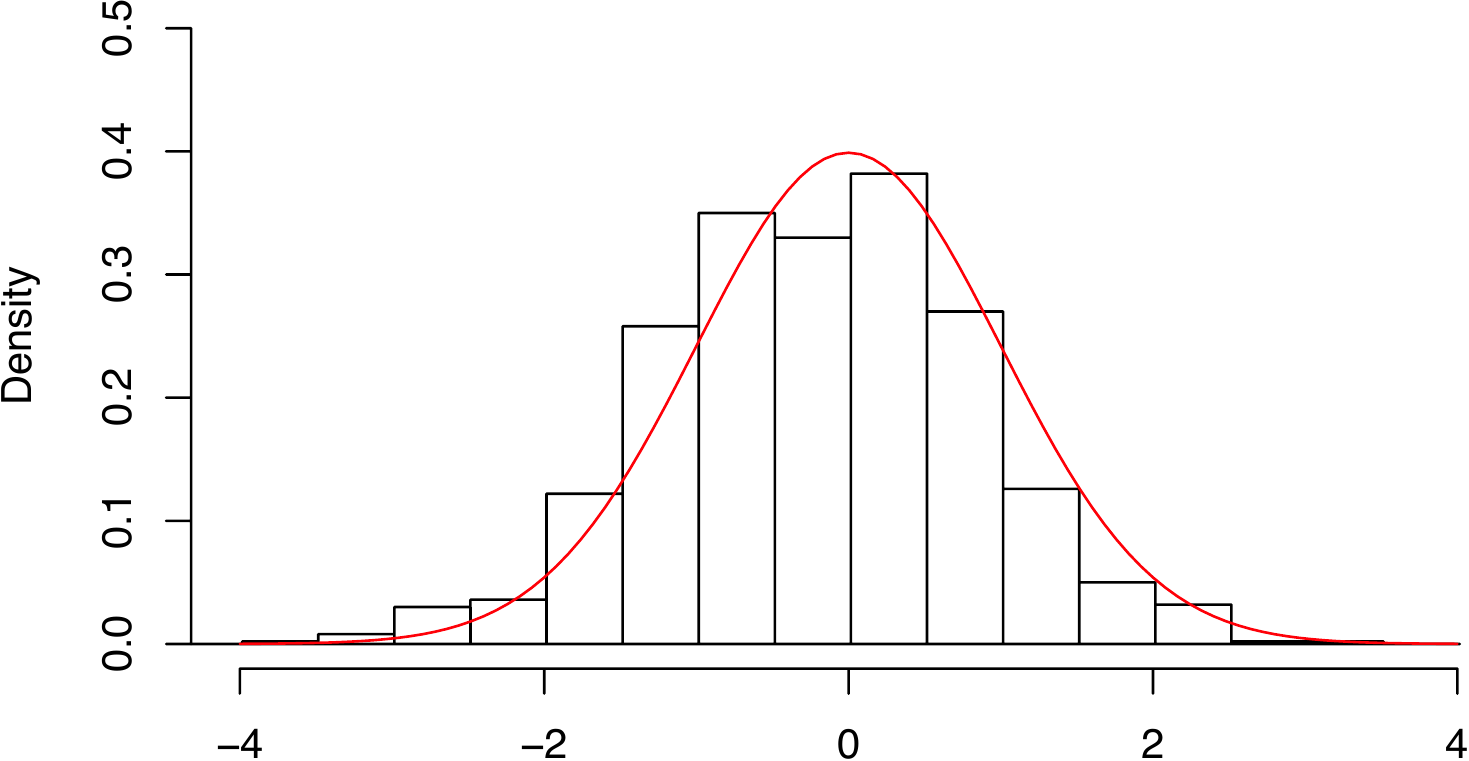}
\end{center}
\end{minipage}

\bigskip \\

\begin{minipage}{0.5 \hsize}
\begin{center}
\includegraphics[scale=0.49]{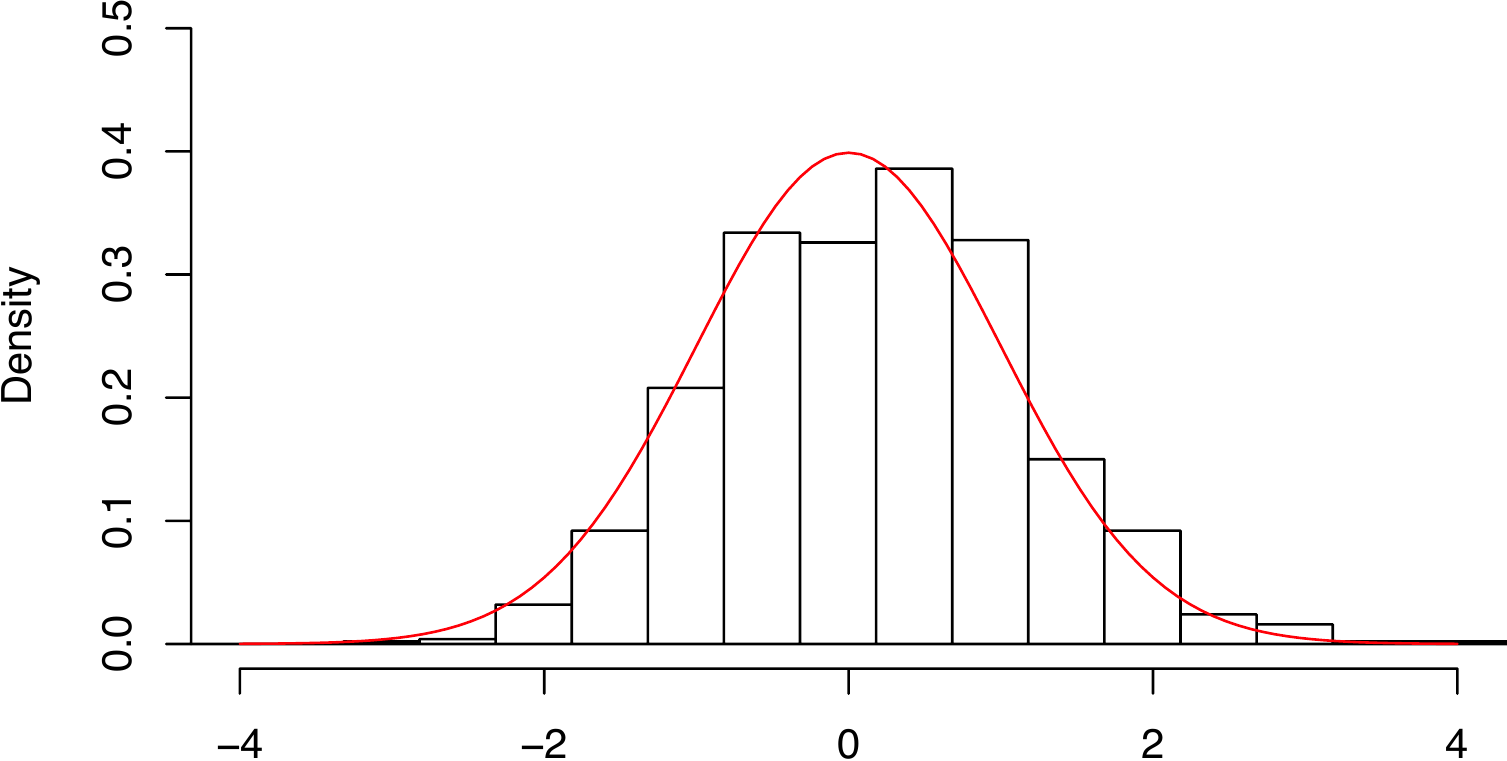}
\end{center}
\end{minipage}

\begin{minipage}{0.5 \hsize}
\begin{center}
\includegraphics[scale=0.49]{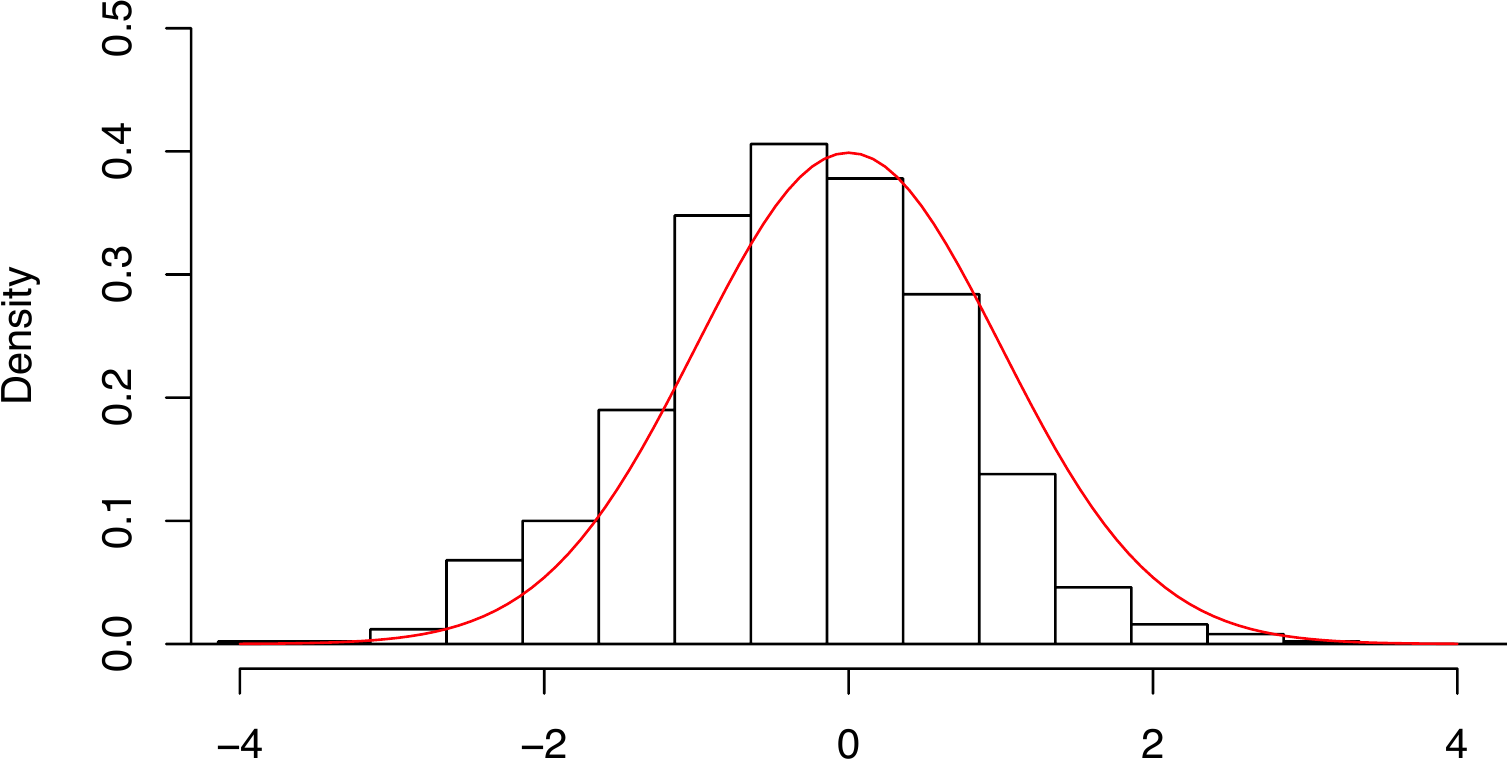}
\end{center}
\end{minipage}

\bigskip \\

\begin{minipage}{0.5 \hsize}
\begin{center}
\includegraphics[scale=0.49]{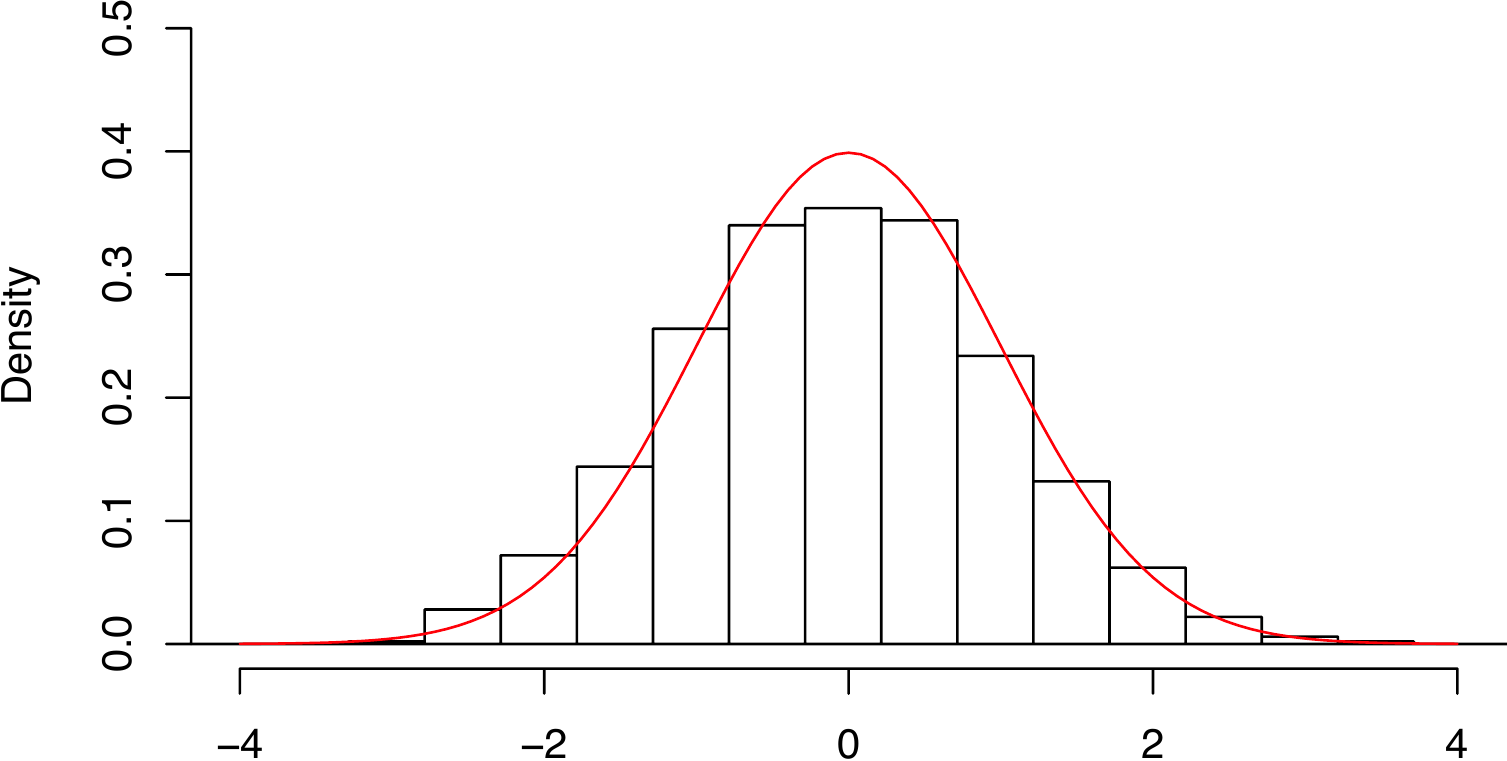}
\end{center}
\end{minipage}

\end{tabular}
\label{hist1}
\end{figure}


\begin{figure}[t]
\caption{Histograms of $\bar{u}^{\al}_{n}$ and $\bar{u}^{\beta}_{n}$ based on the joint estimation
(top left: $\bar{u}^{\al}_{1,n}$, top right: $\bar{u}^{\al}_{2,n}$, center left: $\bar{u}^{\al}_{3,n}$, center right: $\bar{u}^{\beta}_{1,n}$, bottom: $\bar{u}^{\beta}_{2,n}$).}
\begin{tabular}{c}
\ \\

\begin{minipage}{0.5 \hsize}
\begin{center}
\includegraphics[scale=0.49]{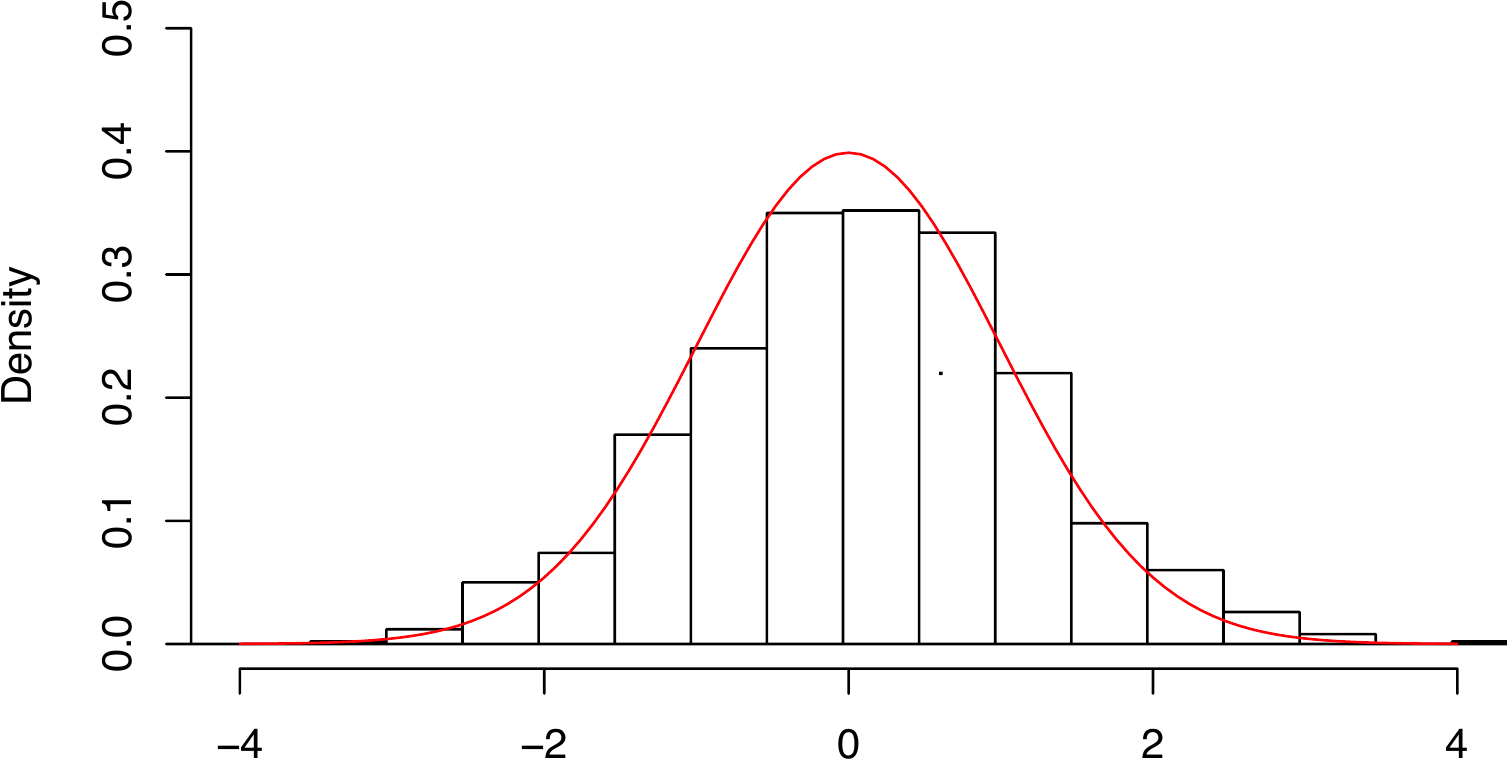}
\end{center}
\end{minipage}

\begin{minipage}{0.5 \hsize}
\begin{center}
\includegraphics[scale=0.49]{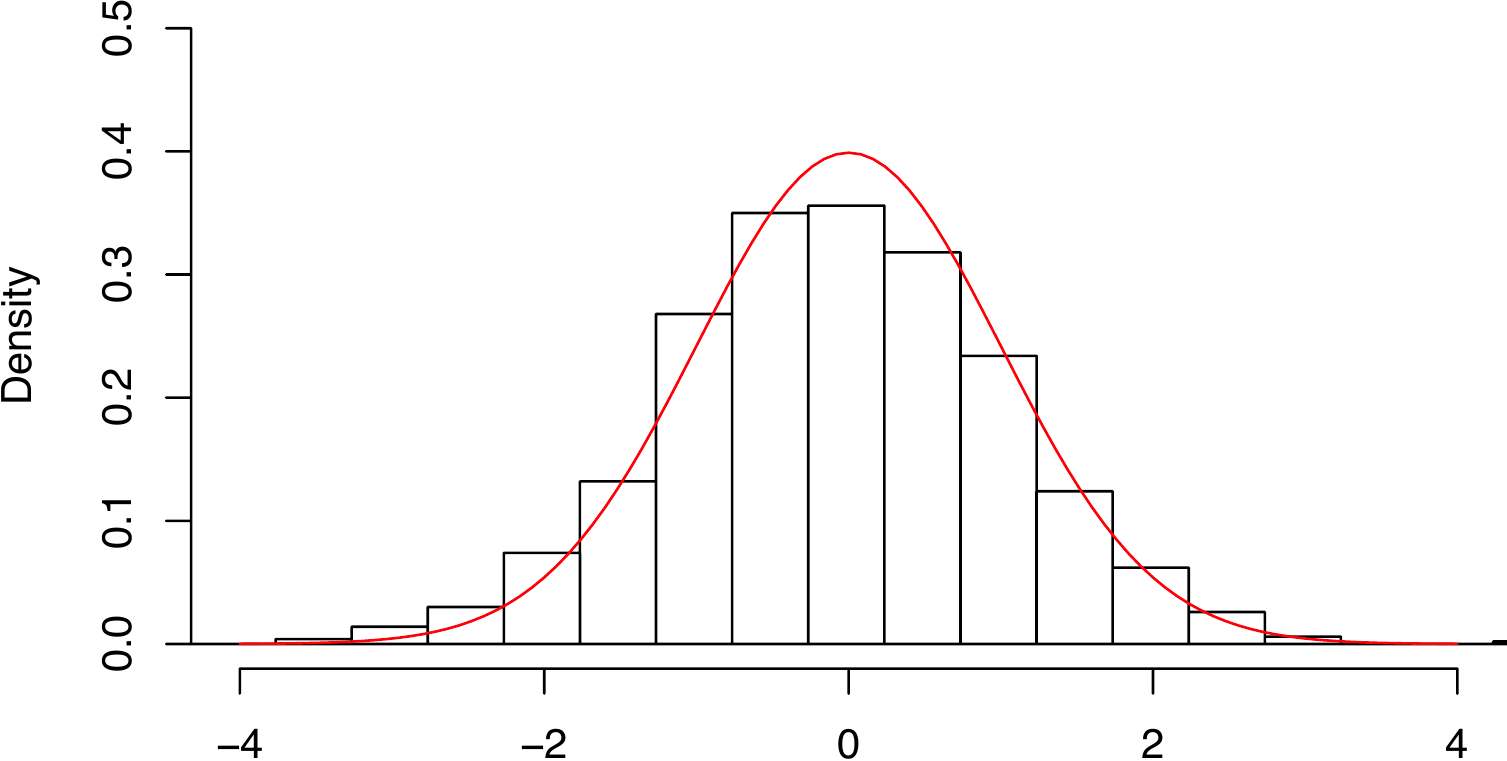}
\end{center}
\end{minipage}

\bigskip \\

\begin{minipage}{0.5 \hsize}
\begin{center}
\includegraphics[scale=0.49]{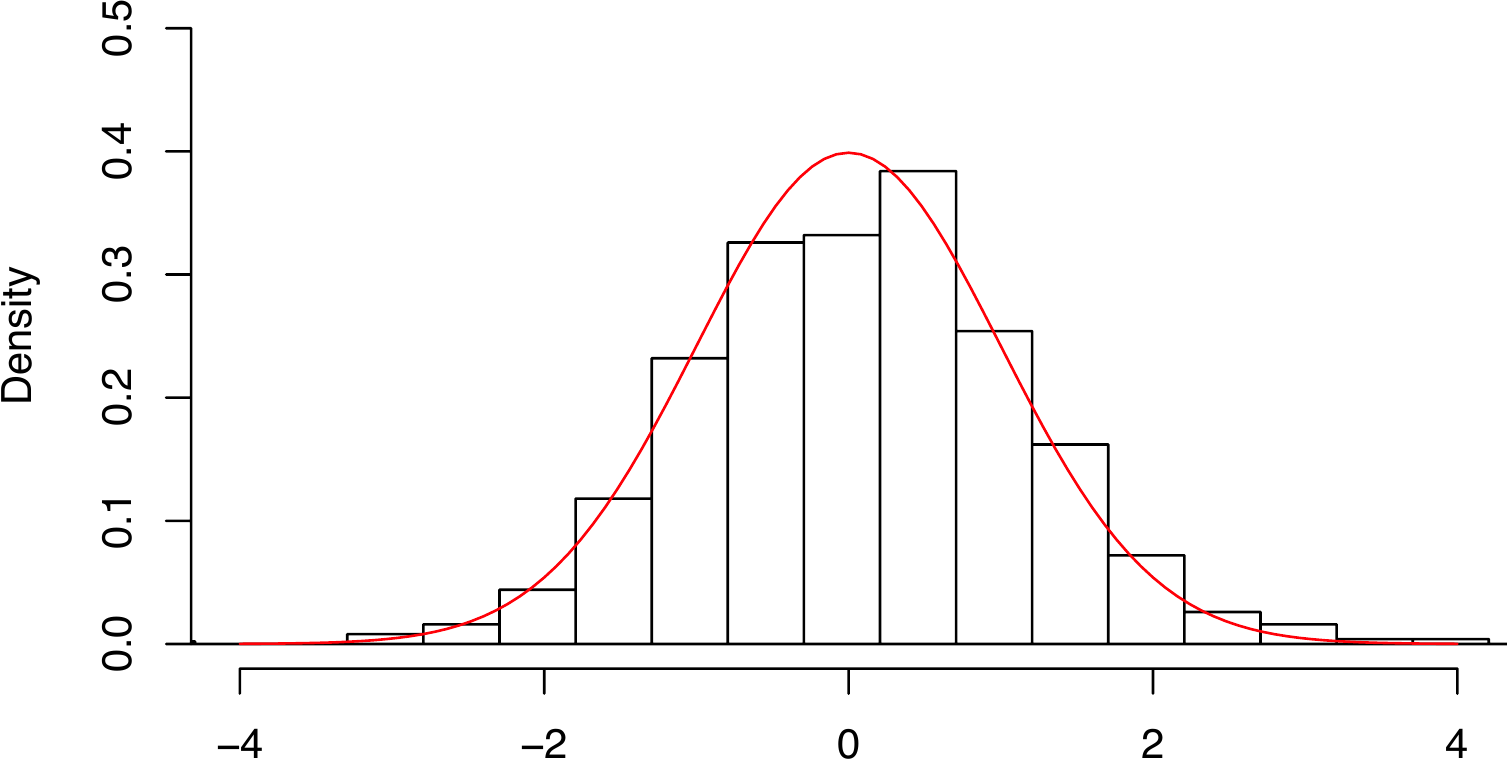}
\end{center}
\end{minipage}

\begin{minipage}{0.5 \hsize}
\begin{center}
\includegraphics[scale=0.49]{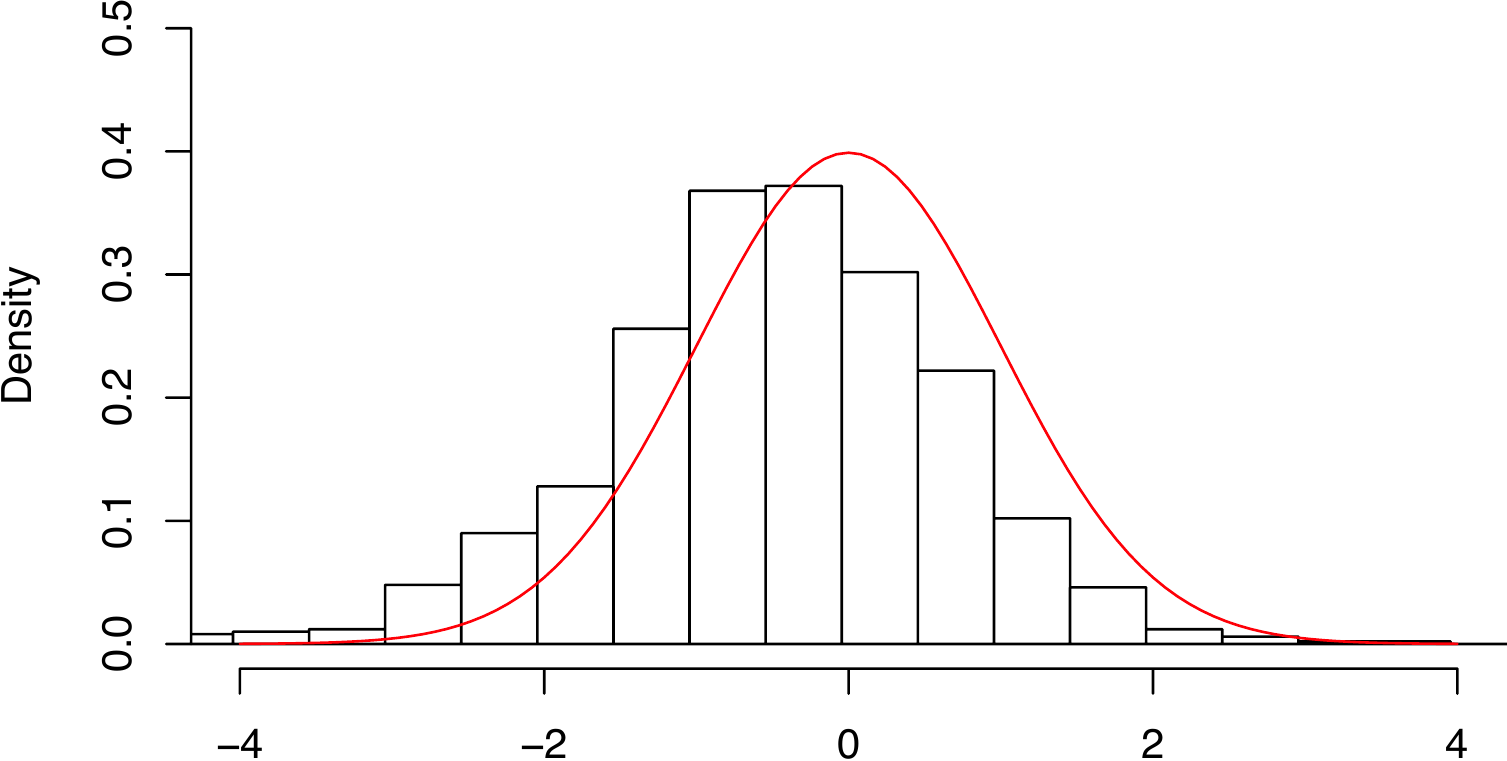}
\end{center}
\end{minipage}

\bigskip \\

\begin{minipage}{0.5 \hsize}
\begin{center}
\includegraphics[scale=0.49]{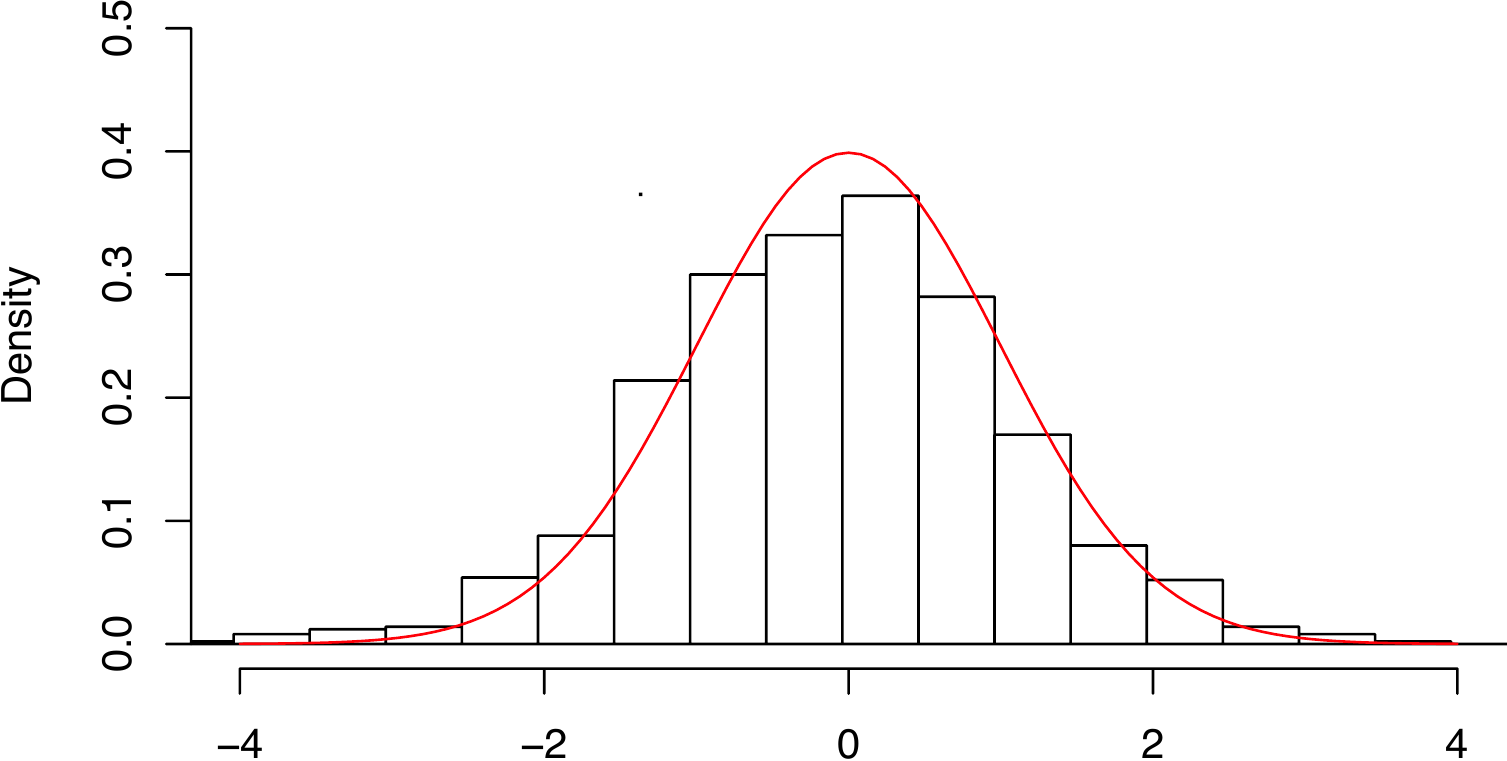}
\end{center}
\end{minipage}

\end{tabular}
\label{hist2}
\end{figure}


\begin{figure}[t]
\caption{Histograms of $\tilde{\epsilon}$ (left: two-step, right: joint).}
\begin{tabular}{c}
\ \\

\begin{minipage}{0.5 \hsize}
\begin{center}
\includegraphics[scale=0.49]{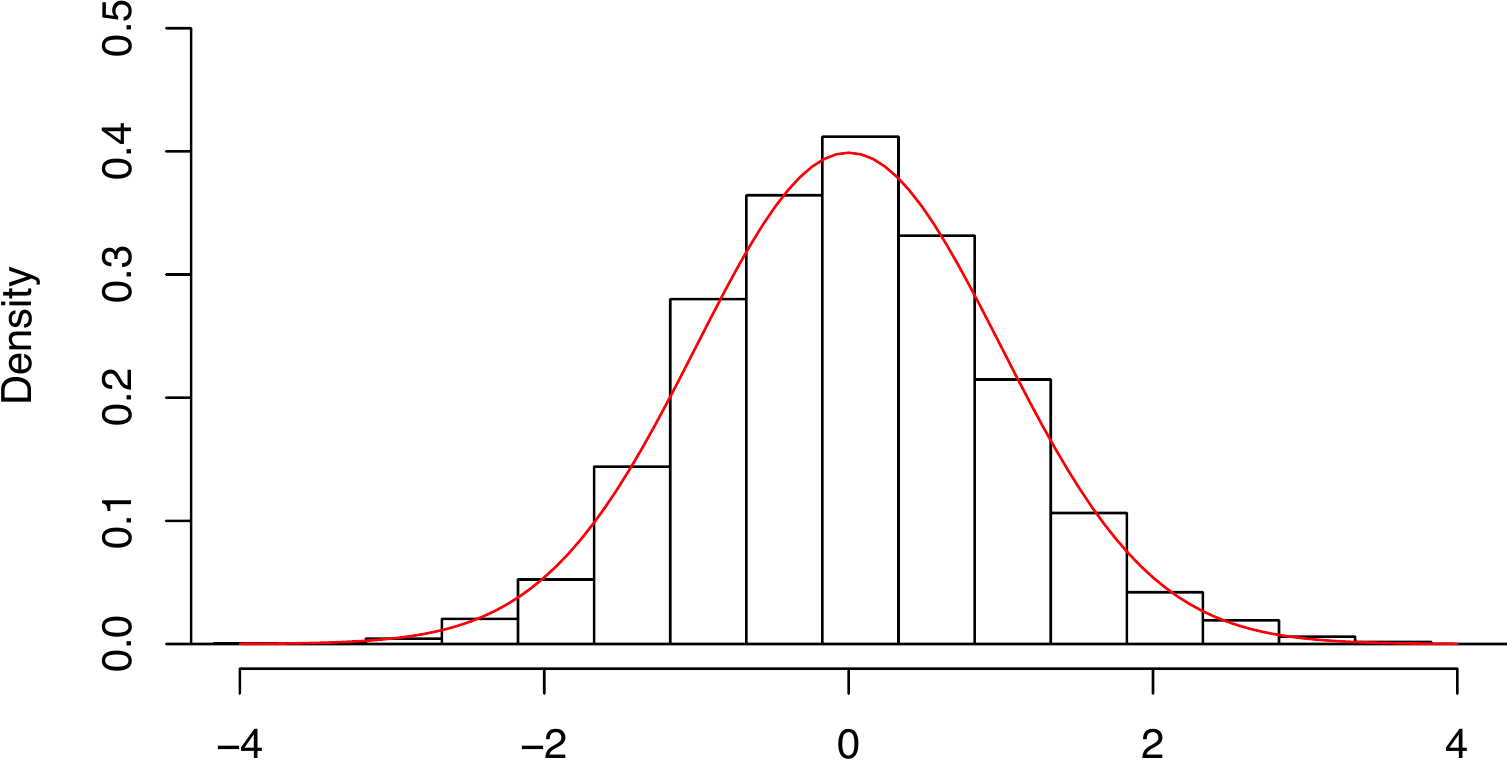}
\end{center}
\end{minipage}

\begin{minipage}{0.5 \hsize}
\begin{center}
\includegraphics[scale=0.49]{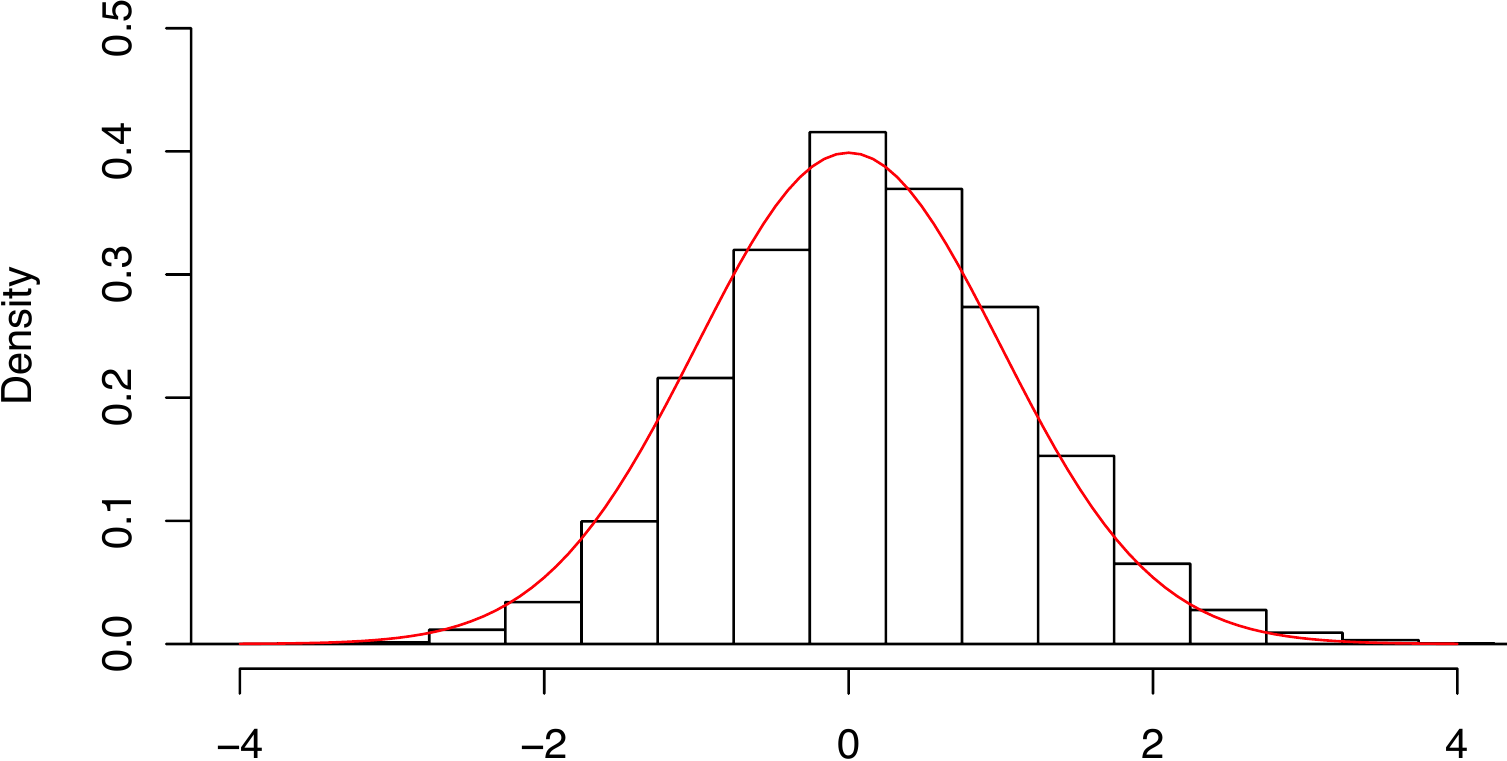}
\end{center}
\end{minipage}

\end{tabular}
\label{hist3}
\end{figure}

\subsection{Model selection} \label{Simu2}

We consider the following diffusion (Diff) and drift (Drif) coefficients:
\begin{align*}
&{\bf Diff}\;{\bf 1:}\exp\Big\{\frac{1}{2}(\alpha_{1}\cos X_{t}+\alpha_{2}\sin X_{t}+\alpha_{3}\cos X_{t}\sin X_{t})\Big\}; \\
&{\bf Diff}\;{\bf 2:}\exp\Big\{\frac{1}{2}(\alpha_{1}\cos X_{t}+\alpha_{2}\sin X_{t})\Big\}; 
\quad \;{\bf Diff}\;{\bf 3:}\exp\Big\{\frac{1}{2}(\alpha_{1}\cos X_{t}+\alpha_{3}\cos X_{t}\sin X_{t})\Big\}; \\
&{\bf Diff}\;{\bf 4:}\exp\Big\{\frac{1}{2}(\alpha_{2}\sin X_{t}+\alpha_{3}\cos X_{t}\sin X_{t})\Big\};
\quad \;{\bf Diff}\;{\bf 5:}\exp\Big\{\frac{1}{2}\alpha_{1}\cos X_{t}\Big\}; \\
&{\bf Diff}\;{\bf 6:}\exp\Big\{\frac{1}{2}\alpha_{2}\sin X_{t}\Big\};
\quad \;{\bf Diff}\;{\bf 7:}\exp\Big\{\frac{1}{2}\alpha_{3}\cos X_{t}\sin X_{t}\Big\},
\end{align*}
and
\begin{align*}
{\bf Drif}\;{\bf 1:}\;\beta_{1}X_{t}+\beta_{2};
\quad \;{\bf Drif}\;{\bf 2:}\;\beta_{1}X_{t};
\quad \;{\bf Drif}\;{\bf 3:}\;\beta_{2}.
\end{align*}
Each candidate model consists of a combination of diffusion and drift coefficients;
for example, in the case of Diff 1 and Drif 1, we consider the statistical model
\begin{align*}
dX_{t}=\exp\biggl\{\frac{1}{2}(\alpha_{1}\cos X_{t}+\alpha_{2}\sin X_{t}+\alpha_{3}\cos X_{t}\sin X_{t})\biggr\}dw_{t}+(\beta_{1}X_{t}+\beta_{2})dt.
\end{align*}
Then, the true model is given by Diff 4 and Drif 2.

In order to empirically quantify relative frequency (percentage) of the model selection,
using the joint m(Q)BIC and two-step m(QBIC) we computed $w_{m_{1},m_{2}}$ and $w_{m_{1},m_{2}}^{\prime}$ defined as follows:
\begin{align*}
w_{m_{1},m_{2}}&=\frac{\ds{\exp\Big\{-\frac{1}{2}\big(\mathrm{m(Q)BIC}_{n}^{(m_{1},m_{2})}-\mathrm{m(Q)BIC}_{n}^{(m_{1,n}^{\ast},m_{2,n}^{\ast})}\big)\Big\}}}{\ds{\sum_{k=1}^{7}\sum_{\ell=1}^{3}}\exp\Big\{-\frac{1}{2}\big(\mathrm{m(Q)BIC}_{n}^{(k,\ell)}-\mathrm{m(Q)BIC}_{n}^{(m_{1,n}^{\ast},m_{2,n}^{\ast})}\big)\Big\}}\times 100, \\
w_{m_{1},m_{2}}^{\prime}&=\frac{\ds{\exp\Big\{-\frac{1}{2}\big(\mathrm{m(Q)BIC}_{n}^{(m_{1})}-\mathrm{m(Q)BIC}_{n}^{(m_{1,n}^{\ast})}\big)\Big\}}}{\ds{\sum_{k=1}^{7}\exp\Big\{-\frac{1}{2}\big(\mathrm{m(Q)BIC}_{n}^{(k)}-\mathrm{m(Q)BIC}_{n}^{(m_{1,n}^{\ast})}\big)\Big\}}} \\
&\quad\times\frac{\ds{\exp\Big\{-\frac{1}{2}\big(\mathrm{m(Q)BIC}_{n}^{(m_{2}|m_{1,n})}-\mathrm{m(Q)BIC}_{n}^{(m_{2,n}^{\ast}|m_{1,n})}\big)\Big\}}}{\ds{\sum_{\ell=1}^{3}\exp\Big\{-\frac{1}{2}\big(\mathrm{m(Q)BIC}_{n}^{(\ell|m_{1,n})}-\mathrm{m(Q)BIC}_{n}^{(m_{2,n}^{\ast}|m_{1,n})}\big)\Big\}}}\times100.
\end{align*}
These ``model weights'' (\cite[Section 6.4.5]{BurAnd02}) are not only numerically stable but also practically convenient,
for one can quantify frequency of relative model evidences among the candidate models from single data set (one sample path).
The model which has the highest $w_{m_{1},m_{2}}$($w_{m_{1},m_{2}}^{\prime}$) value is the most probable model.
Because of the definition, $w_{m_{1},m_{2}}$ and $w_{m_{1},m_{2}}^{\prime}$ satisfy the equation $\sum_{k=1}^{7}\sum_{\ell=1}^{3}w_{k,\ell}=\sum_{k=1}^{7}\sum_{\ell=1}^{3}w_{k,\ell}^{\prime}=100$.

Tables \ref{ms.joint} and \ref{ms.tw} summarize the empirical means of $w_{m_{1},m_{2}}$ and $w_{m_{1},m_{2}}^{\prime}$ and also model-selection frequencies, all computed from $1000$ independent data sets.
The indicators of the true model defined by Diff 4 and Drif 2 are given by $w_{4,2}$ and $w_{4,2}^{\prime}$.
The values of $w_{4,2}$ and $w_{4,2}^{\prime}$ are the highest for all $n$ and become larger as $n$ increases.
Also observed is that $w_{4,2}$ takes higher values than $w_{4,2}^{\prime}$.
Moreover, $w_{m_{1},m_{2}}^{\prime}$ gets close to $w_{m_{1},m_{2}}$ as $n$ increases.

\medskip

\begin{rem}{\rm
Instead of \eqref{se:simu.true}, we also run the same code for the models
\begin{align}
dX_{t}&=\sqrt{\tau}\exp \biggl\{\frac{1}{2}(2\sin X_{t}-\cos X_{t}\sin X_{t}) \biggr\}dw_{t}-\tau X_{t}dt
\label{se:simu.tauall}
\end{align}
and
\begin{align}
dX_{t}&=\sqrt{\tau}\exp \biggl\{\frac{1}{2}(2\sin X_{t}-\cos X_{t}\sin X_{t}) \biggr\}dw_{t}-X_{t}dt
\label{se:simu.taudiff}
\end{align}
with $\tau=2,3$. In the unreported simulation results, we could observe the following: 
in the model \eqref{se:simu.tauall}, similar tendencies were observed for $\tilde{\theta}_{n}$, $\tilde{h}/\tau h_{0}$, and model selection;
in the model \eqref{se:simu.taudiff}, estimation performance of $\tilde{\beta}_{n}$ was inferior.
Both are in accordance with our theoretical findings (see Remark \ref{se:rem.taudiff} for details).
}\qed
\end{rem}

\begin{table}[t]
\begin{center}
\caption{\footnotesize The mean of the model weights $w_{m_{1},m_{2}}$ and model selection frequencies. The true model consists of Diff 4 and Drif 2.}
\begin{tabular}{l l l r r r r r r r} \hline
 & Criteria & & Diff 1 & Diff 2 & Diff 3 & Diff $4^{\ast}$ & Diff 5 & Diff 6 & Diff 7 \\ \hline
 & & & \multicolumn{7}{l}{$n=1000$} \\ \cline{4-10}
Drif 1 & mBIC & weight & 1.34 & 3.12 & 0.35 & 20.06 & 0.08 & 11.24 & 0.21 \\
 & & frequency & 0 & 15 & 1 & 88 & 1 & 89 & 0 \\
 & mQBIC & weight & 8.36 & 3.00 & 0.73 & 23.85 & 0.04 & 3.60 & 0.11 \\
 & & frequency & 20 & 16 & 10 & 169 & 0 & 22 & 0 \\[5pt]
Drif $2^{\ast}$ & mBIC & weight & 2.45 & 5.25 & 0.09 & {\bf 35.19} & 0.00 &15.32 & 0.19 \\
 & & frequency & 4 & 43 & 0 & {\bf 525} & 0 & 230 & 1 \\
 & mQBIC & weight & 11.96 & 4.05 & 0.12 & {\bf 36.32} & 0.00 &4.21 & 0.08 \\
  & & frequency & 51 & 28 & 0 & {\bf 631} & 0 & 52 & 0 \\[5pt]
Drif 3 & mBIC & weight & 0.21 & 0.69 & 0.01 & 3.07 & 0.00 & 1.13 & 0.00 \\
 & & frequency & 0 & 0 & 0 & 3 & 0 & 0 & 0 \\
 & mQBIC & weight & 0.83 & 0.32 & 0.00 & 2.22 & 0.00 & 0.20 & 0.00 \\
 & & frequency & 0 & 0 & 0 & 1 & 0 & 0 & 0 \\[8pt]
 & & & \multicolumn{7}{l}{$n=3000$} \\ \cline{4-10}
Drif 1 & mBIC & weight & 0.91 & 0.61 & 0.00 & 26.70 & 0.00 & 3.32 & 0.00 \\
 & & frequency & 1 & 1 & 0 & 101 & 0 & 26 & 0 \\
 & mQBIC & weight & 4.66 & 0.67 & 0.00 & 26.84 & 0.00 & 0.82 & 0.00 \\
 & & frequency & 10 & 3 & 0 & 136 & 0 & 4 & 0 \\[5pt]
Drif $2^{\ast}$ & mBIC & weight & 2.40 & 1.15 & 0.00 & {\bf 58.17} & 0.00 & 3.91 & 0.00 \\
 & & frequency & 3 & 9 & 0 & {\bf 812} & 0 & 47 & 0 \\
 & mQBIC & weight & 10.76 & 0.89 & 0.00 & {\bf 52.42} & 0.00 & 0.90 & 0.00 \\
 & & frequency & 31 & 9 & 0 & {\bf 801} & 0 & 6 & 0 \\[5pt]
Drif 3 & mBIC & weight & 0.11 & 0.04 & 0.00 & 2.59 & 0.00 & 0.08 & 0.00 \\
 & & frequency & 0 & 0 & 0 & 0 & 0 & 0 & 0 \\
 & mQBIC & weight & 0.36 & 0.01 & 0.00 & 1.64 & 0.00 & 0.01 & 0.00 \\
 & & frequency & 0 & 0 & 0 & 0 & 0 & 0 & 0 \\[8pt]
  & & & \multicolumn{7}{l}{$n=5000$} \\ \cline{4-10}
Drif 1 & mBIC & weight & 0.66 & 0.29 & 0.00 & 27.46 & 0.00 & 1.50 & 0.00 \\
 & & frequency & 0 & 2 & 0 & 102 & 0 & 11 & 0 \\
 & mQBIC & weight & 3.73 & 0.30 & 0.00 & 26.97 & 0.00 & 0.32 & 0.00 \\
 & & frequency & 4 & 3 & 0 & 125 & 0 & 1 & 0 \\[5pt]
Drif $2^{\ast}$ & mBIC & weight & 2.04 & 0.31 & 0.00 & {\bf 64.27} & 0.00 & 1.72 & 0.00 \\
 & & frequency & 0 & 1 & 0 & {\bf 862} & 0 & 21 & 0 \\
 & mQBIC & weight & 9.31 & 0.22 & 0.00 & {\bf 57.53} & 0.00 & 0.36 & 0.00 \\
 & & frequency & 24 & 0 & 0 & {\bf 839} & 0 & 4 & 0 \\[5pt]
Drif 3 & mBIC & weight & 0.05 & 0.00 & 0.00 & 1.68 & 0.00 & 0.01 & 0.00 \\
 & & frequency & 0 & 0 & 0 & 1 & 0 & 0 & 0 \\
 & mQBIC & weight & 0.17 & 0.00 & 0.00 & 1.08 & 0.00 & 0.00 & 0.00 \\
 & & frequency & 0 & 0 & 0 & 0 & 0 & 0 & 0 \\ \hline
\end{tabular}
\label{ms.joint}
\end{center}
\end{table}

\begin{table}[t]
\begin{center}
\caption{\footnotesize The mean of the model weights $w_{m_{1},m_{2}}^{\prime}$ and model selection frequencies. The true model consists of Diff 4 and Drif 2.}
\begin{tabular}{l l l r r r r r r r} \hline
& Criteria & & Diff 1 & Diff 2 & Diff 3 & Diff $4^{\ast}$ & Diff 5 & Diff 6 & Diff 7 \\ \hline
 & & & \multicolumn{7}{l}{$n=1000$} \\ \cline{4-10}
Drif 1 & mBIC & weight & 1.27 & 4.14 & 0.21 & 18.57 & 0.00 & 11.29 & 0.21 \\
 & & frequency & 0 & 14 & 0 & 81 & 0 & 94 & 0 \\
 & mQBIC & weight & 9.01 & 3.83 & 0.26 & 21.37 & 0.00 & 3.47 & 0.09 \\
 & & frequency & 18 & 16 & 1 & 149 & 0 & 27 & 00 \\[5pt]
Drif $2^{\ast}$ & mBIC & weight & 2.43 & 8.58 & 0.19 & {\bf 32.68} & 0.00 & 14.79 & 0.18 \\
 & & frequency & 4 & 95 & 0 & {\bf 489} & 0 & 220 & 0 \\
 & mQBIC & weight & 13.91 & 7.17 & 0.28 & {\bf 32.86} & 0.00 & 3.83 & 0.07 \\
 & & frequency & 66 & 74 & 0 & {\bf 595} & 0 & 53 & 0 \\[5pt]
Drif 3 & mBIC & weight & 0.22 & 1.09 & 0.03 & 2.98 & 0.00 & 1.15 & 0.00 \\
 & & frequency & 0 & 0 & 0 & 3 & 0 & 0 & 0 \\
 & mQBIC & weight & 0.94 & 0.52 & 0.01 & 2.15 & 0.00 & 0.21 & 0.00 \\
 & & frequency & 0 & 0 & 0 & 1 & 0 & 0 & 0 \\[8pt]
  & & & \multicolumn{7}{l}{$n=3000$} \\ \cline{4-10}
Drif 1 & mBIC & weight & 1.27 & 0.78 & 0.00 & 26.01 & 0.00 & 3.62 & 0.00 \\
 & & frequency & 2 & 1 & 0 & 95 & 0 & 31 & 0 \\
 & mQBIC & weight & 6.52 & 0.74 & 0.00 & 25.39 & 0.00 & 0.90 & 0.00 \\
 & & frequency & 9 & 2 & 0 & 131 & 0 & 6 & 0 \\[5pt]
Drif $2^{\ast}$ & mBIC & weight & 2.70 & 1.81 & 0.00 & {\bf 57.05} & 0.00 & 3.90 & 0.00 \\
 & & frequency & 6 & 15 & 0 & {\bf 803} & 0 & 47 & 0 \\
 & mQBIC & weight & 11.97 & 1.39 & 0.00 & {\bf 50.18} & 0.00 & 0.87 & 0.00 \\
 & & frequency & 44 & 15 & 0 & {\bf 787} & 0 & 6 & 0 \\[5pt]
Drif 3 & mBIC & weight & 0.12 & 0.07 & 0.00 & 2.57 & 0.00 & 0.09 & 0.00 \\
 & & frequency & 0 & 0 & 0 & 0 & 0 & 0 & 0 \\
 & mQBIC & weight & 0.40 & 0.02 & 0.00 & 1.60 & 0.00 & 0.01 & 0.00 \\
 & & frequency & 0 & 0 & 0 & 0 & 0 & 0 & 0 \\[8pt]
  & & & \multicolumn{7}{l}{$n=5000$} \\ \cline{4-10}
Drif 1 & mBIC & weight & 1.02 & 0.24 & 0.00 & 27.04 & 0.00 & 1.71 & 0.00 \\
 & & frequency & 1 & 0 & 0 & 97 & 0 & 16 & 0 \\
 & mQBIC & weight & 5.08 & 0.25 & 0.00 & 26.09 & 0.00 & 0.36 & 0.00 \\
 & & frequency & 5 & 1 & 0 & 124 & 0 & 2 & 0 \\[5pt]
Drif $2^{\ast}$ & mBIC & weight & 2.47 & 0.43 & 0.00 & {\bf 63.58} & 0.00 & 1.77 & 0.00 \\
 & & frequency & 6 & 2 & 0 & {\bf 854} & 0 & 23 & 0 \\
 & mQBIC & weight & 10.39 & 0.30 & 0.00 & {\bf 55.91} & 0.00 & 0.35 & 0.00 \\
 & & frequency & 34 & 0 & 0 & {\bf 830} & 0 & 4 & 0 \\[5pt]
Drif 3 & mBIC & weight & 0.05 & 0.01 & 0.00 & 1.68 & 0.00 & 0.02 & 0.00 \\
 & & frequency & 0 & 0 & 0 & 1 & 0 & 0 & 0 \\
 & mQBIC & weight & 0.18 & 0.00 & 0.00 & 1.07 & 0.00 & 0.00 & 0.00 \\ 
 & & frequency & 0 & 0 & 0 & 0 & 0 & 0 & 0 \\ \hline
\end{tabular}
\label{ms.tw}
\end{center}
\end{table}

\section{Proofs}\label{hm:sec_proofs}

\subsection{Preliminaries}
\label{hm:sec_proofs_preliminaries}

We begin with some preliminaries, most of which will be repeatedly used in the sequel, often without mention.
We will denote by $O_{p}^{\ast}$ and $o_{p}^{\ast}$ the stochastic order symbols which are valid uniformly in $\theta$,
and by $\E_{j-1}(\cdot)$ the conditional expectation with respect to the $\sig$-field $\mcf_{t_{j-1}}:=\sig(X_{0})\vee\sig(w_{s};\, s\le t_{j-1})$.
Then $X$ is $(\mcf_{t})$-adapted since we are considering a strong solution to \eqref{hm:sde1}.

Assumption \ref{Ass1} ensures that
\begin{equation}
\max_{i=0,1,2,3}\sup_{\alpha\in\Theta_{\alpha}}\big\|\p_{\alpha}^{i}S^{-1}(x,\alpha)\big\|\lesssim (1+\|x\|)^{C_{0}'}
\nonumber
\end{equation}
for some $C_{0}'\ge 0$.
For any measurable function $f$ on $\mbbr^{d}\times\Theta$ such that 
\begin{equation}
\sup_{\theta}\max_{k=0,1}\|\p^{k}_{x}f(x,\theta)\| \lesssim 1+\|x\|^{C}, 
\label{se:fpg}
\end{equation}
we have
\begin{equation}
\frac{1}{n}\sumj f_{j-1}(\theta) - \int f(x,\theta)\pi(dx) = O_{p}^{\ast}(\sqrt{h_{0}})
\label{se:aux.2}
\end{equation}
by using the basic fact
\begin{equation}
\E(\|X_{t}-X_{s}\|^{q})\lesssim |t-s|^{q/2}
\nonumber
\end{equation}
for $q\ge 2$.
Given an $f$ satisfying \eqref{se:fpg} and taking values in $d\times d$ positive definite matrices, we can make use of the fact
\begin{align*}
\frac{1}{h_{0}^{2}}\bigg(\left\|\E_{j-1}(\D_{j}X)-\tau h_{0} b_{j-1}(\tz)\right\| 
+ \left\|\E_{j-1}\{(\D_{j}X)^{\otimes 2}\}-\tau h_{0} S_{j-1}(\al_{0})\right\|\bigg) \lesssim 1+\|X_{t_{j-1}}\|^{C},
\end{align*}
with the aid of the Sobolev inequality to deduce that
\begin{align}
&\sup_{\theta}\bigg|\frac{1}{\sqrt{n}}\sumj f_{j-1}(\theta)\big[\D_{j}X,b_{j-1}(\theta)\big]\bigg| \nn\\
&\leq\sup_{\theta}\sqrt{h_{0}}\bigg|\frac{1}{\sqrt{nh_{0}}}\sumj f_{j-1}(\theta)\big[\D_{j}X-\mbbe_{j-1}\big(\D_{j}X\big),b_{j-1}(\theta)\big]\bigg| \nn\\
&{}\qquad+\sup_{\theta}\bigg|\frac{1}{\sqrt{n}}\sumj f_{j-1}(\theta)\big[\mbbe_{j-1}\big(\D_{j}X\big)-\tau h_{0} b_{j-1}(\tz),b_{j-1}(\theta)\big]\bigg| \nn\\
&{}\qquad+\sup_{\theta}\bigg|\frac{1}{\sqrt{n}}\sumj f_{j-1}(\theta)\big[\tau h_{0} b_{j-1}(\tz),b_{j-1}(\theta)\big]\bigg| \nn\\
&=O_{p}(\sqrt{h_{0}})+O_{p}\big(\sqrt{n}h_{0}^{2}\big)+O_{p}\big(\sqrt{n}h_{0}\big), \label{se:aux.4}
\end{align}
and that
\begin{align}
&\sup_{\theta}\bigg|\frac{1}{nh_{0}}\sumj f_{j-1}(\theta)\big[(\D_{j}X)^{\otimes2}\big]-\frac{\tau}{n}\sumj f_{j-1}(\theta)\big[S_{j-1}(\alpha_{0})\big]\bigg| \nn\\
&\leq\sup_{\theta}\frac{1}{\sqrt{n}}\bigg|\frac{1}{\sqrt{n}h_{0}}\sumj f_{j-1}(\theta)\big[(\D_{j}X)^{\otimes2}-\mbbe_{j-1}\big\{(\D_{j}X)^{\otimes2}\big\}\big]\bigg| \nn\\
&{}\qquad+\sup_{\theta}\bigg|\frac{1}{nh_{0}}\sumj f_{j-1}(\theta)\big[\mbbe_{j-1}\big\{(\D_{j}X)^{\otimes2}\big\}-\tau h_{0} S_{j-1}(\alpha_{0})\big]\bigg| \nn\\
&=O_{p}\big(n^{-1/2}\big)+O_{p}(h_{0}). \label{se:aux.5}
\end{align}
Because of \eqref{se:aux.2} to \eqref{se:aux.5}, we have under \eqref{hm:sampling.design},
\begin{align}
\frac{1}{\sqrt{n}}\sumj f_{j-1}(\theta)\big[\D_{j}X,b_{j-1}(\theta)\big] &= O_{p}^{\ast}(\sqrt{n}h_{0}), \label{hm:aux.1} \\
\frac{1}{nh_{0}}\sumj f_{j-1}(\theta)\big[(\D_{j}X)^{\otimes2}\big] 
&= \frac{\tau}{n}\sumj f_{j-1}(\theta)[S_{j-1}(\al_{0})] + O_{p}^{\ast}(n^{-1/2}) \nn\\
&= \tau\int_{\mbbr^{d}}\tr\left\{f(x,\theta)S(x,\alpha_{0})\right\}\pi(dx) + O_{p}^{\ast}(\sqrt{h_{0}}). \label{se:aux.3}
\end{align}
We refer to \cite{Mas13} or \cite{Yos11} for details of the above facts.

\subsection{Proof of Theorem \ref{hm:thm2.diffusion.an}} \label{proof.AN}

In this proof we will only consider the case where the constant $C_{0}$ in Assumption \ref{Ass1}(ii) is positive;
then, we see from Fatou's lemma that $\int |x|^{q}\pi(dx)<\infty$ for any $q>0$, so that the law of large numbers \eqref{hm:disc.LLN} is in force for any $g$ of at most polynomial growth.
The proof for $C_{0}=0$ is entirely analogous and is easier: in this case, it will be enough to consider bounded $g$.

\subsubsection{Consistency} \label{hm:proof.consistency}

Let
\begin{align*}
\tilde{\mbby}_{0}^{1}(\alpha)&:=-\frac{1}{2}\left\{d\log\left(\frac{1}{d}\int_{\mbbr^{d}}\tr\big(S^{-1}(x,\alpha)S(x,\alpha_{0})\big)\pi(dx)\right) 
-\int_{\mbbr^{d}}\log\big|S^{-1}(x,\alpha)S(x,\alpha_{0})\big|\pi(dx)\right\}, \\
\tilde{\mbby}_{0}^{2}(\beta)&:=-\frac{\tau}{2}\int_{\mbbr^{d}}S^{-1}(x,\alpha_{0})\big[\big(b(x,\alpha_{0},\beta)-b(x,\tz)\big)^{\otimes2}\big]\pi(dx). 
\end{align*}
These quantities serve as quasi-entropies for estimating $\al$ and $\beta$, hence should appropriately separate the models.

\medskip

\noindent
\textit{Proof of $\tilde{\alpha}_{n}\cip\alpha_{0}$.}
Let
\begin{equation}
\tilde{\mbby}_{n}^{1}(\alpha,\beta) := \frac{1}{n}\left(\tilde{\mbbh}_{n}(\alpha,\beta)-\tilde{\mbbh}_{n}(\alpha_{0},\beta)\right)
\nonumber
\end{equation}
It suffices to deduce that
\begin{align}
& \argmax_{\alpha\in\overline{\Theta}_{\alpha}}\tilde{\mbby}_{0}^{1}(\alpha)=\{\alpha_{0}\}, \label{hm:con.al.1} \\
&  \big| \tilde{\mbby}_{n}^{1}(\alpha,\beta) - \tilde{\mbby}_{0}^{1}(\alpha) \big| =o_{p}^{\ast}(1).
\label{hm:con.al.2}
\end{align}
Indeed, the argmax theorem (see for example \cite{van98}) then concludes the consistency of $\aet$ since $\aet\in\argmax_{\al}\tilde{\mbby}_{n}^{1}(\alpha,\bet)$ and \eqref{hm:con.al.2} implies that $|\tilde{\mbby}_{n}^{1}(\alpha,\bet)-\tilde{\mbby}_{0}^{1}(\alpha) | =o_{p}^{\ast}(1)$.

Let $\lambda_{1}(x,\alpha),\ldots,\lambda_{d}(x,\alpha)$ denote the eigenvalues of $S^{-1}(x,\alpha)S(x,\alpha_{0})$.
By means of the arithmetic-geometric mean inequality and Jensen's inequalities, we see that for every $\alpha$,
\begin{align*}
\tilde{\mbby}_{0}^{1}(\alpha)&=\frac{1}{2}\bigg\{\int_{\mbbr^{d}}\log\big|S^{-1}(x,\alpha)S(x,\alpha_{0})\big|\pi(dx)-d\log\bigg(\frac{1}{d}\int_{\mbbr^{d}}\tr\bigg(S^{-1}(x,\alpha)S(x,\alpha_{0})\bigg)\pi(dx)\bigg)\bigg\} \\
&\leq\frac{1}{2}\bigg[\int_{\mbbr^{d}}\log\bigg(\prod_{i=1}^{d}\lambda_{i}(x,\alpha)\bigg)\pi(dx)-d\log\bigg\{\int_{\mbbr^{d}}\bigg(\prod_{i=1}^{d}\lambda_{i}(x,\alpha)\bigg)^{1/d}\pi(dx)\bigg\}\bigg] \\
&\leq\frac{1}{2}\bigg[\int_{\mbbr^{d}}\log\bigg(\prod_{i=1}^{d}\lambda_{i}(x,\alpha)\bigg)\pi(dx)-\bigg\{\int_{\mbbr^{d}}\log\bigg(\prod_{i=1}^{d}\lambda_{i}(x,\alpha)\bigg)\pi(dx)\bigg\}\bigg] \\
&=0.
\end{align*}
It follows that the following conditions are equivalent:
\begin{itemize}
\item $\tilde{\mbby}_{0}^{1}(\alpha)=0$;
\item The eigenvalues $\lambda_{1}(x,\alpha),\ldots,\lambda_{d}(x,\alpha)$ are constant as a function of $x$, and moreover they are all equal.
\end{itemize}
Under Assumption \ref{Ass3}(i) the equality $\tilde{\mbby}_{0}^{1}(\alpha)=0$ holds only when $\al=\al_{0}$, hence we obtain \eqref{hm:con.al.1}.

Under Assumption \ref{Ass1} we have
\begin{equation}
(1+\|x\|)^{-C} \lesssim \inf_{\al}\lam_{\min}\{S(x,\al)\} \le \sup_{\al}\lam_{\max}\{S(x,\al)\} \lesssim (1+\|x\|)^{C}.
\nonumber
\end{equation}
Hence, the arithmetic-geometric mean inequality gives
\begin{align}
\inf_{\alpha}\int \tr\big(S^{-1}(x,\alpha)S(x,\alpha_{0})\big)\pi(dx)
&\ge \inf_{\alpha}d \int \left[\big|S^{-1}(x,\al)S(x,\al_{0})\big|\right]^{1/d}\pi(dx) \nn\\
&\ge d \int \lam_{\min}\{S(x,\al_{0})\} \left(\inf_{\alpha}\lam_{\min}\{S^{-1}(x,\al)\}\right)\pi(dx) \nn\\
&\gtrsim \int \lam_{\min}\{S(x,\al_{0})\}\left(\sup_{\alpha}\lam_{\max}\{S(x,\al)\}\right)^{-1}\pi(dx) \nn\\
&\gtrsim \int (1+\|x\|)^{-C}\pi(dx) > 0.
\nonumber
\end{align}
Then, it follows from the definition \eqref{hm:mGQMLE.def} that 
\begin{align}
\tilde{\mbby}_{n}^{1}(\alpha,\beta)
&= O_{p}^{\ast}(h_{0})
-\frac{1}{2}\bigg[-\frac{1}{n}\sumj\log\big|S_{j-1}^{-1}(\alpha)S_{j-1}(\alpha_{0})\big| \nn\\
&{}\qquad +d\bigg\{
\log\bigg(\frac{1}{nh_{0}d}\sumj S_{j-1}^{-1}(\alpha)\big[(\D_{j}X)^{\otimes2}\big]\bigg)
-\log\bigg(\frac{1}{nh_{0}d}\sumj S_{j-1}^{-1}(\alpha_{0})\big[(\D_{j}X)^{\otimes2}\big]\bigg)\bigg\}\bigg]
\nn\\
&=o_{p}^{\ast}(1) + \frac{1}{2}\bigg[ \int_{\mbbr^{d}}\log\big|S^{-1}(x,\alpha)S(x,\alpha_{0})\big|\pi(dx) \nn\\
&{}\qquad -d\bigg\{
\log\bigg( \frac{\tau}{d}\int_{\mbbr^{d}}\tr\big(S^{-1}(x,\al)S(x,\alpha_{0})\big)\pi(dx) + o_{p}^{\ast}(1) \bigg)
-\log\left( \tau + o_{p}^{\ast}(1) \right)
\bigg\}\bigg]
\nn\\
&=o_{p}^{\ast}(1) + \frac{1}{2}\bigg[ \int_{\mbbr^{d}}\log\big|S^{-1}(x,\alpha)S(x,\alpha_{0})\big|\pi(dx) \nn\\
&{}\qquad -d
\log\bigg( \frac{1}{d}\int_{\mbbr^{d}}\tr\big(S^{-1}(x,\al)S(x,\alpha_{0})\big)\pi(dx) \bigg) + o_{p}^{\ast}(1) \bigg] \nn\\
&=\tilde{\mbby}_{0}^{1}(\alpha)+o_{p}^{\ast}(1).
\nn
\end{align}
Thus \eqref{hm:con.al.2} is verified, concluding the consistency of $\aet$.

\medskip

\noindent
\textit{Proof of $\tilde{\beta}_{n}\cip\beta_{0}$.}
Let
\begin{align*}
\tilde{\mbby}_{n}^{2}(\beta;\tilde{\alpha}_{n}) &:=\frac{1}{nh_{0}}\left(\tilde{\mbbh}_{n}(\tilde{\alpha}_{n},\beta)-\tilde{\mbbh}_{n}(\tilde{\alpha}_{n},\beta_{0})\right) \\
&=\frac{1}{nh_{0}}\sumj S_{j-1}^{-1}(\tilde{\alpha}_{n})\big[\D_{j}X,b_{j-1}(\tilde{\alpha}_{n},\beta)-b_{j-1}(\tilde{\alpha}_{n},\beta_{0})\big] \\
&\quad-\frac{1}{2}\bigg(\frac{1}{h_{0}}\frac{1}{nd}\sumj S_{j-1}^{-1}(\tilde{\alpha}_{n})\big[(\D_{j}X)^{\otimes2}\big]\bigg)
\bigg(\frac{1}{n}\sumj S_{j-1}^{-1}(\tilde{\alpha}_{n})\big[b_{j-1}(\tilde{\alpha}_{n},\beta)^{\otimes2}-b_{j-1}(\tilde{\alpha}_{n},\beta_{0})^{\otimes2}\big]\bigg).
\end{align*}
Since $\{\beta_{0}\}=\argmax_{\beta}\tilde{\mbby}_{0}^{2}(\beta)$ by Assumption \ref{Ass3}(ii) and $\bet\in\argmax_{\beta}\tilde{\mbby}_{n}^{2}(\beta;\tilde{\alpha}_{n})$ it remains to show that $|\tilde{\mbby}_{n}^{2}(\beta;\tilde{\alpha}_{n})-\tilde{\mbby}_{0}^{2}(\beta)| = o_{p}^{\ast}(1)$.
By the consistency of $\aet$, we have $b_{j-1}(\tilde{\alpha}_{n},\beta)=b_{j-1}(\alpha_{0},\beta)+o_{p}^{\ast}(1)$ and
\begin{align}
\frac{1}{h_{0}}\frac{1}{nd}\sumj S_{j-1}^{-1}(\tilde{\alpha}_{n})\big[(\D_{j}X)^{\otimes2}\big]
=\frac{\tau}{d}\int_{\mbbr^{d}}\tr\left\{S^{-1}(x,\al_{0})S(x,\alpha_{0})\right\}\pi(dx) + o_{p}^{\ast}(1)
=\tau + o_{p}^{\ast}(1).
\nonumber
\end{align}
A Sobolev-inequality argument for the martingale term then yields the desired convergence:
\begin{align}
\tilde{\mbby}_{n}^{2}(\beta;\tilde{\alpha}_{n})
&=\frac{1}{\sqrt{nh_{0}}}\bigg(\frac{1}{\sqrt{nh_{0}}}\sumj S_{j-1}^{-1}(\tilde{\alpha}_{n})\big[\D_{j}X-\E_{j-1}(\D_{j}X),b_{j-1}(\alpha_{0},\beta)-b_{j-1}(\tz)\big]\bigg) \nn\\
&{}\qquad+\frac{\tau}{n}\sumj S_{j-1}^{-1}(\al_{0})\big[b_{j-1}(\tz),b_{j-1}(\alpha_{0},\beta)-b_{j-1}(\tz)\big] \nn\\
&{}\qquad-\frac{\tau}{2n}
\sumj S_{j-1}^{-1}(\al_{0})\big[b_{j-1}(\alpha_{0},\beta)^{\otimes2}-b_{j-1}(\tz)^{\otimes2}\big] + o_{p}^{\ast}(1)
\nn\\
&=-\frac{\tau}{2n}
\sumj S_{j-1}^{-1}(\al_{0})\big[\big\{b_{j-1}(\alpha_{0},\beta)-b_{j-1}(\tz)\big\}^{\otimes2}\big] + o_{p}^{\ast}(1)
\nn\\
&=\tilde{\mbby}_{0}^{2}(\beta) + o_{p}^{\ast}(1).
\label{hm:al-consis+2}
\end{align}

\subsubsection{Asymptotic normality} \label{proof.AN.an}
Prior to the proof of \eqref{hm:diffusion.an1+}, we will show
\begin{equation}
\big(\sqrt{n}(\tilde{\alpha}_{n}-\alpha_{0}),\, \sqrt{n\tau h_{0}}(\tilde{\beta}_{n}-\beta_{0})\big)\cil N_{p}\left(0,\diag(\tilde{\Gam}_{1,0}^{-1},\, \tilde{\Gam}^{-1}_{2,0})\right).
\label{hm:diffusion.an1}
\end{equation}
Let
\begin{equation}
D_{n}=D_{n}(h_{0}):=\diag\big(\sqrt{n}I_{p_{\al}},\, \sqrt{n\tau h_{0}}I_{p_{\beta}}\big), \qquad \tilde{\Gam}_{0}:=\diag(\tilde{\Gam}_{1,0},\, \tilde{\Gam}_{2,0}).
\nonumber
\end{equation}
By the standard argument, the consistency of $\tet$ ensures that $\pr\{\p_{\theta}\tilde{\mbbh}_{n}(\tet)=0\}\to 1$, so that we may and do focus on the event $\{\p_{\theta}\tilde{\mbbh}_{n}(\tet)=0\}$.
Then, by the Taylor expansion of $\theta\mapsto\p_{\theta}\tilde{\mbbh}_{n}(\tet)=0$ around $\tz$ and the measurable selection theorem (recall that $\overline{\Theta}$ is compact, so that $\tet$ a.s. exists), it suffices for \eqref{hm:diffusion.an1} to show
\begin{align}
\tilde{\D}_{n}&:= D_{n}^{-1}\p_{\theta}\tilde{\mbbh}_{n}(\tz) \cil N_{p}(0,\, \tilde{\Gam}_{0}), \label{hm:an.1-1} \\
\tilde{\Gam}_{n}(\rho_{n}) &:= -D_{n}^{-1}\p_{\theta}^{2}\tilde{\mbbh}_{n}(\rho_{n})D_{n}^{-1} \cip \tilde{\Gam}_{0},
\label{hm:an.1-2}
\end{align}
for any family $(\rho_{n})$ of random variables such that $\rho_{n} \cip \tz$.

\medskip

For brevity, from now on we will often remove the dependence on $\tz$ from the notation: $\tilde{\mbbh}_{n}:=\tilde{\mbbh}_{n}(\tz)$, $S_{j-1}:=S_{j-1}(\al_{0})$, $h:=h(\al_{0})$, and so on.

\medskip

\paragraph{{\it Proof of \eqref{hm:an.1-1}}}

First, we will specify the leading term of $\tilde{\D}_{n}$.
Introduce the following martingale-difference arrays:
\begin{align*}
\eta_{j} &:= S_{j-1}^{-1}\bigg[a_{j-1}\frac{\D_{j}w}{\sqrt{h_{0}}}, \, \p_{\beta}b_{j-1}\bigg] \in \mbbr^{p_{\beta}}, \\
\zeta_{1,j} &:=(\p_{\alpha}S_{j-1}^{-1})\bigg[\bigg(a_{j-1}\frac{\D_{j}w}{\sqrt{h_{0}}}\bigg)^{\otimes2}\bigg] + \tr\Big(S_{j-1}^{-1}\big(\p_{\alpha}S_{j-1}\big)\Big)\in \mbbr^{p_{\al}}, \\
\zeta_{2,j} &:= \bigg\{\int_{\mbbr^{d}}\tr\Big(S^{-1}(x,\alpha_{0})\p_{\alpha}S(x,\alpha_{0})\Big)\pi(dx)\bigg\}
\bigg(\frac{1}{d}\bigg\|\frac{\D_{j}w}{\sqrt{h_{0}}}\bigg\|^{2}-1\bigg) \in \mbbr^{p_{\al}}.
\end{align*}
Obviously,
\begin{equation}
\sup_{n}\sup_{j\le n}\big\{\E(\|\eta_{j}\|^{q}) \vee \E(\|\zeta_{1,j}\|^{q}) \vee \E(\|\zeta_{2,j}\|^{q})\big\}<\infty
\label{hm:an.1-2.5}
\end{equation}
for every $q>0$. 
By the standard arguments(for example, \cite{Mas13, Yos11}) we have
\begin{equation}
\frac{1}{\sqrt{n}}\sumj \bigg(\p_{\al}\big(\log|S_{j-1}|\big) + \frac{1}{\tau h_{0}}\p_{\al}S_{j-1}^{-1}\big[(\D_{j}X)^{\otimes2}\big]\bigg)=\frac{1}{\sqrt{n}}\sumj \zeta_{1,j} + o_{p}(1).
\label{hm:an.1-2.5+1}
\end{equation}
We can observe that
\begin{align}
\frac{1}{\sqrt{n}}\p_{\alpha}\tilde{\mbbh}_{n}
&= -\frac{1}{2\sqrt{n}}\sumj \p_{\al}\big(\log|S_{j-1}|\big) -\frac{nd}{2\sqrt{n}}\cdot\frac{\ds{\sumj\p_{\alpha}S_{j-1}^{-1}\big[(\D_{j}X)^{\otimes2}\big]}}
{\ds{\sumj S_{j-1}^{-1}\big[(\D_{j}X)^{\otimes2}\big]}} \nn\\
&\quad +\frac{1}{\sqrt{n}}\bigg(\sumj\p_{\alpha}S_{j-1}^{-1}\big[\D_{j}X,b_{j-1}\big]+\sumj S_{j-1}^{-1}\big[\D_{j}X,\p_{\alpha}b_{j-1}\big]\bigg) \nn\\
&\quad -\frac{1}{2\sqrt{n}}\bigg(\frac{1}{nd}\sumj\p_{\alpha}S_{j-1}^{-1}\big[(\D_{j}X)^{\otimes2}\big]\bigg)\bigg(\sumj S_{j-1}^{-1}\big[b_{j-1}^{\otimes2}\big]\bigg) \nn\\
&\quad -\frac{1}{2\sqrt{n}}\bigg(\frac{1}{nd}\sumj S_{j-1}^{-1}\big[(\D_{j}X)^{\otimes2}\big]\bigg)\bigg(\sumj\p_{\alpha}S_{j-1}^{-1}\big[b_{j-1}^{\otimes2}\big]+2\sumj S_{j-1}^{-1}\big[b_{j-1},\p_{\alpha}b_{j-1}\big]\bigg) \nn\\
&=-\frac{1}{2\sqrt{n}}\sumj \p_{\al}\big(\log|S_{j-1}|\big) -\frac{1}{2\sqrt{n}h}\cdot
\sumj\p_{\alpha}S_{j-1}^{-1}\big[(\D_{j}X)^{\otimes2}\big] +O_{p}\big(\sqrt{h_{0}}\big)+ O_{p}(\sqrt{n}h_{0})
\nn\\
&=-\frac{1}{2\sqrt{n}}\sumj \zeta_{1,j} +\frac{1}{2n}\sumj \frac{1}{\tau h_{0}}\p_{\al}S_{j-1}^{-1}\big[(\D_{j}X)^{\otimes2}\big] \cdot \frac{\tau h_{0}}{h}\cdot \sqrt{n}\bigg(\frac{h}{\tau h_{0}}-1\bigg) + o_{p}(1)
\nn\\
&=-\frac{1}{2\sqrt{n}}\sumj \zeta_{1,j}-\frac{1}{2}\bigg\{\int \tr\Big(S^{-1}(x,\al_{0})\big(\p_{\al}S(x,\al_{0})\big)\Big)\pi(dx) + o_{p}(1)\bigg\} \nn \\
&\quad\cdot\frac{\tau h_{0}}{h}\cdot\sqrt{n}\bigg(\frac{h}{\tau h_{0}}-1\bigg) + o_{p}(1).
\label{hm:an.1-3-1}
\end{align}
As in \eqref{hm:an.1-2.5+1}, we can deduce that
\begin{align}
\sqrt{n}\bigg(\frac{h}{\tau h_{0}}-1\bigg)
&=\frac{1}{\tau d\sqrt{n}}\sumj S_{j-1}^{-1}\bigg[\bigg(\frac{\D_{j}X}{\sqrt{h_{0}}}\bigg)^{\otimes2}
-\E_{j-1}\bigg\{\bigg(\frac{\D_{j}X}{\sqrt{h_{0}}}\bigg)^{\otimes2}\bigg\}\bigg]
\nn\\
&\quad +\sqrt{n}\frac{1}{\tau dn}\sumj S_{j-1}^{-1}
\bigg[\E_{j-1}\bigg\{\bigg(\frac{\D_{j}X}{\sqrt{h_{0}}}\bigg)^{\otimes2}\bigg\} - \tau S_{j-1}\bigg]
\nn\\
&=\frac{1}{\tau \sqrt{n}}\sumj \bigg(\frac{1}{d}S_{j-1}^{-1}\bigg[\bigg(\sqrt{\tau}a_{j-1}\frac{\D_{j}w}{\sqrt{h_{0}}}\bigg)^{\otimes2}\bigg]-\tau\bigg)
+O_{p}(\sqrt{n}h_{0}) \nn\\
&=\frac{1}{\sqrt{n}}\sumj \bigg(\frac{1}{d}\bigg\|\frac{\D_{j}w}{\sqrt{h_{0}}}\bigg\|^{2}-1\bigg)+O_{p}(\sqrt{n}h_{0})=O_{p}(1)
\label{hm:an.1-3-2}
\end{align}
by the Lindeberg-Feller theorem.
Substituting the last expression of \eqref{hm:an.1-3-2} into \eqref{hm:an.1-3-1}, we conclude that
\begin{align}
\frac{1}{\sqrt{n}}\p_{\alpha}\tilde{\mbbh}_{n} &=-\frac{1}{2\sqrt{n}}\sumj \zeta_{1,j}-\frac{1}{2}\bigg\{\int \tr\Big(S^{-1}(x,\al_{0})\big(\p_{\al}S(x,\al_{0})\big)\Big)\pi(dx) + o_{p}(1)\bigg\} \nn \\
&\quad\cdot\left(1+o_{p}(1)\right)\cdot O_{p}(1) + o_{p}(1) \nn \\
&= -\frac{1}{2\sqrt{n}}\sumj(\zeta_{1,j}+\zeta_{2,j})+o_{p}(1).
\label{hm:an.1-3}
\end{align}
As for the $\beta$-part, we have
\begin{align}
\frac{1}{\sqrt{n\tau h_{0}}}\p_{\beta}\tilde{\mbbh}_{n}
&= \frac{1}{\sqrt{n\tau h_{0}}}\sumj S_{j-1}^{-1}\big[\D_{j}X, \p_{\beta}b_{j-1}\big] 
- \frac{h}{\sqrt{n\tau h_{0}}}\sumj S_{j-1}^{-1}\big[b_{j-1}, \p_{\beta}b_{j-1}\big]
\label{hm:aux.2} \\
&= \frac{1}{\sqrt{n\tau h_{0}}}\sumj S_{j-1}^{-1}\big[\D_{j}X - \E_{j-1}(\D_{j}X), \p_{\beta}b_{j-1}\big] \nn\\
&\quad
- \sqrt{\tau h_{0}} \bigg\{\sqrt{n}\bigg(\frac{h}{\tau h_{0}}-1\bigg) \cdot \frac{1}{n}\sumj S_{j-1}^{-1}\big[b_{j-1}, \p_{\beta}b_{j-1}\big]\bigg\} + O_{p}(\sqrt{n}h_{0}) \nn\\
&=\frac{1}{\sqrt{n\tau h_{0}}}\sumj S_{j-1}^{-1}\big[\D_{j}X - \E_{j-1}(\D_{j}X), \p_{\beta}b_{j-1}\big] + O_{p}(\sqrt{h_{0}})+ O_{p}(\sqrt{n}h_{0}) \nn\\
&=\frac{1}{\sqrt{nh_{0}}}\sumj S_{j-1}^{-1}\big[a_{j-1}\D_{j}w,\, \p_{\beta}b_{j-1}\big] \nn\\
&\quad +\frac{1}{\sqrt{n}}\sumj S_{j-1}^{-1}\bigg[\frac{1}{\sqrt{h_{0}}}\int_{t_{j-1}}^{t_{j}}\left(
a(X_{s},\al_{0})-a_{j-1}\right) dw_{s},\, \p_{\beta}b_{j-1}\bigg] + O_{p}(\sqrt{n}h_{0}) \nn\\
&=\frac{1}{\sqrt{n}}\sumj \eta_{j} + o_{p}(1).
\label{hm:an.1-4}
\end{align}
Thanks to \eqref{hm:an.1-2.5}, the convergence (the Lyapunov condition)
\begin{equation}
\sumj\E_{j-1}\bigg(\bigg\|\frac{1}{\sqrt{n}}(\zeta_{1,j}+\zeta_{2,j})\bigg\|^{4}\bigg) \,\vee\,
\sum\E_{j-1}\bigg( \bigg\|\frac{1}{\sqrt{n}}\eta_{j}\bigg\|^{4}\bigg) \cip 0
\nn
\end{equation}
is trivial. In view of the stochastic expansions \eqref{hm:an.1-3} and \eqref{hm:an.1-4} and the central limit theorem for martingale difference arrays, the convergence \eqref{hm:an.1-1} follows from the convergences of the quadratic characteristics:
\begin{equation}
\left\{
\begin{array}{l}
\ds{\frac{1}{4n}\sumj\E_{j-1}\big\{(\zeta_{1,j}+\zeta_{2,j})^{\otimes 2}\big\}\cip\tilde{\Gam}_{1,0},} \\
\ds{\frac{1}{n}\sumj\E_{j-1}\big(\eta_{j}^{\otimes 2}\big) \cip \tilde{\Gam}_{2,0},} \\
\ds{\frac{1}{2n}\sumj\E_{j-1}\big\{ (\zeta_{1,j}+\zeta_{2,j})\eta_{j}^{\top}\big\} \cip 0.}
\end{array}
\right.
\label{hm:an.1-5}
\end{equation}
The third one is trivial since $\sumj\E_{j-1}\big\{ (\zeta_{1,j}+\zeta_{2,j})\eta_{j}^{\top}\big\}=0$ a.s.
We will only show the first one, for the second one is exactly the same as in the case where $h_{0}$ is known (obviously $\tilde{\Gam}_{2,0}>0$).

\medskip

Fix any $u_{1}\in\mbbr^{p_{\al}}$.
Since $\|\D_{j}w / \sqrt{h_{0}}\|^{2} \sim \chi^{2}(d)$ conditional on $\mcf_{t_{j-1}}$, it follows that
\begin{align}
\frac{1}{4n}\sumj\E_{j-1}\big(\zeta_{2,j}^{\otimes 2}\big)[u_{1}^{\otimes 2}]
&= \frac{1}{4}\bigg\{\int_{\mbbr^{d}}\tr\Big(S^{-1}(x,\alpha_{0})\big(\p_{\alpha}S(x,\alpha_{0})\big)\Big)[u_{1}]\pi(dx)\bigg\}^{2}
\bigg(\frac{2}{d}+o_{p}(1)\bigg) \nn\\
&\cip \frac{1}{2d}\bigg\{\int_{\mbbr^{d}}\tr\Big(S^{-1}(x,\alpha_{0})\big(\p_{\alpha}S(x,\alpha_{0})\big)\Big)[u_{1}]\pi(dx)\bigg\}^{2}.
\label{hm:an.1-6}
\end{align}
Write $\ep_{j}=h_{0}^{-1/2}\D_{j}w$ and $A_{1,j-1}=S_{j-1}^{-1/2}(\p_{\al}S_{j-1}[u_{1}])S_{j-1}^{-1/2}$; then, $v_{j}:=S_{j-1}^{-1/2}a_{j-1}\ep_{j} \sim N_{d}(0,I_{d})$ conditional on $\mcf_{t_{j-1}}$.
We will make use of the following moment expression (see \cite[Theorem 4.2]{MagNeu79}): 
for the Wishart distributed random variable $W_{j}:=v_{j}v_{j}^{\top}$ and for any $\mcf_{t_{j-1}}$-measurable $d\times d$-symmetric matrices $M'_{j-1}$ and $M''_{j-1}$,
\begin{equation}
\E_{j-1}\left\{\tr(M'_{j-1}W_{j})\tr(M''_{j-1}W_{j})\right\} 
= 2\tr(M'_{j-1}M''_{j-1}) + \tr(M'_{j-1})\tr(M''_{j-1}).
\label{hm:wishart}
\end{equation}
Write $m_{0}=\int_{\mbbr^{d}}\tr(S^{-1}(x,\alpha_{0})\p_{\alpha}S(x,\alpha_{0}))[u_{1}]\pi(dx)$. By \eqref{hm:wishart} we have
\begin{align}
\frac{1}{2n}\sumj\E_{j-1}\big(\zeta_{1,j}\zeta_{2,j}^{\top}\big)[u_{1}^{\otimes 2}]
&=\frac{1}{2n}\sumj\E_{j-1}\bigg\{\bigg(\frac{1}{h_{0}}\p_{\alpha}S_{j-1}^{-1}[u_{1}]\bigg)\big[\big(a_{j-1}\D_{j}w\big)^{\otimes2}\big]\bigg\} \nn\\
&=-\frac{m_{0}}{2}\frac{1}{n} \sumj \E_{j-1}\bigg( \frac{1}{d}v_{j}^{\top}A_{1,j-1}v_{j} v_{j}^{\top}v_{j} - v_{j}^{\top}A_{1,j-1}v_{j}\bigg) \nn\\
&=-\frac{m_{0}}{2}\frac{1}{n} \sumj \bigg\{\frac{1}{d}\E_{j-1}\left(\tr(A_{1,j-1}W_{j}) \tr(I_{d}W_{j})\right) - \tr(A_{1,j-1}I_{d}) \bigg\} \nn\\
&=-\frac{m_{0}}{2}\frac{1}{n} \sumj \bigg\{\frac{1}{d}\left( 2\tr(A_{1,j-1}) + d\tr(A_{1,j-1}) \right) - \tr(A_{1,j-1}) \bigg\} \nn\\
&=-\frac{m_{0}}{d}\frac{1}{n} \sumj \tr(A_{1,j-1}) \cip -\frac{m_{0}^{2}}{d}.
\label{hm:an.1-7}
\end{align}
Further, again by \eqref{hm:wishart},
\begin{align}
& \frac{1}{4n}\sumj\E_{j-1}\big(\zeta_{1,j}^{\otimes 2}\big)[u_{1}^{\otimes 2}] \nn\\
&=\frac{1}{4n}\sumj \E_{j-1}\left\{\left(-\tr(A_{1,j-1}W_{j}) + \tr(A_{1,j-1})\right)^{2}\right\} \nn\\
&=\frac{1}{4n}\sumj \left[ \E_{j-1}\left\{\tr(A_{1,j-1}W_{j})^{2} \right\} -2\tr(A_{1,j-1})\E_{j-1}\left\{\tr(A_{1,j-1}W_{j})\right\} + \tr(A_{1,j-1})^{2}\right] \nn\\
&= \frac{1}{2n}\sumj \tr(A_{1,j-1}^{2}) \nn\\
&\cip \frac{1}{2}\int_{\mbbr^{d}}\tr\Big(S^{-1}(x,\alpha_{0})\big(\p_{\alpha}S(x,\alpha_{0})\big)S^{-1}(x,\alpha_{0})\big(\p_{\alpha}S(x,\alpha_{0})\big)\Big)[u_{1}^{\otimes2}]\pi(dx).
\label{hm:an.1-8}
\end{align}
The first one in \eqref{hm:an.1-5} follows from \eqref{hm:an.1-6}, \eqref{hm:an.1-7} and \eqref{hm:an.1-8}. 

\medskip

It remains to show the positive definiteness of $\tilde{\Gam}_{1,0}$.
Let $M(x;u_{1}):=(S^{-1}(\p_{\al}S[u_{1}]))(x,\al_{0})$. The Cauchy-Schwarz inequality gives
\begin{align}
\tilde{\Gam}_{1,0}[u_{1}^{\otimes 2}] &= \frac{1}{2d}\bigg\{d\int\tr\big(M(x;u_{1})^{2} \big)\pi(dx) -\bigg(\int\tr\big(M(x;u_{1})\big)\pi(dx)\bigg)^{2}\bigg\}
\nn\\
&\ge \frac{1}{2d}\int \left( d\tr\big(M(x;u_{1})^{2} \big) - \big\{\tr\big(M(x;u_{1})\big)\big\}^{2}\right)\pi(dx).
\nonumber
\end{align}
Under Assumption \ref{Ass3}(i), the positivity of the last lower bound for $u_{1}\ne 0$ is ensured by the Cauchy-Schwarz type inequality $\int \tr\{C(x)\}^{2}\pi(dx) \le d\int \tr\{C(x)^{2}\}\pi(dx)$ for any $d\times d$-matrix valued function $C(x)$ with the equality holding only when $x\mapsto\tr\{C(x)\}$ is $\pi$-a.e. constant.
This together with Assumption \ref{Ass3}(i) verifies the positive definiteness of $\tilde{\Gam}_{1,0}$.

\medskip

\paragraph{{\it Proof of \eqref{hm:an.1-2}}}

Fix any $u_{1}\in\mbbr^{p_{\al}}$ and $u_{2}\in\mbbr^{p_{\beta}}$, and write $\rho_{n}=(\rho_{\al,n},\rho_{\beta,n})\in\mbbr^{p_{\al}}\times\mbbr^{p_{\beta}}$.

We can handle $\p_{\alpha}\p_{\beta}\tilde{\mbbh}_{n}$ and $\p_{\beta}^{2}\tilde{\mbbh}_{n}$ in similar ways to the case of known $h_{0}$ (see \cite{Kes97} for details): by \eqref{hm:aux.1} and \eqref{hm:aux.2} with $\tz$ replaced by $\theta$ we have
\begin{align*}
-\frac{1}{n\sqrt{\tau h_{0}}}\p_{\alpha}\p_{\beta}\tilde{\mbbh}_{n}(\theta)[u_{1},u_{2}] =  o_{p}^{\ast}(1) \cip 0.
\end{align*}
Since it can be shown that $D_{n}(\rho_{n}-\tz)=O_{p}(1)$ and $\sup_{\theta}\|\frac{1}{n\tau h_{0}}\p_{\theta}\p_{\beta}^{2}\tilde{\mbbh}_{n}(\theta)\|=O_{p}(1)$, we obtain 
\begin{align}
-\frac{1}{n\tau h_{0}}\p_{\beta}^{2}\tilde{\mbbh}_{n}(\rho_{n})[u_{2}^{\otimes 2}]
&=-\frac{1}{n\tau h_{0}}\left\{\p_{\beta}^{2}\tilde{\mbbh}_{n}(\tz)[u_{2}^{\otimes 2}]+D_{n}^{-1}\int_{0}^{1}\p_{\theta}\p_{\beta}^{2}\tilde{\mbbh}_{n}\big(\tz+s(\rho_{n}-\tz)\big)ds[D_{n}(\rho_{n}-\tz)]\right\} \nn\\
&=-\frac{1}{n\tau h_{0}}\sumj S_{j-1}^{-1}(\alpha_{0})\big[\D_{j}X, \p_{\beta}^{2}b_{j-1}(\tz)[u_{2}^{\otimes 2}]\big] \nn\\
&\qquad+\frac{h(\alpha_{0})}{n\tau h_{0}}\bigg(\sumj S_{j-1}^{-1}(\alpha_{0})\big[b_{j-1}(\tz), \p_{\beta}^{2}b_{j-1}(\tz)[u_{2}^{\otimes 2}]\big] \nn\\
&\qquad+\sumj S_{j-1}^{-1}(\alpha_{0})\big[\p_{\beta}b_{j-1}(\tz)[u_{2}], \p_{\beta}b_{j-1}(\tz)[u_{2}]\big]\bigg)+o_{p}(1) \nn\\
&=-\frac{1}{n\tau h_{0}}\sumj S_{j-1}^{-1}(\alpha_{0})\big[\D_{j}X-\E_{j-1}(\D_{j}X), \p_{\beta}^{2}b_{j-1}(\tz)[u_{2}^{\otimes 2}]\big] \nn\\
&\qquad-\frac{1}{n\tau h_{0}}\sumj S_{j-1}^{-1}(\alpha_{0})\big[\E_{j-1}(\D_{j}X)-\tau h_{0}b_{j-1}(\tz), \p_{\beta}^{2}b_{j-1}(\tz)[u_{2}^{\otimes 2}]\big] \nn\\
&\qquad-\frac{1}{n\tau h_{0}}\sumj S_{j-1}^{-1}(\alpha_{0})\big[\tau h_{0}b_{j-1}(\tz), \p_{\beta}^{2}b_{j-1}(\tz)[u_{2}^{\otimes 2}]\big] \nn\\
&\qquad+\big((1+o_{p}(1)\big)\bigg(\frac{1}{n}\sumj S_{j-1}^{-1}(\alpha_{0})\big[b_{j-1}(\tz), \p_{\beta}^{2}b_{j-1}(\tz)[u_{2}^{\otimes 2}]\big] \nn\\
&\qquad+\frac{1}{n}\sumj S_{j-1}^{-1}(\alpha_{0})\big[\p_{\beta}b_{j-1}(\tz)[u_{2}], \p_{\beta}b_{j-1}(\tz)[u_{2}]\big]\bigg)+o_{p}(1) \nn\\
&\cip \tilde{\Gam}_{2,0}[u_{2}^{\otimes 2}].
\label{hm:an.2-1-1}
\end{align}
For a square matrix $S=S(\al)$ with $\dot{S}$ and $\ddot{S}$ respectively denoting the first and second derivatives with respect to $\al$, the following two identities hold:
\begin{align}
\p_{\al}^{2}(S^{-1}) &= 2S^{-1}\dot{S}S^{-1}\dot{S}S^{-1} - S^{-1}\ddot{S}S^{-1}, \nn\\
\p_{\al}^{2}\big(\log|S|\big) &= \tr(-S^{-1}\dot{S}S^{-1}\dot{S} + S^{-1}\ddot{S}).
\nonumber
\end{align}
We have
\begin{equation}
\frac{\p_{\al}^{k}h(\al)}{h_{0}} 
= \frac{\tau}{nd}\sumj \tr\left\{ \p_{\al}^{k}S_{j-1}^{-1}(\al) \cdot S_{j-1}(\al_{0}) \right\} + O_{p}^{\ast}\big(n^{-1/2}\big)
\label{hm:an.2-1}
\end{equation}
for each $k\in\{0,1,2\}$;
moreover, it holds that $\sqrt{n}(h(\bar{\al}_{n})/h_{0}-\tau)=O_{p}(1)$ for any $\sqrt{n}$-consistent estimator $\bar{\al}_{n}$ of $\al$.
Letting $\dot{S}_{1,j-1}(\al):=\p_{\al}S_{j-1}(\al)[u_{1}]$ and $\ddot{S}_{1,j-1}(\al):=\p_{\al}^{2}S_{j-1}(\al)[u_{1}^{\otimes 2}]$ and substituting the expression \eqref{hm:an.2-1}, we see that
\begin{align}
& -\frac{1}{n}\p_{\al}^{2}\tilde{\mbbh}_{n}(\rho_{n})[u_{1}^{\otimes 2}] \nn\\
&= \frac{1}{2}\bigg[\frac{1}{n}\sumj\tr(-S_{j-1}^{-1}\dot{S}_{1,j-1}S^{-1}_{j-1}\dot{S}_{1,j-1} + S^{-1}_{j-1}\ddot{S}_{1,j-1})(\rho_{\al,n})
\nn\\
&\quad +d \bigg\{\frac{1}{h(\rho_{\al,n})}\p_{\al}^{2}h(\rho_{\al,n})[u_{1}^{\otimes 2}] 
- \bigg(\frac{1}{h(\rho_{\al,n})}\p_{\al}h(\rho_{\al,n})[u_{1}]\bigg)^{2}\bigg\}\bigg]
+o_{p}(1) \nn\\
&= \frac{1}{2}\bigg\{\frac{1}{n}\sumj\tr(-S_{j-1}^{-1}\dot{S}_{1,j-1}S^{-1}_{j-1}\dot{S}_{1,j-1} + S^{-1}_{j-1}\ddot{S}_{1,j-1})(\rho_{\al,n})
\nn\\
&\quad + \frac{1}{n}\sumj \tr\big(\p_{\al}^{2}S_{j-1}^{-1}(\rho_{\al,n})[u_{1}^{\otimes 2}] S_{j-1}(\al_{0})\big)
- d\bigg( \frac{1}{dn}\sumj \tr\big(\dot{S}_{1,j-1} S_{j-1}(\al_{0}) \bigg)^{2}\bigg\} +o_{p}(1) \nn\\
&= \frac{1}{2n}\sumj\tr(S_{j-1}^{-1}\dot{S}_{1,j-1}S^{-1}_{j-1}\dot{S}_{1,j-1})(\al_{0})
- \frac{1}{2d}\bigg( \frac{1}{n}\sumj \tr\big(\dot{S}_{1,j-1} S_{j-1}\big)(\al_{0}) \bigg)^{2} +o_{p}(1) \nn\\
&\cip \tilde{\Gam}_{1,0}[u_{1}^{\otimes2}]. \nn
\end{align}
The proof of \eqref{hm:an.1-2} is complete.

\medskip

\paragraph{{\it Proof of Theorem \ref{hm:thm2.diffusion.an}}}

The convergence \eqref{hm:diffusion.an1+} follows on showing that
\begin{equation}
\left( \tilde{u}^{h}_{n}, \tilde{u}^{\al}_{n}, \tilde{u}^{\beta}_{n} \right) :=
\bigg(\sqrt{n}\bigg(\frac{\tilde{h}}{\tau h_{0}}-1\bigg),\, \sqrt{n}(\tilde{\alpha}_{n}-\alpha_{0}),\, \sqrt{n\tau h_{0}}(\tilde{\beta}_{n}-\beta_{0})\bigg)
\cil N_{1+p}\left(0,\Sigma(\tz)\right),
\label{hm:thm2.diffusion.an-1}
\end{equation}
since we then have $\sqrt{n\tilde{h}}=\sqrt{n\tau h_{0} + O_{p}(\sqrt{n}h_{0})}=\sqrt{n\tau h_{0}} + o_{p}(1)$, ensuring that we may replace ``$\tau h_{0}$'' in $\sqrt{n\tau h_{0}}(\tilde{\beta}_{n}-\beta_{0})$ by $\tilde{h}$.

Let
\begin{equation}
\tilde{\D}^{\al}_{n} := -\frac{1}{2\sqrt{n}}\sumj(\zeta_{1,j}+\zeta_{2,j}), \qquad 
\tilde{\D}^{\beta}_{n} := \frac{1}{\sqrt{n}}\sumj \eta_{j}, \qquad \tilde{\D}^{h}_{n} := \frac{1}{\sqrt{n}}\sumj \ep_{j},
\nonumber
\end{equation}
where $\ep_{j}:=d^{-1}| h_{0}^{-1/2}\D_{j}w|^{2}-1$.
Proceeding as in \eqref{hm:an.1-3-2} with expanding $S_{j-1}^{-1}(\aet)$ around $\al_{0}$, we can deduce that
\begin{equation}
\tilde{u}^{h}_{n}
= \tilde{\D}^{h}_{n} - K[\tilde{u}^{\al}_{n}] +o_{p}(1).
\nn
\end{equation}
We have seen in the proof of \eqref{hm:diffusion.an1} that $(\tilde{u}^{\al}_{n},\, \tilde{u}^{\beta}_{n}) = \diag\big( \tilde{\Gam}_{1,0},\, \tilde{\Gam}_{2,0}\big)^{-1}\big[(\tilde{\D}^{\al}_{n}, \, \tilde{\D}^{\beta}_{n})\big]+ o_{p}(1)$. Hence
\begin{equation}
\left( \tilde{u}^{h}_{n}, \tilde{u}^{\al}_{n}, \tilde{u}^{\beta}_{n} \right)
=\left(
\begin{array}{ccc}
1 & - K^{\top}\tilde{\Gam}_{1,0}^{-1} & 0 \\
0 & \tilde{\Gam}_{1,0}^{-1} & 0 \\
0 & 0 & \tilde{\Gam}^{-1}_{2,0}
\end{array}
\right)
\big[ \big( \tilde{\D}^{h}_{n}, \tilde{\D}^{\al}_{n}, \tilde{\D}^{\beta}_{n} \big) \big] + o_{p}(1).
\label{hm:thm2.diffusion.an-3}
\end{equation}
It is easy to see that for $\del:=h_{0}^{-1/2}\D_{1}w\sim N_{d}(0,I_{d})$,
\begin{align}
\E_{j-1}(\ep_{j}\eta_{j}) &= 0, \nn\\
\E_{j-1}(\ep_{j}\zeta_{1,j}) &= \tr \bigg\{ \left\{ a_{j-1}^{\top}(\p_{\al}S_{j-1}^{-1})a_{j-1}\right\}\bigg(
\frac{1}{d}\E\left(\del^{\otimes 2}|\del|^{2}\right)-I_{d}\bigg)\bigg\} \nn\\
&=\frac{2}{d}\tr \left\{ a_{j-1}^{\top}(\p_{\al}S_{j-1}^{-1})a_{j-1}\right\}
=-\frac{2}{d}\tr \left\{ S_{j-1}^{-1}(\p_{\al}S_{j-1})\right\},
\nn\\
\E_{j-1}(\ep_{j}\zeta_{2,j}) &=dK \E_{j-1}(|\del|^{2}) = 2K.
\nonumber
\end{align}
Then, applying the martingale central limit theorem we conclude that
\begin{equation}
\big( \tilde{\D}^{h}_{n}, \tilde{\D}^{\al}_{n}, \tilde{\D}^{\beta}_{n} \big) \cil N_{1+p}\left(0,\, 
\left(
\begin{array}{ccc}
2/d & & \text{{\rm sym.}} \\
0 & \tilde{\Gam}_{1,0} &  \\
0 & 0 & \tilde{\Gam}_{2,0}
\end{array}
\right)\right)
\label{hm:thm2.diffusion.an-4}
\end{equation}
Combining \eqref{hm:thm2.diffusion.an-3} and \eqref{hm:thm2.diffusion.an-4} leads to \eqref{hm:thm2.diffusion.an-1}.

Finally, the convergences $\tilde{K}_{n} \cip K$, $\tilde{\Gam}_{1,n} \cip \tilde{\Gam}_{1,0}$, and $\tilde{\Gam}_{2,n} \cip \tilde{\Gam}_{2,0}$ are direct consequences of the uniform law of large numbers.

\subsection{Proof of Theorem \ref{hm:thm_stepwise.ergo.diff}}

Building on the proof of the consistency of $\aet$, we see that 
\begin{equation}
\tilde{\mbby}_{n}^{1}(\alpha,\beta) = \frac{1}{n}\left(\tilde{\mbbh}_{1,n}(\alpha)-\tilde{\mbbh}_{1,n}(\alpha_{0})\right)+o_{p}^{\ast}(1).
\nonumber
\end{equation}
In Section \ref{hm:proof.consistency} we saw that $|\tilde{\mbby}_{n}^{1}(\alpha,\beta) - \tilde{\mbby}_{0}^{1}(\al)| = o_{p}^{\ast}(1)$,
hence $| n^{-1}\{\tilde{\mbbh}_{1,n}(\alpha)-\tilde{\mbbh}_{1,n}(\alpha_{0})\} - \tilde{\mbby}_{0}^{1}(\al)| = o_{p}^{\ast}(1)$ as well.
This combined with the fact that the function $\tilde{\mbby}_{0}^{1}(\al)$ is free from $\beta$ leads to the the consistency $\aet'\cip\al_{0}$.
The consistency of $\bet'$ is completely the same as in the proof of $\bet\cip\beta_{0}$. It follows that $\tet^{\prime}\cip\tz$.

Combining \eqref{hm:an.1-3-1}, the convergence in probability of $-n^{-1}\p_{\al}^{2}\tilde{\mbbh}_{n}(\rho_{n})$ for any $\rho_{n}\cip\tz$ which we have seen at the end of the proof of \eqref{hm:an.1-2}, and \eqref{hm:an.1-3}, we deduce that
\begin{equation}
\sqrt{n}(\aet'-\al_{0}) = -\tilde{\Gam}_{1,0}^{-1}\frac{1}{2\sqrt{n}}\sumj(\zeta_{1,j}+\zeta_{2,j}) + o_{p}(1).
\label{hm:2step-1}
\end{equation}
In particular, the asymptotic normality $\sqrt{n}(\aet'-\al_{0}) \cil N_{p_{\al}}(0,\tilde{\Gam}_{1,0}^{-1})$ follows without reference to structure of the drift.
Since \eqref{hm:2step-1} ensures the tightness of $\{\sqrt{n}(\aet'-\al_{0})\}_{n}$, we can show the equations $\tilde{h}^{\prime}=\tau h_{0}+o_{p}(h_{0}/\sqrt{n})$ and
\begin{align}
\sqrt{n}\left(\frac{\tilde{h}^{\prime}}{\tau h_{0}}-1\right)&= \frac{1}{\sqrt{n}}\sumj \ep_{j}-K\big[\sqrt{n}(\aet'-\al_{0})\big]+o_{p}(1). \label{se:2step-3}
\end{align}
Turning to $\beta$, we may suppose that $\p_{\beta}\tilde{\mbbh}_{2,n}(\tet')=0$, the probability of which tends to $1$ by the consistency of $\tet'$.
Then we have the expansion
\begin{align*}
\frac{1}{\sqrt{n\tau h_{0}}}\p_{\beta}\tilde{\mbbh}_{2,n}(\aet',\beta_{0})&=-\frac{1}{n\tau h_{0}}\p_{\beta}^{2}\tilde{\mbbh}_{2,n}(\aet', \rho'_{\beta,n}) \big[\sqrt{n\tau h_{0}}(\bet'-\beta_{0})\big] \\
&=-\frac{1}{nh_{0}}\p_{\beta}^{2}\tilde{\mbbh}_{2,n}(\aet', \rho'_{\beta,n}) \left[\sqrt{n\tilde{h}^{\prime}}(\bet'-\beta_{0})\right]+o_{p}(1),
\end{align*}
where $\rho'_{\beta,n}$ is a random point on the segment joining $\bet'$ and $\beta_{0}$.
As in \eqref{hm:an.2-1-1} we have
\begin{equation}
-\frac{1}{nh_{0}}\p_{\beta}^{2}\tilde{\mbbh}_{2,n}(\aet', \rho'_{\beta,n}) = \tilde{\Gam}_{2,0}+o_{p}(1).
\nonumber
\end{equation}
Further, the tightness of $\{\sqrt{n}(\aet'-\al_{0})\}_{n}$ and \eqref{hm:an.1-4} imply the equation
\begin{align}
\frac{1}{\sqrt{n\tau h_{0}}}\p_{\beta}\tilde{\mbbh}_{2,n}(\aet',\beta_{0})
&=\frac{1}{\sqrt{n\tau h_{0}}}\p_{\beta}\tilde{\mbbh}_{2,n}(\tz) + \frac{1}{\sqrt{n\tau h_{0}}}\p_{\al}\p_{\beta}\tilde{\mbbh}_{2,n}(\rho'_{\al,n},\beta_{0})[\aet'-\al_{0}] \nn\\
&=\frac{1}{\sqrt{n}}\sumj \eta_{j} + O_{p}(\sqrt{n}h_{0})+O_{p}(\sqrt{h_{0}}) \nn\\
&=\frac{1}{\sqrt{n}}\sumj \eta_{j} + o_{p}(1).
\nonumber
\end{align}
Piecing together these observations we deduce that
\begin{equation}
\sqrt{n\tilde{h}^{\prime}}(\bet'-\beta_{0}) = \tilde{\Gam}^{-1}_{2,0}\frac{1}{\sqrt{n}}\sumj \eta_{j} + o_{p}(1).
\label{hm:2step-2}
\end{equation}
Having \eqref{hm:2step-1}, \eqref{se:2step-3}, and \eqref{hm:2step-2} in hand, we can derive the convergence \eqref{hm:diffusion.an1_2step} in the same way as in the proof of Theorem \ref{hm:thm2.diffusion.an}.

\subsection{Proof of Theorem \ref{se:thm.pldi}}
\label{se:proof.pldi}

The PLDI for ergodic diffusion model with $h=h_{0}$ ($\tau=1$) being known has been derived \cite[Section 6]{Yos11}.
However, the same scenario would not go through in our proof without additional considerations because the random function $\tilde{\mbbh}_{n}(\theta)$ is different from the original GQLF $\mbbh_{n}(\theta;h)$.

\subsubsection{$g_{q}$-exponential ergodicity}
\label{hm:proof.pldi-g.exp.ergo}

Let $\{P_{t}(x,dy)\}_{t\in\mbbrp}$ denote the family of the transition functions of $X$. 
Given a function $\rho:\mbbr^{d}\to\mbbrp$ and a signed measure $m$ on the $d$-dimensional Borel space, 
we define
\begin{equation}
\|m\|_{\rho}=\sup\left\{\left|\int f(y)m(dy)\right|:\, \text{$f$ is $\mbbr$-valued and measurable, such that $|f|\le\rho$}\right\}.
\nn
\end{equation}

\begin{lem}
Under Assumption \ref{Ass1+}, the following statements hold.
\begin{enumerate}
\item There exist a probability measure $\pi_{0}$ and a nonnegative $\mcc^{2}$ function $g$ such that
\begin{equation}
\limsup_{|x|\to\infty}\frac{1+\|x\|^{q}}{g(x)} = 0
\nonumber
\end{equation}
for every $q>0$, and that
\begin{equation}
\|P_{t}(x,\cdot)-\pi_{0}(\cdot)\|_{g_{q}}\lesssim e^{-at}g_{q}(x),\quad x\in\mbbr^{d}
\label{hm:Ass:pldi1-1}
\end{equation}
for some constant $a>0$.
\item $\sup_{t}\E(\|X_{t}\|^{q})<\infty$ for every $q>0$.
\end{enumerate}
\label{hm:lem_pldi-g.exp.ergo}
\end{lem}

\begin{proof}
(1)
In view of the general results developed in \cite[Section 6]{MeyTwe-III}, under Assumption \ref{Ass1+} it suffices for \eqref{hm:Ass:pldi1-1} to show the following:
(i) every compact sets in $\mbbr^{d}$ are \textit{petite} for the Markov chain $(X_{t_{j}})_{j=0}^{n}$ for any $h>0$ small enough;
(ii) the drift condition
\begin{equation}
\mathcal{A}_{\theta} g(x) \le -c_{1}' g(x) + c'_{2}, \qquad x\in\mbbr^{d},
\nonumber
\end{equation}
holds for some $c_{1}', c'_{2}>0$ and a nonnegative $\mcc^{2}$ function $g$ such that $g(x)\to\infty$ as $\|x\|\to\infty$ faster than any polynomial,
where $\mathcal{A}_{\theta}$ denotes the infinitesimal generator of $X$: with writing $b=(b_{k})$ and $S=(S_{kl})$,
\begin{equation}
\mathcal{A}_{\theta} g(x) := \sum_{k}b_{k}(x,\theta)\p_{x_k}g(x) + \frac{1}{2}\sum_{k}\sum_{l}S_{kl}(x,\al)\p_{x_k}\p_{x_l}g(x).
\label{hm:generator}
\end{equation}
Indeed, both conditions follow from \cite[Propositions 1.1, 1.2 and 5.1]{Gob02}:
under Assumption \ref{Ass1+}, $X$ admits a transition density which is positive for every $(x,y,h)\in\mbbr^{d}\times\mbbr^{d}\times(0,1]$, ensuring the topological condition (i)
(see the bounds \eqref{hm:gobet.bound}  below, for which we note that the $\mcc^{1+\gam}$-property assumed in the condition (R)-1 in \cite{Gob02} is not necessary);
the drift condition (ii) can be derived for $g$ equaling $\exp(c''\|x\|^{2})$ or $\exp(c''\|x\|)$ for some $c''>0$ outside a neighborhood of the origin in case of Assumption \ref{Ass1+}(2)(a) or (b), respectively.

\medskip

(2) follows from a standard application of the drift condition. See \cite[Propositions 1.1(1) and 5.1(1)]{Gob02}.
\end{proof}

\subsubsection{Bounding inverse moment}

The next lemma will be used to deduce some moment bounds required later.

\begin{lem}
Under Assumption \ref{Ass1+}, for every $q>0$ we have
\begin{align*}
\sup_{n>2q/d}\mbbe\bigg( \sup_{\alpha\in\overline{\Theta}_{\alpha}}\bigg| \frac{1}{nh_{0}}\sumj S_{j-1}^{-1}(\alpha)\big[(\D_{j}X)^{\otimes2}\big] \bigg|^{-q} \bigg)<\infty.
\end{align*}
\label{hm:lem_inv.moment}
\end{lem}

\begin{proof}
Write $\zeta_{j} = \|h_{0}^{-1/2}\D_{j}X\|^{2}$.
First, observe that the expectation can be bounded from above by
\begin{align}
& \bigg( \sup_{(x,\al)}\lam_{\max}^{q}\left(S(x,\al)\right) \bigg) \E\bigg\{\bigg( \frac{1}{n}\sumj \left\| h_{0}^{-1/2}\D_{j}X \right\|^{2} \bigg)^{-q}\bigg\}
\lesssim \E\bigg\{\bigg( \frac{1}{n}\sumj \zeta_{j} \bigg)^{-q}\bigg\}.
\label{hm:lem_inv.moment-1}
\end{align}
Under Assumption \ref{Ass1+}, we have the following modified Aronson type bound with possibly unbounded drift coefficient for the transition density of $X$,
say $p_{h_{0}}(x,y)$ ($\pr(X_{h_{0}}\in dy|X_{0}=x)=p_{h_{0}}(x,y)dy$):
there exist constants $A,B>1$ and $B_{0}\ge 0$ such that for every $(x,y,h_{0})\in\mbbr^{d}\times\mbbr^{d}\times(0,1]$,
\begin{equation}
\frac{1}{Ah_{0}^{d/2}}e^{-B_{0}h_{0}\|x\|^{2}}\exp\bigg( -\frac{B\|y-x\|^{2}}{h_{0}} \bigg) \le p_{h_{0}}(x,y) 
\le \frac{A}{h_{0}^{d/2}}e^{B_{0}h_{0}\|x\|^{2}}\exp\bigg( -\frac{\|y-x\|^{2}}{h_{0}B} \bigg),
\label{hm:gobet.bound}
\end{equation}
where we can take $B_{0}=0$ especially when the drift $b$ is bounded (see \cite[Proposition 1.2]{Gob02} for details).
The bound \eqref{hm:gobet.bound} implies that the conditional distribution of $h_{0}^{-1/2}\D_{j}X$ given $X_{t_{j-1}}=x$,
which we denote by $y\mapsto \bar{p}_{h_{0}}(y|x)$, satisfies that
\begin{equation}
A^{-1}e^{-B_{0}h_{0}\|x\|^{2}} \exp(-B\|y\|^{2}) \le \bar{p}_{h_{0}}(y|x) \le A e^{B_{0}h_{0}\|x\|^{2}} \exp( -\|y\|^{2}/B).
\nonumber
\end{equation}
It follows that
\begin{equation}
\sup_{y} \bar{p}_{h_{0}}(y|x) \lesssim \exp\left(h_{0} B_{0}\|x\|^{2}\right).
\label{hm:gobet.bound+}
\end{equation}
In what follows, the constant $B_{0}\ge 0$ may change at each appearance, with keeping the rule that $B_{0}=0$ if $b$ is bounded (that is, if the Assumption \ref{Ass1+}(2)(b) holds).

Let $k\in\mbbn$. We will prove that
\begin{align}
\sup_{i\ge 0}\pr\bigg( \sum_{j=1}^{k}\zeta_{i+j} \le \ep \bigg) \lesssim \ep^{kd/2}, \qquad \ep>0.
\label{hm:lem_inv.moment-3}
\end{align}
To this end, we make use the argument of the proof of \cite[Eq.(2.2)]{BhaPap91}, while our conditions are apparently weaker.
Write $\pr_{l}(\cdot)$ for the conditional expectation given $\mcf_{t_{l}}$.
Observe that from \eqref{hm:gobet.bound+} we have a.s.
\begin{align}
\pr_{i+k-1}(\zeta_{i+k}\le\ep) &= \int_{\|y\|\le\sqrt{\ep}} \bar{p}_{h_{0}}(y|X_{t_{i+k-1}})dy
\lesssim \ep^{d/2} \exp\left( h_{0}B_{0}\|X_{t_{i+k-1}}\|^{2}\right).
\nonumber
\end{align}
This in turn implies that
\begin{align}
& I(\zeta_{i+k-1} \le \ep) \pr_{i+k-1}(\zeta_{i+k}\le\ep) \nn\\
&\lesssim \ep^{d/2} I(\zeta_{i+k-1} \le \ep)I\left( \|X_{t_{i+k-1}}\| \le \sqrt{\ep h_{0}} + \|X_{t_{i+k-2}}\| \right) \exp\left( h_{0}B_{0}\|X_{t_{i+k-1}}\|^{2}\right) \nn\\
&\lesssim \ep^{d/2} I(\zeta_{i+k-1} \le \ep) \exp\left( h_{0}B_{0}\|X_{t_{i+k-2}}\|^{2}\right).
\nonumber
\end{align}
Iterating the same manner along with taking the conditional expectations successively, we can deduce
\begin{align}
\pr\bigg( \sum_{j=1}^{k}\zeta_{i+j} \le \ep \bigg) &\le \E\bigg( \prod_{j=1}^{k}I(\zeta_{i+j} \le \ep) \bigg) \nn\\
&= \E\bigg\{ \bigg( \prod_{j=1}^{k-1}I(\zeta_{i+j} \le \ep) \bigg) \pr_{i+k-1}(\zeta_{i+k}\le\ep)\bigg\} \nn\\
&\lesssim \ep^{d/2} \E\bigg\{ \bigg( \prod_{j=1}^{k-2}I(\zeta_{i+j} \le \ep) \bigg) I(\zeta_{i+k-1} \le \ep) \exp\left( h_{0}B_{0}\|X_{t_{i+k-2}}\|^{2}\right) \bigg\} \nn\\
&\lesssim \ep^{d/2} \E\bigg\{ \bigg( \prod_{j=1}^{k-2}I(\zeta_{i+j} \le \ep) \bigg) \exp\left( h_{0}B_{0}\|X_{t_{i+k-2}}\|^{2}\right) \pr_{i+k-2}(\zeta_{i+k-1}\le\ep) \bigg\} \nn\\
&\lesssim (\ep^{d/2})^{2} \E\bigg\{ \bigg( \prod_{j=1}^{k-2}I(\zeta_{i+j} \le \ep) \bigg) \exp\left( h_{0}B_{0}\|X_{t_{i+k-3}}\|^{2}\right) \bigg\} \nn\\
&\lesssim \dots\lesssim \ep^{(k-1)d/2} \, \E\bigg( I(\zeta_{i+1} \le \ep) \exp\left( h_{0}B_{0}\|X_{t_{i}}\|^{2}\right) \bigg) \nn\\
&\lesssim \ep^{kd/2} \sup_{t}\E\left\{ \exp\left( h_{0}B_{0}\|X_{t}\|^{2}\right) \right\} \lesssim \ep^{kd/2},
\nn
\end{align}
the last estimate holding for every $h_{0}$ small enough; again note that we can take $B_{0}=0$ when $b$ is bounded.
Thus we have verified the estimate \eqref{hm:lem_inv.moment-3}, so that
\begin{equation}
\sup_{i\ge 0}\pr\bigg\{\bigg( \sum_{j=1}^{k}\zeta_{i+j} \bigg)^{-q} \ge r \bigg\} \lesssim r^{-kd/(2q)}, \qquad r>0.
\nonumber
\end{equation}
Now, letting $k>2q/d$ we obtain
\begin{equation}
\sup_{i\ge 0} \E\bigg\{\bigg( \sum_{j=1}^{k}\zeta_{i+j} \bigg)^{-q} \bigg\} 
\le 1+\int_{1}^{\infty}r^{-kd/(2q)}dr 
\lesssim 1.
\label{hm:lem_inv.moment-2}
\end{equation}
Having \eqref{hm:lem_inv.moment-2} in hand, we can complete the proof in a similar manner to \cite[Lemma A.1]{FinWei02}.
Let $m:=[n/k]$. Obviously it suffices to consider $n=km$ ($m\in\mbbn$). Write $\lam_{l}=\sum_{j=k(l-1)+1}^{kl}\zeta_{j}$; then $\sumj\zeta_{j}=\sum_{l=1}^{m}\lam_{l}$, and \eqref{hm:lem_inv.moment-2} ensures that $\sup_{l}\E( \lam_{l}^{-q} )\lesssim 1$. By the convexity of the mapping $s\mapsto s^{-q}$ ($s>0$), Jensen's inequality gives
\begin{align}
\E\bigg\{\bigg( \frac{1}{n}\sumj \zeta_{j} \bigg)^{-q}\bigg\}
&= k^{q}\,\E\bigg\{\bigg( \frac{1}{m}\sum_{l=1}^{m} \lam_{l} \bigg)^{-q}\bigg\} \lesssim k^{q},
\nonumber
\end{align}
which combined with \eqref{hm:lem_inv.moment-1} completes the proof.
\end{proof}

\medskip

Now, in order to show \eqref{pldi1} and \eqref{pldi2}, it is enough to check [A1$''$], [A4$'$], and [A6] of \cite{Yos11}; the conditions [B1] and [B2] therein are given in Sections \ref{proof.AN.an} and \ref{proof:thm.bic}, respectively.
The claim \eqref{se:thm.pldi-1} is trivial from the \eqref{pldi1}, \eqref{pldi2}, and the definition of $\tilde{h}'$.
We put $\epsilon_{1}=\epsilon_{0}/2$ in the sequel.

\subsubsection{Proof of \eqref{pldi1}}

We will verify the following conditions: for every $M>0$,
\begin{align}
&\mbbe\bigg(\sup_{\beta}\bigg\|\frac{1}{\sqrt{n}}\p_{\alpha}\tilde{\mbbh}_{n}(\alpha_{0},\beta)\bigg\|^{M}\bigg)<\infty; \label{pldi.al1} \\
&\mbbe\bigg\{\bigg(\sup_{\theta}n^{\epsilon_{1}}\big|\tilde{\mbby}_{n}^{1}(\alpha,\beta)-\tilde{\mbby}_{0}^{1}(\alpha)\big|\bigg)^{M}\bigg\}<\infty; \label{pldi.al2} \\
&\mbbe\bigg\{\bigg(\frac{1}{n}\sup_{\theta}\big\|\p_{\alpha}^{3}\tilde{\mbbh}_{n}(\theta)\big\|\bigg)^{M}\bigg\}<\infty; \label{pldi.al3} \\
&\mbbe\bigg\{\sup_{\beta}\bigg(n^{\epsilon_{1}}\bigg\|-\frac{1}{n}\p_{\alpha}^{2}\tilde{\mbbh}_{n}(\alpha_{0},\beta)-\tilde{\Gam}_{1,0}\bigg\|\bigg)^{M}\bigg\}<\infty. \label{pldi.al4}
\end{align}
The conditions \eqref{pldi.al1} to \eqref{pldi.al4} imply [A1$''$] and [A6].
The left-hand side of \eqref{pldi.al2} satisfies
\begin{align}
&\mbbe\bigg\{\bigg(\sup_{\theta}n^{\epsilon_{1}}\big|\tilde{\mbby}_{n}^{1}(\alpha,\beta)-\tilde{\mbby}_{0}^{1}(\alpha)\big|\bigg)^{M}\bigg\} \nn\\
&\lesssim\mbbe\bigg\{\bigg(\sup_{\theta}n^{\epsilon_{1}}\bigg|\frac{1}{2n}\sumj\log\big|S_{j-1}^{-1}(\alpha)S_{j-1}(\alpha_{0})\big|-\frac{1}{2}\int_{\mbbr^{d}}\log\big|S^{-1}(x,\alpha)S(x,\alpha_{0})\big|\pi(dx)\bigg|\bigg)^{M}\bigg\} \nn\\
&\quad+\mbbe\bigg\{\bigg(\sup_{\alpha}n^{\epsilon_{1}}\bigg|\log\frac{h(\alpha)/h_{0}}{\frac{\tau}{d}\int_{\mbbr^{d}}\tr\big(S^{-1}(x,\alpha)S(x,\alpha_{0})\big)\pi(dx)}\bigg|\bigg)^{M}\bigg\} \nn\\
&\quad+\mbbe\bigg\{\bigg(\sup_{\theta}n^{\epsilon_{1}}\bigg|\frac{1}{n}\sumj S_{j-1}^{-1}(\alpha)[\D_{j}X,b_{j-1}(\theta)]\bigg|\bigg)^{M}\bigg\}+\mbbe\bigg\{\bigg(\sup_{\theta}n^{\epsilon_{1}}\bigg|\frac{1}{2}h(\alpha)\bigg(\frac{1}{n}\sumj S_{j-1}^{-1}[b_{j-1}(\theta)^{\otimes2}]\bigg)\bigg|\bigg)^{M}\bigg\} \nn\\
&\lesssim\mbbe\bigg\{\bigg(\sup_{\theta}n^{\epsilon_{1}}\bigg|\frac{1}{2n}\sumj\log\big|S_{j-1}^{-1}(\alpha)S_{j-1}(\alpha_{0})\big|-\frac{1}{2}\int_{\mbbr^{d}}\log\big|S^{-1}(x,\alpha)S(x,\alpha_{0})\big|\pi(dx)\bigg|\bigg)^{M}\bigg\} \nn\\
&\quad+\mbbe\bigg[\bigg\{\sup_{\alpha}n^{\epsilon_{1}}\bigg|\frac{h(\alpha)/h_{0}}{\frac{\tau}{d}\int_{\mbbr^{d}}\tr\big(S^{-1}(x,\alpha)S(x,\alpha_{0})\big)\pi(dx)}\bigg|^{-1}\times
\bigg|\frac{h(\alpha)/h_{0}}{\frac{\tau}{d}\int_{\mbbr^{d}}\tr\big(S^{-1}(x,\alpha)S(x,\alpha_{0})\big)\pi(dx)}-1\bigg|\bigg\}^{M}\bigg] \nn\\
&\quad+\mbbe\bigg[\bigg\{\sup_{\alpha}n^{\epsilon_{1}}\bigg|\frac{h(\alpha)/h_{0}}{\frac{\tau}{d}\int_{\mbbr^{d}}\tr\big(S^{-1}(x,\alpha)S(x,\alpha_{0})\big)\pi(dx)}-1\bigg|\bigg\}^{M}\bigg] \nn\\
&\quad+1 \nn\\
&\lesssim\mbbe\bigg\{\bigg(\sup_{\theta}n^{\epsilon_{1}}\bigg|\frac{1}{2n}\sumj\log\big|S_{j-1}^{-1}(\alpha)S_{j-1}(\alpha_{0})\big|-\frac{1}{2}\int_{\mbbr^{d}}\log\big|S^{-1}(x,\alpha)S(x,\alpha_{0})\big|\pi(dx)\bigg|\bigg)^{M}\bigg\} \nn\\
&\quad+\mbbe\bigg[\bigg\{\sup_{\alpha}n^{\epsilon_{1}}\bigg|\frac{h(\alpha)}{h_{0}}\bigg|^{-1}\bigg|\frac{h(\alpha)}{h_{0}}-\frac{\tau}{d}\int_{\mbbr^{d}}\tr\big(S^{-1}(x,\alpha)S(x,\alpha_{0})\big)\pi(dx)\bigg|\bigg\}^{M}\bigg] \nn\\
&\quad+\mbbe\bigg[\bigg\{\sup_{\alpha}n^{\epsilon_{1}}\bigg|\frac{h(\alpha)/h_{0}}{\frac{\tau}{d}\int_{\mbbr^{d}}\tr\big(S^{-1}(x,\alpha)S(x,\alpha_{0})\big)\pi(dx)}-1\bigg|\bigg\}^{M}\bigg] \nn\\
&\quad+1 \nn\\
&\leq\mbbe\bigg\{\bigg(\sup_{\theta}n^{\epsilon_{1}}\bigg|\frac{1}{2n}\sumj\log\big|S_{j-1}^{-1}(\alpha)S_{j-1}(\alpha_{0})\big|-\frac{1}{2}\int_{\mbbr^{d}}\log\big|S^{-1}(x,\alpha)S(x,\alpha_{0})\big|\pi(dx)\bigg|\bigg)^{M}\bigg\} \nn\\
&\quad+\mbbe\bigg(\sup_{\alpha}\bigg|\frac{h(\alpha)}{h_{0}}\bigg|^{-2M}\bigg)^{1/2}\mbbe\bigg\{\bigg(\sup_{\alpha}n^{\epsilon_{1}}\bigg|\frac{h(\alpha)}{h_{0}}-\frac{\tau}{d}\int_{\mbbr^{d}}\tr\big(S^{-1}(x,\alpha)S(x,\alpha_{0})\big)\pi(dx)\bigg|\bigg)^{2M}\bigg\}^{1/2} \nn\\
&\quad+\mbbe\bigg[\bigg\{\sup_{\alpha}n^{\epsilon_{1}}\bigg|\frac{h(\alpha)/h_{0}}{\frac{\tau}{d}\int_{\mbbr^{d}}\tr\big(S^{-1}(x,\alpha)S(x,\alpha_{0})\big)\pi(dx)}-1\bigg|\bigg\}^{M}\bigg] \nn\\
&\quad+1, \nn
\end{align}
where in the second step we used \eqref{se:aux.4}, Lemma \ref{hm:lem_pldi-g.exp.ergo}(2), 
the inequality $|\log x|\le (1+ |x|^{-1})|x-1|$ for $x>0$ for the second term,
and H\"{o}lder's inequality for the third term and fourth term.
As in \cite[Lemma 4.3]{Mas13}, Lemma \ref{hm:lem_pldi-g.exp.ergo} ensures
\begin{align*}
&\mbbe\bigg\{\bigg(\sup_{\theta}n^{\epsilon_{1}}\bigg|\frac{1}{2n}\sumj\log\big|S_{j-1}^{-1}(\alpha)S_{j-1}(\alpha_{0})\big|-\frac{1}{2}\int_{\mbbr^{d}}\log\big|S^{-1}(x,\alpha)S(x,\alpha_{0})\big|\pi(dx)\bigg|\bigg)^{M}\bigg\}<\infty, \\
&\mbbe\bigg[\bigg\{\sup_{\alpha}n^{\epsilon_{1}}\bigg|\frac{h(\alpha)/h_{0}}{\frac{\tau}{d}\int_{\mbbr^{d}}\tr\big(S^{-1}(x,\alpha)S(x,\alpha_{0})\big)\pi(dx)}-1\bigg|\bigg\}^{M}\bigg]<\infty,\\
&\mbbe\bigg\{\bigg(\sup_{\alpha}n^{\epsilon_{1}}\bigg|\frac{h(\alpha)}{h_{0}}-\frac{\tau}{d}\int_{\mbbr^{d}}\tr\big(S^{-1}(x,\alpha)S(x,\alpha_{0})\big)\pi(dx)\bigg|\bigg)^{2M}\bigg\}<\infty.
\end{align*}
Further, Lemma \ref{hm:lem_inv.moment} implies
\begin{align*}
\mbbe\left(\sup_{\alpha}\left|\frac{h(\alpha)}{h_{0}}\right|^{-2M}\right)<\infty.
\end{align*}
Hence, \eqref{pldi.al2} is established.
In a similar way, we have
\begin{align*}
&\mbbe\bigg(\sup_{\beta}\bigg\|\frac{1}{\sqrt{n}}\p_{\alpha}\tilde{\mbbh}_{n}(\alpha_{0},\beta)\bigg\|^{M}\bigg) \\
&\lesssim1+\mbbe\bigg(\bigg\|\frac{1}{2\sqrt{n}}\sumj\tr\Big(S_{j-1}^{-1}(\alpha_{0})\big(\p_{\alpha}S_{j-1}(\alpha_{0})\big)\Big)+\frac{1}{h(\alpha_{0})}\frac{1}{2\sqrt{n}}\sumj \p_{\alpha}S_{j-1}^{-1}(\alpha_{0})\big[(\D_{j}X)^{2}\big]\bigg\|^{M}\bigg) \\
&\lesssim1+\mbbe\bigg\{\bigg\|\frac{h_{0}}{h(\alpha_{0})}\frac{1}{2\sqrt{n}}\sumj\tr\Big(S_{j-1}^{-1}(\alpha_{0})\big(\p_{\alpha}S_{j-1}(\alpha_{0})\big)\Big)\bigg(\frac{h(\alpha_{0})}{h_{0}}-\tau\bigg)\bigg\|^{M}\bigg\} \\
&\leq 1+\mbbe\bigg(\bigg|\frac{h_{0}}{h(\alpha_{0})}\bigg|^{3M}\bigg)^{1/3}\mbbe\bigg(\bigg\|\frac{1}{2n}\sumj\tr\Big(S_{j-1}^{-1}(\alpha_{0})\big(\p_{\alpha}S_{j-1}(\alpha_{0})\big)\Big)\bigg\|^{3M}\bigg)^{1/3} \\
&\quad\times\mbbe\bigg(\bigg|\frac{1}{\sqrt{n}h_{0}d}\sumj S_{j-1}^{-1}(\alpha_{0})\big[(\D_{j}X)^{\otimes2}-\tau h_{0} S_{j-1}(\alpha_{0})\big]\bigg|^{3M}\bigg)^{1/3} \\
&<\infty, \\
&\mbbe\bigg\{\sup_{\beta}\bigg(n^{\epsilon_{1}}\bigg\|-\frac{1}{n}\p_{\alpha}^{2}\tilde{\mbbh}_{n}(\alpha_{0},\beta)-\tilde{\Gam}_{1,0}\bigg\|\bigg)^{M}\bigg\} \\
&\lesssim1+\mbbe\bigg[\bigg\{n^{\epsilon_{1}}\bigg\|\frac{1}{2n}\sumj\p_{\alpha}^{2}\Big(\log\big|S_{j-1}(\alpha_{0})\big|\Big)+\frac{1}{2}\bigg(\frac{1}{n}\sumj\p_{\alpha}^{2}S_{j-1}^{-1}(\alpha_{0})\big[S_{j-1}(\alpha_{0})\big]\bigg) \\
&\qquad\quad-\frac{1}{2d}\bigg(\frac{1}{n}\sumj\p_{\alpha}S_{j-1}^{-1}(\alpha_{0})\big[S_{j-1}(\alpha_{0})\big]\bigg)^{\otimes2}-\tilde{\Gam}_{1,0}\bigg\|\bigg\}^{M}\bigg] \\
&\quad+\mbbe\bigg[\bigg\{n^{\epsilon_{1}}\bigg\|\frac{1}{2}\bigg(\frac{1}{n\tau h_{0}}\sumj\p_{\alpha}^{2}S_{j-1}^{-1}(\alpha_{0})\big[(\D_{j}X)^{\otimes2}\big]\bigg)-\frac{1}{2d}\bigg(\frac{1}{n\tau h_{0}}\sumj\p_{\alpha}S_{j-1}^{-1}(\alpha_{0})\big[(\D_{j}X)^{\otimes2}\big]\bigg)^{\otimes2} \\
&\qquad\quad-\frac{d}{2}\p_{\alpha}^{2}\big(\log h(\alpha)\big)\bigg\|\bigg\}^{M}\bigg] \\
&\lesssim1+\mbbe\bigg[\bigg\{n^{\epsilon_{1}}\bigg\|\frac{1}{2}\bigg(\frac{1}{n\tau h_{0}}\sumj\p_{\alpha}^{2}S_{j-1}^{-1}(\alpha_{0})\big[(\D_{j}X)^{\otimes2}\big]\bigg)-\frac{1}{2d}\bigg(\frac{1}{n\tau h_{0}}\sumj\p_{\alpha}S_{j-1}^{-1}(\alpha_{0})\big[(\D_{j}X)^{\otimes2}\big]\bigg)^{\otimes2} \\
&\qquad\quad-\frac{d}{2}\bigg(\frac{1}{nd}\sumj\p_{\alpha}^{2}S_{j-1}^{-1}(\alpha_{0})\big[(\D_{j}X)^{\otimes2}\big]\bigg)\frac{1}{h(\alpha)} \nn\\
&\qquad \quad
+\frac{d}{2}\bigg(\frac{1}{nd}\sumj\p_{\alpha}S_{j-1}^{-1}(\alpha_{0})\big[(\D_{j}X)^{\otimes2}\big]\bigg)^{\otimes2}\bigg(\frac{1}{h(\alpha)}\bigg)^{2}\bigg\|\bigg\}^{M}\bigg] \\
&\leq 1+\mbbe\bigg[\bigg\{n^{\epsilon_{1}}\bigg\|\frac{h_{0}}{h(\alpha)}\bigg(\frac{1}{nh_{0}}\sumj\p_{\alpha}^{2}S_{j-1}^{-1}(\alpha_{0})\big[(\D_{j}X)^{\otimes2}\big]\bigg)\bigg(\frac{h(\alpha)}{\tau h_{0}}-1\bigg)\bigg\|\bigg\}^{M}\bigg] \\
&\quad+\mbbe\bigg[\bigg\{n^{\epsilon_{1}}\bigg\|\bigg(\frac{h_{0}}{h(\alpha)}\bigg)^{2}\bigg(\frac{1}{nh_{0}}\sumj\p_{\alpha}^{2}S_{j-1}^{-1}(\alpha_{0})\big[(\D_{j}X)^{\otimes2}\big]\bigg)^{\otimes2}\bigg\{\bigg(\frac{h(\alpha)}{\tau h_{0}}-1\bigg)^{2}+2\bigg(\frac{h(\alpha)}{\tau h_{0}}-1\bigg)\bigg\}\bigg\|\bigg\}^{M}\bigg] \\
&<\infty, \\
&\mbbe\bigg\{\bigg(\frac{1}{n}\sup_{\theta}\big\|\p_{\alpha}^{3}\tilde{\mbbh}_{n}(\theta)\big\|\bigg)^{M}\bigg\} \\
&\lesssim1+\mbbe\bigg\{\big(\sup_{\alpha}\big\|\p_{\alpha}^{3}\big(\log h(\alpha)\big)\big\|\big)^{M}\bigg\} \\
&\lesssim1+\mbbe\bigg[\bigg\{\sup_{\alpha}\bigg\|\frac{1}{nh_{0}d}\sumj\p_{\alpha}^{3}S_{j-1}^{-1}(\alpha)\big[(\D_{j}X)^{\otimes2}\big]\bigg\|\bigg|\frac{h_{0}}{h(\alpha)}\bigg|\bigg\}^{M}\bigg] \\
&\quad+\mbbe\bigg[\bigg\{\sup_{\alpha}\bigg\|\frac{1}{nh_{0}d}\sumj\p_{\alpha}^{2}S_{j-1}^{-1}(\alpha)\big[(\D_{j}X)^{\otimes2}\big]\bigg\|\bigg\|\frac{1}{nh_{0}d}\sumj\p_{\alpha}S_{j-1}^{-1}(\alpha)\big[(\D_{j}X)^{\otimes2}\big]\bigg\|\bigg|\frac{h_{0}}{h(\alpha)}\bigg|^{2}\bigg\}^{M}\bigg] \\
&\quad+\mbbe\bigg[\bigg\{\sup_{\alpha}\bigg\|\frac{1}{nh_{0}d}\sumj\p_{\alpha}S_{j-1}^{-1}(\alpha)\big[(\D_{j}X)^{\otimes2}\big]\bigg\|^{3}\bigg|\frac{h_{0}}{h(\alpha)}\bigg|^{3}\bigg\}^{M}\bigg] \\
&<\infty.
\end{align*}
The proofs of the conditions \eqref{pldi.al1}, \eqref{pldi.al3}, and \eqref{pldi.al4} are complete. 
The tuning-parameter condition [A4$'$] can be verified exactly in the same way as in \cite[Section 6]{Yos11}.
We thus obtain \eqref{pldi1}.

\subsubsection{Proof of \eqref{pldi2}}

We will prove the following conditions: for every $M>0$,
\begin{align}
&\sup_{n}\mbbe\bigg(\bigg\|\frac{1}{\sqrt{n\tau h_{0}}}\p_{\beta}\tilde{\mbbh}_{n}(\tz)\bigg\|^{M}\bigg)<\infty; \label{pldi.be1} \\
&\sup_{n}\mbbe\bigg\{\bigg(\sup_{\beta}(n\tau h_{0})^{\epsilon_{1}}\big|\tilde{\mbby}_{n}^{2}(\beta;\alpha_{0})-\tilde{\mbby}_{0}^{2}(\beta)\big|\bigg)^{M}\bigg\}<\infty; \label{pldi.be2} \\
&\sup_{n}\mbbe\bigg\{\bigg(\frac{1}{n\tau h_{0}}\sup_{\beta}\big\|\p_{\beta}^{3}\tilde{\mbbh}_{n}(\alpha_{0},\beta)\big\|\bigg)^{M}\bigg\}<\infty; \label{pldi.be3} \\
&\sup_{n}\mbbe\bigg\{\bigg((n\tau h_{0})^{\epsilon_{1}}\bigg\|-\frac{1}{n\tau h_{0}}\p_{\beta}^{2}\tilde{\mbbh}_{n}(\tz)-\tilde{\Gam}_{2,0}\bigg\|\bigg)^{M}\bigg\}<\infty. \label{pldi.be4}
\end{align}
We have
\begin{align*}
&\left\|\frac{1}{\sqrt{n\tau h_{0}}}\p_{\beta}\tilde{\mbbh}_{n}(\tz)\right\| \\
&\lesssim\left\|\frac{1}{\sqrt{n\tau h_{0}}}\sumj S_{j-1}^{-1}(\alpha_{0})\big[\D_{j}X-\tau h_{0} b_{j-1}(\tz), \p_{\beta}b_{j-1}(\tz)\big]\right\| \\
&\quad+\sqrt{\frac{nh_{0}}{\tau}}\left|\frac{1}{nh_{0}d}\sumj S_{j-1}^{-1}(\alpha_{0})\big[(\D_{j}X)^{\otimes2}-\mbbe_{j-1}\{(\D_{j}X)^{\otimes2}\}\big]\right|\left\|\frac{1}{n}\sumj S_{j-1}^{-1}(\alpha_{0})\big[b_{j-1}(\tz), \p_{\beta}b_{j-1}(\tz)\big]\right\| \\
&\quad+\sqrt{\frac{nh_{0}}{\tau}}\left|\frac{1}{nh_{0}d}\sumj S_{j-1}^{-1}(\alpha_{0})\big[\mbbe_{j-1}\{(\D_{j}X)^{\otimes2}\}-\tau h_{0} S_{j-1}(\alpha_{0})\big]\right|\left\|\frac{1}{n}\sumj S_{j-1}^{-1}(\alpha_{0})\big[b_{j-1}(\tz), \p_{\beta}b_{j-1}(\tz)\big]\right\|.
\end{align*}
We can deduce \eqref{pldi.be1} from using Lemma 8(b) of \cite{Yos11}, Burkholder's inequality, and H\"{o}lder's inequality.
From Lemmas 8(a) and 9 of \cite{Yos11}, we can show the following inequalities in a similar way as the proof of \eqref{pldi.be1}: 
\begin{align*}
&\mbbe\bigg\{\bigg(\sup_{\beta}(n\tau h_{0})^{\epsilon_{1}}\big|\tilde{\mbby}_{n}^{2}(\beta;\alpha_{0})-\tilde{\mbby}_{0}^{2}(\beta)\big|\bigg)^{M}\bigg\} \\
&\lesssim1+\tau^{\epsilon_{1}}\mbbe\bigg\{\bigg(\sup_{\beta}(nh_{0})^{\epsilon_{1}}\bigg|\frac{1}{nh_{0}}\sumj\Big(S_{j-1}^{-1}(\alpha_{0})\big[\D_{j}X,b_{j-1}(\alpha_{0},\beta)-b_{j-1}(\tz)\big] \\
&\qquad\quad-\frac{h_{0}}{2}S_{j-1}^{-1}(\alpha_{0})\big[b_{j-1}(\alpha_{0},\beta)^{\otimes2}-b_{j-1}(\tz)^{\otimes2}\big]\Big)-\tilde{\mbby}_{0}^{2}(\beta)\bigg|\bigg)^{M}\bigg\}<\infty, \\
&\mbbe\bigg\{\bigg(\frac{1}{n\tau h_{0}}\sup_{\beta}\big\|\p_{\beta}^{3}\tilde{\mbbh}_{n}(\alpha_{0},\beta)\big\|\bigg)^{M}\bigg\}<\infty, \\
&\mbbe\bigg\{\bigg((n\tau h_{0})^{\epsilon_{1}}\bigg\|-\frac{1}{n\tau h_{0}}\p_{\beta}^{2}\tilde{\mbbh}_{n}(\tz)-\tilde{\Gam}_{2,0}\bigg\|\bigg)^{M}\bigg\} \\
&\lesssim1+\tau^{\epsilon_{1}}\mbbe\bigg\{\bigg((nh_{0})^{\epsilon_{1}}\bigg\|\frac{1}{n}\sumj S_{j-1}^{-1}(\alpha_{0})\big[\p_{\beta}b_{j-1}(\tz), \p_{\beta}b_{j-1}(\tz)\big]-\tilde{\Gam}_{2,0}\bigg\|\bigg)^{M}\bigg\} \\
&\quad+\tau^{\epsilon_{1}}\mbbe\bigg\{\bigg((nh_{0})^{\epsilon_{1}}\bigg\|\bigg(\frac{h(\alpha_{0})}{\tau h_{0}}-1\bigg)\frac{1}{n}\sumj S_{j-1}^{-1}(\alpha_{0})\big[b_{j-1}(\tz), \p_{\beta}^{2}b_{j-1}(\tz)\big]\bigg\|\bigg)^{M}\bigg\}<\infty.
\end{align*}
Hence, we have established \eqref{pldi.be2} to \eqref{pldi.be4} as well.
Finally, the tuning-parameter condition [A4$'$] can be verified as before, completing the proof of \eqref{pldi2}.

\subsection{Proof of Theorem \ref{se:bic}} \label{proof:thm.bic}

Let
\begin{align*}
\tilde{\mbbz}_{n}(u)&=\exp\left\{ \tilde{\mbbh}_{n}(\tz+D_{n}^{-1}u)-\tilde{\mbbh}_{n}(\tz) \right\}, \\
\tilde{\mbbz}_{n}^{0}(u)&=\exp\left(\tilde{\D}_{n}[u]-\frac{1}{2}\diag(\tilde{\Gam}_{1,0},\tau\tilde{\Gam}_{2,0})[u,u]\right)
\end{align*}
and $\mathbb{U}_{n}(\tz)=\{u\in\mbbr^{p};\tz+D_{n}^{-1}u\in\Theta\}$.
In what follows, we deal with the zero-extended version of $\tilde{\mbbz}_{n}$
and use the same notation:
$\tilde{\mbbz}_{n}$ vanishes outside $\mathbb{U}_{n}(\tz)$, so that
\begin{align*}
\int_{\mbbr^{p}\setminus\mathbb{U}_{n}(\tz)}\tilde{\mbbz}_{n}(u)du=0.
\end{align*}

\medskip

\paragraph{{\it Proof of \eqref{se:thm.bic1}}}
By the change of variable $\theta=\tz+D_{n}^{-1}u$, the modified marginal quasi-log likelihood function (recall the definition \eqref{hm:mmqlf_def}) satisfies the equation
\begin{align*}
\mfL_{n}&=\tilde{\mbbh}_{n}(\tz)-\frac{1}{2}p_{\alpha}\log n-\frac{1}{2}p_{\beta}\log(nh_{0})+\log\left(\int_{\mathbb{U}_{n}(\tz)}\tilde{\mbbz}_{n}(u)\mathfrak{p}(\tz+D_{n}^{-1}u)du\right).
\end{align*}
If $\int_{\mathbb{U}_{n}(\tz)}\big|\tilde{\mbbz}_{n}(u)\mathfrak{p}(\tz+D_{n}^{-1}u)-\tilde{\mbbz}_{n}^{0}(u)\mathfrak{p}(\tz)\big|du$ converges to 0 in probability, we obtain
\begin{align*}
\mfL_{n}
&=\tilde{\mbbh}_{n}(\tz)-\frac{1}{2}p_{\alpha}\log n-\frac{1}{2}p_{\beta}\log(nh_{0})+\log\left(\int_{\mathbb{U}_{n}(\tz)}\tilde{\mbbz}_{n}^{0}(u)\mathfrak{p}(\tz)du\right)+o_{p}(1) \\
&=\tilde{\mbbh}_{n}(\tz)-\frac{1}{2}p_{\alpha}\log n-\frac{1}{2}p_{\beta}\log(nh_{0})+\log \mathfrak{p}(\tz)+\frac{p}{2}\log(2\pi) \\
&\quad-\frac{1}{2}\log\big|\tilde{\Gam}_{1,0}\big|-\frac{1}{2}\log\big|\tau\tilde{\Gam}_{2,0}\big|+\frac{1}{2}\diag(\tilde{\Gam}_{1,0},\tau\tilde{\Gam}_{2,0})^{-1}\big[\tilde{\D}_{n}^{\otimes2}\big]+o_{p}(1).
\end{align*}
We will show that 
\begin{align*}
\int_{\mathbb{U}_{n}(\tz)}\left|\tilde{\mbbz}_{n}(u)\mathfrak{p}(\tz+D_{n}^{-1}u)-\tilde{\mbbz}_{n}^{0}(u)\mathfrak{p}(\tz)\right|du\cip0.
\end{align*}
To this end, it suffices to verify the conditions (\ref{appendix.thm1}) to (\ref{appendix.thm6}) in Theorem \ref{hm:thm.qbic}.
We have seen in the proof of Theorem \ref{hm:thm2.diffusion.an} that \eqref{appendix.thm1} holds.
The conditions (\ref{appendix.thm2}) to (\ref{appendix.thm4}) can be deduced in a similar manner as in Section \ref{proof.AN.an}.
Since \eqref{se:aux.2}, \eqref{hm:aux.1}, and \eqref{se:aux.3} ensure
\begin{align}
\tilde{\mbby}_{n}^{1}(\alpha,\beta)
&= O_{p}^{\ast}(h_{0})
+\frac{1}{2}\bigg[ \int_{\mbbr^{d}}\log\big|S^{-1}(x,\alpha)S(x,\alpha_{0})\big|\pi(dx)\nn\\
&{}\qquad -d\bigg\{
\log\bigg( \tau\int_{\mbbr^{d}}\tr\big(S^{-1}(x,\al)S(x,\alpha_{0})\big)\pi(dx))\bigg)
-\log\left(\tau d)\right)
\bigg\}\bigg],\nn\\
\tilde{\mbby}_{n}^{2}(\beta;\alpha_{0}) 
&=O_{p}^{\ast}(\sqrt{h_{0}})-\frac{\tau}{2}\int_{\mbbr^{d}}S^{-1}(x,\alpha_{0})\big[\big(b(x,\alpha_{0},\beta)-b(x,\tz)\big)^{\otimes2}\big]\pi(dx), \nn
\end{align}
we obtain
\begin{align*}
(\sqrt{n})^{q}\sup_{\theta}|\tilde{\mbby}_{n}^{1}(\alpha,\beta)-\tilde{\mbby}_{0}^{1}(\alpha)|&=O_{p}\big((n^{q}h_{0}^{2})^{1/2}\big), \\
(\sqrt{nh_{0}})^{q}\sup_{\beta}|\tilde{\mbby}_{n}^{2}(\beta;\alpha_{0})-\tilde{\mbby}_{0}^{2}(\beta)|&=O_{p}\big((n^{q}h_{0}^{1+q})^{1/2}\big).
\end{align*}
Therefore (\ref{appendix.thm5}) holds for $q\in(0,1)$.
As for \eqref{appendix.thm6}, recall that in Section \ref{hm:proof.consistency} we have seen that the function $\tilde{\mbby}_{0}^{1}$ admits a unique maximum $0$ only at $\al_{0}$.
Since $\Theta_{\al}$ has a bounded closure and the function $\tilde{\mbby}_{0}^{1}$ is twice continuously differentiable with respect to $\al$ under the integration sign, with the associated Hessian matrix $-\p_{\al}^{2}\tilde{\mbby}_{0}^{1}(\al_{0})=\tilde{\Gam}_{1,0}$ being positive definite at $\al_{0}$, we see that there exists a constant $\chi_{1}>0$ such that
\begin{equation}
\tilde{\mbby}_{0}^{1}(\alpha) \le -\chi_{1}\|\al-\al_{0}\|^{2}
\label{hm:Y1_iden}
\end{equation}
for all $\alpha\in\Theta_{\alpha}$.
Exactly in the same way, there exists a constant $\chi_{2}>0$ such that
\begin{equation}
\tilde{\mbby}_{0}^{2}(\beta) \le -\chi_{2}\|\beta-\beta_{0}\|^{2}
\label{se:Y2_iden}
\end{equation}
for all $\beta\in\Theta_{\beta}$. Then \eqref{appendix.thm6} follows from \eqref{hm:Y1_iden} and \eqref{se:Y2_iden}, completing the proof of (\ref{se:thm.bic1}).

\medskip

\paragraph{{\it Proof of \eqref{se:thm.bic2}}}
By the Taylor expansion, we obtain
\begin{align*}
\tilde{\D}_{n}&=\p_{\theta}\tilde{\mbbh}_{n}(\tilde{\theta}_{n})+\big(\diag(\tilde{\Gam}_{1,0},\tau\tilde{\Gam}_{2,0})+o_{p}(1)\big)\big[D_{n}(\tilde{\theta}_{n}-\tz)\big] \\
&=\big(\diag(\tilde{\Gam}_{1,0},\tau\tilde{\Gam}_{2,0})+o_{p}(1)\big)\big[D_{n}(\tilde{\theta}_{n}-\tz)\big].
\end{align*}
Therefore, $D_{n}(\tilde{\theta}_{n}-\tz)=\diag(\tilde{\Gam}_{1,0},\tau\tilde{\Gam}_{2,0})^{-1}\tilde{\D}_{n}+o_{p}(1)$ and
\begin{align*}
\tilde{\mbbh}_{n}(\tz)&=\tilde{\mbbh}_{n}(\tilde{\theta}_{n})-\frac{1}{2}\diag(\tilde{\Gam}_{1,0},\tau\tilde{\Gam}_{2,0})\big[\big(D_{n}(\tilde{\theta}_{n}-\tz)\big)^{\otimes2}\big]+o_{p}(1) \\
&=\tilde{\mbbh}_{n}(\tilde{\theta}_{n})-\frac{1}{2}\diag(\tilde{\Gam}_{1,0},\tau\tilde{\Gam}_{2,0})^{-1}\big[\tilde{\D}_{n}^{\otimes2}\big]+o_{p}(1).
\end{align*}
Then \eqref{se:thm.bic2} follows from the equations $\tilde{h}=\tau h_{0}+o_{p}(1)$, $\mathfrak{p}(\tilde{\theta}_{n})=\mathfrak{p}(\tz)+o_{p}(1)$, $-n^{-1}\p_{\alpha}^{2}\tilde{\mbbh}_{n}(\tilde{\theta}_{n})=\tilde{\Gam}_{1,0}+o_{p}(1)$, and $-(n\tilde{h})^{-1}$ $\p_{\beta}^{2}\tilde{\mbbh}_{n}(\tilde{\theta}_{n})=\tilde{\Gam}_{2,0}+o_{p}(1)$:
\begin{align*}
\mfL_{n}&=\tilde{\mbbh}_{n}(\tilde{\theta}_{n})-\frac{1}{2}p_{\alpha}\log n-\frac{1}{2}p_{\beta}\log n\tilde{h}+\log \mathfrak{p}(\tz)+\frac{p}{2}\log(2\pi) \\
&\quad-\frac{1}{2}\log\left|-\frac{1}{n}\p_{\alpha}^{2}\tilde{\mbbh}_{n}(\tilde{\theta}_{n})\right|-\frac{1}{2}\log\left|-\frac{1}{n\tilde{h}}\p_{\beta}^{2}\tilde{\mbbh}_{n}(\tilde{\theta}_{n})\right|+o_{p}(1).
\end{align*}

\subsection{Proof of Theorem \ref{se:model.consis}} \label{proof:thm.model.consis}
We only prove Theorem \ref{se:model.consis}(1) because Theorem \ref{se:model.consis}(2) can be handled analogously to (1) and \cite[Theorem 5.5]{EguMas18a}. 

If both $m_{1}\neq m_{1,0}$ and $m_{2}\neq m_{2,0}$ hold, we have
\begin{align}
\mathbb{P}\left(\mbic^{(m_{1,0},m_{2,0})}-\mbic^{(m_{1},m_{2})}\geq0\right)&\leq\mathbb{P}\left(\mbic^{(m_{1,0},m_{2,0})}-\mbic^{(m_{1},m_{2,0})}\geq0\right) \nn\\
&\qquad+\mathbb{P}\left(\mbic^{(m_{1},m_{2,0})}-\mbic^{(m_{1},m_{2})}\geq0\right). \label{se:mod.consis1}
\end{align} 
By the Taylor expansion, we obtain
\begin{align*}
\mbbh_{n}^{(m_{1,0},m_{2,0})}(\tilde{\theta}_{m_{1,0},m_{2,0},n})&=\mbbh_{n}^{(m_{1,0},m_{2,0})}(\theta_{m_{1,0},m_{2,0},0})+O_{p}(1), \\
\mbbh_{n}^{(m_{1},m_{2,0})}(\tilde{\theta}_{m_{1},m_{2,0},n})&=\mbbh_{n}^{(m_{1},m_{2,0})}(\theta_{m_{1},m_{2,0},0})+O_{p}(1). 
\end{align*} 
Since each candidate model includes the true model, $\mbbh_{n}^{(m_{1,0},m_{2,0})}(\theta_{m_{1,0},m_{2,0},0})=\mbbh_{n}^{(m_{1},m_{2,0})}(\theta_{m_{1},m_{2,0},0})$.
Moreover, the definition of $m_{1,0}$ implies that $p_{\alpha_{m_{1,0}}}<p_{\alpha_{m_{1}}}$.
Thus, we have
\begin{align}
&\mathbb{P}\left(\mbic^{(m_{1,0},m_{2,0})}-\mbic^{(m_{1},m_{2,0})}\geq0\right) \nn\\
&=\mathbb{P}\bigg\{-2\mbbh_{n}^{(m_{1,0},m_{2,0})}(\tilde{\theta}_{m_{1,0},m_{2,0},n})+2\mbbh_{n}^{(m_{1},m_{2,0})}(\tilde{\theta}_{m_{1},m_{2,0},n})+(p_{\alpha_{m_{1,0}}}-p_{\alpha_{m_{1}}})\log n \nn\\
&\qquad+p_{\beta_{m_{2,0}}}\log\frac{h(\tilde{\alpha}_{m_{1,0},n})}{h(\tilde{\alpha}_{m_{1},n})}\geq0\bigg\} \nn\\
&=\mathbb{P}\bigg\{O_{p}(1)+o_{p}(1)\geq(p_{\alpha_{m_{1}}}-p_{\alpha_{m_{1,0}}})\log n\bigg\} \nn\\
&\to0 \nn
\end{align}
as $n\to\infty$.
In a similar way as above, we can show that the second term of the right-hand side of \eqref{se:mod.consis1} tends to zero, hence the claim is proved.
In the case of $m_{1}\neq m_{1,0}$ and $m_{2}=m_{2,0}$ or in the case of $m_{1}=m_{1,0}$ and $m_{2}\neq m_{2,0}$, the proof is similar and simpler.

\bigskip

\appendix

\section{Stochastic expansion of the quasi-marginal log likelihood}
\label{hm:sec_appen1}

We here step away from the main context and present a set of conditions under which a quasi-marginal log likelihood admits a Schwarz type stochastic expansion, by making use of \cite[Proof of Theorem 2.1]{JKM}.

\medskip

Let $\mbbh_{n}: \Theta\times \Omega\to\mbbr$ be a $\mcc^{3}(\Theta)$-random function where $\Theta\subset\mbbr^{p}$ is a bounded convex domain.
Set $\theta=(\al,\beta)\in\mbbr^{p_{\al}}\times\mbbr^{p_{\beta}}$,
and let $\tz=(\al_{0},\beta_{0})\in\Theta$ be a constant, and $D_{n}=D_{n}(\tz)=\diag\big(\sqrt{r_{1,n}}I_{p_{\al}},\, \sqrt{r_{2,n}}I_{p_{\beta}}\big)$, where $(r_{1,n})$ and $(r_{2,n})$ are positive sequences possibly depending on $\tz$ and satisfying that $r_{1,n}\wedge r_{2,n}\to\infty$ and that $r_{2,n}/r_{1,n}\to 0$ as $n\to\infty$. We then introduce the random field on $\mbbr^{p}$ associated with $\mbbh_{n}$:
\begin{equation}
\mbbz_{n}(u):=\exp\left\{ \mbbh_{n}(\tz+D_{n}^{-1}u) - \mbbh_{n}(\theta_{0}) \right\}.
\nn
\end{equation}
Here we set $\mbbz_{n}\equiv 0$ outside the set $\mbbu_{n}=\mbbu_{n}(\tz):=D_{n}(\Theta-\tz)\subset\mbbr^{p}$.
Let $\mfp(\theta)$ be a bounded prior probability density on $\Theta$, which is assumed to be continuous and positive at $\tz$.
Let $\Delta_{n}(\tz):=D_{n}^{-1}\p_{\theta}\mbbh_{n}(\tz)$ and
\begin{equation}
\Gam_{0}:=\diag(\Gam_{1,0},\,\Gam_{2,0}),
\nonumber
\end{equation}
where $\Gam_{1,0}\in\mbbr^{p_{\al}}\otimes\mbbr^{p_{\al}}$ and $\Gam_{2,0}\in\mbbr^{p_{\beta}}\otimes\mbbr^{p_{\beta}}$ are a.s. positive definite random matrices.
Further, let
\begin{align}
\mbby_{1,n}(\theta) &:=\frac{1}{r_{1,n}}\{\mbbh_{n}(\al,\beta) - \mbbh_{n}(\al_{0},\beta)\},
\nn \\
\mbby_{2,n}(\beta) &:= \frac{1}{r_{2,n}} \{\mbbh_{n}(\al_{0},\beta) - \mbbh_{n}(\al_{0},\beta_{0})\},
\nn
\end{align}
and $\mbby_{1}(\al)$ and $\mbby_{2}(\beta)$ be $\mbbr$-valued random functions.
Finally, we introduce the quadratic random field
\begin{equation}
\mbbz^{0}_{n}(u) =\exp\bigg( \Delta_{n}(\tz)[u] - \frac{1}{2}\Gam_{0}[u,u] \bigg).
\nonumber
\end{equation}

\begin{thm}
In addition to the aforementioned setting, suppose the following conditions.
\begin{itemize}
\item There exists an a.s. positive definite random matrix $\Sig_{0}\in\mbbr^{p}\otimes\mbbr^{p}$ such that
\begin{equation}
\left(\Delta_{n}(\tz),\, -D_{n}^{-1}\p_{\theta}^{2}\mbbh_{n}(\tz)D_{n}^{-1}\right) \overset{\mcl}\to \Big( \Sig_{0}^{1/2}\eta,\, \Gam_{0}\Big), \label{appendix.thm1}
\end{equation}
where $\eta \sim N_{p}(0,I_{p})$ is a random variable defined on an extension of the original probability space.

\item We have
\begin{align}
& \sup_{\beta}\bigg\| \frac{1}{\sqrt{r_{1,n}}} \p_{\al}\mbbh_{n}(\al_{0},\beta) \bigg\| = O_{p}(1), \label{appendix.thm2}\\
& \sup_{\beta} \bigg\| -\frac{1}{r_{1,n}}\p_{\al}^{2}\mbbh_{n}(\al_{0},\beta) - \Gam_{1,0}\bigg\|=o_{p}(1), \label{appendix.thm3}\\
& \sup_{\theta}\left\| D_{n}^{-1}\p_{\theta}^{3}\mbbh_{n}(\theta)D_{n}^{-1} \right\| = O_{p}(1). \label{appendix.thm4}
\end{align}

\item There exists a constant $q\in(0,1)$ for which
\begin{equation}
r_{1,n}^{q/2}\sup_{\theta}\left|\mbby_{1,n}(\theta)-\mbby_{1}(\al)\right| \vee r_{2,n}^{q/2}\sup_{\beta}\left|\mbby_{2,n}(\beta)-\mbby_{2}(\beta)\right| \cip 0.
\label{appendix.thm5}
\end{equation}

\item There exists an a.s. positive random variable $\chi_{0}$ such that for each $\kappa>0$,
\begin{equation}
\sup_{\al;\, |\al-\al_{0}|\ge\kappa}\mbby_{1}(\al) \vee 
\sup_{\beta;\, |\beta-\beta_{0}|\ge\kappa}\mbby_{2}(\beta) \le -\chi_{0}\kappa^{2}\qquad \text{a.s.}
\label{appendix.thm6}
\end{equation}

\end{itemize}
Then, any $\tes \in \argmax\mbbh_{n}$ satisfies that $D_{n}(\tes-\tz)\cil \Gam_{0}^{-1}\Sigma_{0}^{1/2}\eta$, and we have
\begin{align}
\int \bigg|\mbbz_{n}(u)\pi(\tz+D_{n}^{-1}u) - \mbbz_{n}^{0}(u)\pi(\tz) \bigg|du \cip 0
\nn
\end{align}
and
\begin{align}
\mfL_{n}:=\log\bigg(\int_{\Theta}\exp\{\mbbh_{n}(\theta)\}\mfp(\theta)d\theta\bigg)
&=\mbbh_{n}(\tz) + \log|D_{n}^{-1}| + \log\mfp(\tz) + \frac{p}{2}\log(2\pi) \nn\\
&{}\qquad -\frac{1}{2}\log|\Gam_{0}| + \frac{1}{2}\Gam_{0}^{-1}\left[\D_{n}(\tz)^{\otimes 2}\right] + o_{p}(1).
\nonumber
\end{align}
Further, if $\log|D_{n}^{-1}| = \log|D_{n}^{-1}(\tes)| + o_{p}(1)$ and $\log\mfp(\tz)=\log\mfp(\tes)+o_{p}(1)$, then
\begin{align}
\mfL_{n}
&=\mbbh_{n}(\tes) + \log|D_{n}^{-1}(\tes)| + \log\mfp(\tes) + \frac{p}{2}\log(2\pi) \nn\\
&{}\qquad -\frac{1}{2}\log|-D_{n}^{-1}\p_{\theta}^{2}\mbbh_{n}(\tes)D_{n}^{-1}| + o_{p}(1).
\nonumber
\end{align}
\label{hm:thm.qbic}
\end{thm}

Theorem \ref{hm:thm.qbic} can apply to general locally asymptotically quadratic models under weaker conditions compared with \cite[Theorem 3.7]{EguMas18a}. A formal extension of Theorem \ref{hm:thm.qbic} to cases of more than two rates is straightforward.

\bigskip
\noindent
\textbf{Acknowledgements.} 
The authors thank the two anonymous referees for careful reading and valuable comments which helped to greatly improve the paper.
They also grateful to Prof. Isao Shoji for sending us his unpublished version of manuscript \cite{Sho18}, which deals with a calibration problem of the sampling frequency from a completely different point of view from ours,
and to Yuma Uehara for a helpful comment on Theorem \ref{se:thm.pldi}.
This work was partially supported by JST CREST Grant Number JPMJCR14D7, Japan.

\bigskip 
\bibliographystyle{abbrv} 

\end{document}